\documentclass{article}
\usepackage{tikz-cd}
\usepackage{nicematrix}
\usepackage{colortbl}
\usepackage{color}

\usepackage{fullpage}
\usepackage{authblk}
\usepackage{blkarray}
\usepackage{amsmath}
\usepackage{amsthm}
\usepackage{amssymb}
\usepackage{bm}
\usepackage{setspace}
\usepackage[numbers,sort&compress]{natbib}

\usepackage{amsfonts}
\usepackage{eucal}
\usepackage{latexsym}
\usepackage{mathdots}
\usepackage{mathrsfs}
\usepackage{enumitem}

\newcommand{\be}{\begin{equation}}
\newcommand{\ee}{\end{equation}}
\newcommand{\ba}{\begin{eqnarray}}
\newcommand{\ea}{\end{eqnarray}}
\newcommand{\nn}{\nonumber}

\newcommand{\bal}{\begin{align}}
\newcommand{\eal}{\end{align}}
\newcommand{\baln}{\begin{align*}}
\newcommand{\ealn}{\end{align*}}

\newcommand{\bi}{\begin{itemize}}
\newcommand{\ei}{\end{itemize}}
\newcommand{\bn}{\begin{enumerate}}
\newcommand{\en}{\end{enumerate}}
\newcommand{\bbm}{\begin{bmatrix}}
\newcommand{\ebm}{\end{bmatrix}}
\newcommand{\bpm}{\begin{pNiceMatrix}}
\newcommand{\epm}{\end{pNiceMatrix}}
\newcommand{\bsm}{\left ( \begin{smallmatrix}}
\newcommand{\esm}{\end{smallmatrix} \right) }

\newcommand{\mr}{\ensuremath{\mathrm}}
\newcommand{\scr}{\ensuremath{\mathscr}}
\newcommand{\mbf}{\ensuremath{\boldsymbol}}
\newcommand{\mc}{\ensuremath{\mathcal}}
\newcommand{\mf}{\ensuremath{\mathfrak}}
\newcommand{\ov}{\ensuremath{\overline}}
\newcommand{\sm}{\ensuremath{\setminus}}
\newcommand{\wt}{\ensuremath{\widetilde}}

\def\tr{\mathrm{tr} \, }

\newcommand{\Ga}{\ensuremath{\Gamma}}
\newcommand{\ga}{\ensuremath{\gamma}}
\newcommand{\Om}{\ensuremath{\Omega}}

\newcommand{\la}{\ensuremath{\lambda }}
\newcommand{\om}{\ensuremath{\omega}}

\newcommand{\eps}{\ensuremath{\epsilon }}

\def\C{\mathbb{C}}
\def\R{\mathbb{R}}
\def\D{\mathbb{D}}

\def\N{\mathbb{N}}
\def\B{\mathbb{B}}
\def\F{\mathbb{F}}
\def\bH{\mathbb{H}}
\newcommand{\cH}{\ensuremath{\mathcal{H}}}
\newcommand{\cJ}{\ensuremath{\mathcal{J} }}
\newcommand{\cK}{\ensuremath{\mathcal{K} }}

\newcommand{\ff}{\mathfrak{f}}

\def\wvec{\overrightarrow}
\def\hardy{\mathbb{H} ^2 _d}
\def\mult{\mathbb{H} ^\infty _d}
\def\rmult{\mathbb{H} ^{\infty; \mrt} _d}
\def\mrt{\mathrm{t}}
\def\frt{\mathfrak{r} ^\mathrm{t}}
\def\fbt{\mathfrak{b} ^\mathrm{t}}

\def\fr{\mathfrak{r}}
\def\fb{\mathfrak{b}}
\def\fa{\mathfrak{a}}
\def\fc{\mathfrak{c}}

\def\vfb{\vec{\mathfrak{b}}}
\def\posncm{ \left( \scr{A} _d   \right) ^\dag _+ }
\def\nbre{\mathrm{Re} \, }
\def\nbim{\mathrm{Im} \, }
\def\nbdom{\mathrm{Dom} \, }
\def\nbran{\mathrm{Ran} \, }
\def\nbker{\mathrm{Ker} \, }
\def\nbdim{\mathrm{dim} \, }

\def\A{\mathbb{A} _d}
\def\fp{\mathbb{C} \{ \mathbb{\mathfrak{z}} \} }

\newcommand{\Addresses}{{
  \bigskip
  \footnotesize

  M.~T.~Jury, \textsc{Department of Mathematics, University of Florida}\par\nopagebreak
  \textit{E-mail address:} \texttt{mjury@ad.ufl.edu}

  \medskip

  R.~T.~W.~Martin, \textsc{Department of Mathematics, University of Manitoba}\par\nopagebreak
  \textit{E-mail address:} \texttt{Robert.Martin@umanitoba.ca}
    
\medskip

E.~Shamovich, \textsc{Department of Mathematics, Ben-Gurion University of the Negev}\par\nopagebreak
\textit{E-mail address:} \texttt{shamovic@bgu.ac.il}

\vspace{1cm}

}}

\newcommand{\ip}[2]{\ensuremath{\langle {#1} , {#2} \rangle}}
\newcommand{\ipcn}[2]{\ensuremath{\left( {#1} , {#2} \right) _{\C ^n}}}

\newcommand{\ran}[1]{\ensuremath{\mathrm{Ran} \left( {#1} \right) }}

\renewcommand{\ker}[1]{\ensuremath{\mathrm{Ker} \left( {#1} \right) }}

\numberwithin{equation}{section}
\numberwithin{subsection}{section}

\newtheorem{thm}{Theorem}

\newtheorem{lemma}{Lemma}
\newtheorem{prop}{Proposition}
\newtheorem{cor}{Corollary}
\newtheorem*{thm*}{Theorem}
\newtheorem{thmA}{Theorem}

\theoremstyle{definition}
\newtheorem{defn}{Definition}
\newtheorem{remark}{Remark}
\newtheorem{eg}{Example}

\title{Non-commutative rational Clark measures}

\author{Michael T. Jury\thanks{Supported by NSF grant DMS-1900364}}
\affil[1]{\footnotesize University of Florida}

\author[2]{Robert T.W. Martin\thanks{Supported by NSERC grant 2020-05683}}
\affil[2]{\footnotesize University of Manitoba}

\author[3]{Eli Shamovich}
\affil[3]{ \footnotesize Ben-Gurion University of the Negev}

\date{}

\begin{document}

\NiceMatrixOptions{cell-space-top-limit=1pt}
\NiceMatrixOptions{cell-space-bottom-limit=1pt}

\bibliographystyle{abbrv}
\maketitle

\begin{abstract}
We characterize the non-commutative Aleksandrov--Clark measures and the minimal realization formulas of contractive and, in particular, isometric non-commutative rational multipliers of the Fock space.  Here, the full Fock space over $\C ^d$ is defined as the Hilbert space of square--summable power series in several non-commuting formal variables, and we interpret this space as the non-commutative and multi-variable analogue of the Hardy space of square--summable Taylor series in the complex unit disk.  We further obtain analogues of several classical results in Aleksandrov--Clark measure theory for non-commutative and contractive rational multipliers.


Non-commutative measures are defined as positive linear functionals on a certain self-adjoint subspace of the Cuntz--Toeplitz algebra, the unital $C^*-$algebra generated by the left creation operators on the full Fock space. Our results demonstrate that there is a fundamental relationship between NC Hardy space theory, representation theory of the Cuntz--Toeplitz and Cuntz algebras and the emerging field of non-commutative rational functions. 
\end{abstract}

\section{Introduction}

The full Fock space over $\C ^d$, $\hardy$, can be defined as the Hilbert space of square--summable power series in several non-commuting (NC) formal variables, $\mf{z} := ( \mf{z} _1 , \cdots , \mf{z} _d)$. As such, the Fock space is an obvious NC and multi-variable generalization of the Hardy space $H^2$ of square-summable Taylor series in the complex unit disk, $\D$. Namely, any $h \in \hardy$ is a power series of the form
$$ h (\mf{z} ) := \sum _{\om \in \F ^d} \hat{h} _\om \mf{z} ^\om; \quad \quad \hat{h} _\om \in \C, $$ where $\F ^d$ is the \emph{free monoid}, the set of all words in the $d$ letters $\{1 , 2 , ... , d \}$, and if $\om = i_1 \cdots i_n \in \F ^d$, $1 \leq i_k \leq d$, the free monomials are defined in the obvious way as $\mf{z} ^\om := \mf{z} _{i_1} \cdots \mf{z} _{i_d}$. (This is a monoid with product given by concatenation of words, and the unit is the \emph{empty word}, $\emptyset$, containing no letters.) The Fock space is a Hilbert space when equipped with the $\ell ^2-$inner product of its power series coefficients. Remarkably, elements of $\hardy$ are bona fide functions on the \emph{NC unit row-ball} of all strict row contractions acting on a separable Hilbert space. That is, a $d-$tuple of $n\times n$ complex matrices, $Z := (Z_1 , \cdots , Z_d )$ can be viewed as a linear map $Z : \C ^n \otimes \C ^d \rightarrow \C ^n$ from $d$ copies of $\C ^n$ into one copy. If this linear map is a (strict) contraction, then $Z$ is said to be a (strict) \emph{row contraction}. The formal power series of any $h \in \hardy$ converges absolutely in operator norm when evaluated at any such $Z$, and also uniformly on compacta in a suitable sense. Moreover, such NC functions are \emph{free non-commutative functions} in the NC unit row-ball, $\B ^d _\N$, in the sense of NC Function Theory: They are graded, preserve direct sums and the joint similarities which respect their NC domain, $\B ^d _\N$. Here we write $\B ^d _\N = \bigsqcup _{n=1} ^\infty \B ^d _n$, where $\B ^d _n$ is the set of all strict row contractions on $\C ^n$. 

The Hardy algebra, $H ^\infty$, of all uniformly bounded analytic functions in the complex unit disk is the \emph{multiplier algebra} of $H^2$. That is, if $g \in H^\infty$ and $h \in H^2$, the linear map $h \mapsto g\cdot h$ defines a bounded linear multiplication operator and $H^\infty \subset H^2$. The Hardy algebra contains many rational functions; any rational function in $\C$ with poles in $\C \sm \ov{\D}$ automatically belongs to $H^\infty$ and is in fact analytic in a disk of radius greater than one. This is similarly the case for the Fock space, or \emph{NC Hardy space}, $\hardy$ \cite[Theorem A]{JMS-ncrat}. We define the \emph{NC Hardy algebra}, $\mult$, as the unital algebra of all uniformly bounded free NC functions in the unit row-ball, and this can be identified with the (left) multiplier algebra of the NC Hardy space. Rational functions can also be defined in several NC variables, and this yields a rich sub-discipline of NC function theory which has deep and novel connections to several branches of algebra and analysis including NC Algebra, Free Probability Theory, Multi-variable Operator Theory, Control Theory and Free Algebraic Geometry  \cite{Taylor,Taylor2,Voic,KVV,KVV-ncrat,KVV-ncratdiff,Volcic,Pro-Hilb,Volcic-Hilb,HKV-ncpolyfact,HMV-ncconvex,HMc-ncratconvex,HKMS,HMS-realize,Williams,Ball-control,Ball-robust,FreeAG,HelMc-annals}.

A complex NC rational expression is any valid linear combination of NC polynomials, inverses and products. The domain, $\nbdom \mr{r}$, of such an expression is the collection of all $d$-tuples of matrices of all sizes, $X = (X _1 , \cdots , X_d ) \in \C ^{n\times n} \otimes \C ^{1 \times d}$, $n \in \N$, for which $\mr{r} (X) \in \C ^{n \times n}$ is defined. An NC rational function, $\fr$, is an equivalence class of NC rational expressions with respect to the relation $\mr{r} _1 \equiv \mr{r} _2$, if $\mr{r} _1$ and $\mr{r} _2$ agree on the intersection of their domains and the domain of the equivalence class, $\fr$, is the union of the domains of every $\mr{r} \in \fr $. We say that $\fr$ is \emph{regular at $0$} if $0 = (0 , \cdots , 0) \in \C ^d$ belongs to $\nbdom \fr$. Any NC rational function in $d-$variables, $\mf{r}$, which is regular at $0$ has a finite--dimensional \emph{realization}: There is a triple $(A, b, c)$ with $A \in \C ^d _n := \C ^{n \times n} \otimes \C ^{1\times d}$ and $b,c \in \C ^n$, so that for any $X \in \C ^d _m$,
$$ \mf{r} (X) = b^* L_A (X) ^{-1} c; \quad \quad L_A (X) := I_n \otimes I_m - \sum A_j \otimes X_j.$$ Here, $L_A (\cdot ) $ is called a (monic, affine) \emph{linear pencil}. Realizations of NC rational functions have been studied extensively and have numerous applications. 

In this paper we seek NC multi-variable analogues of classical results for contractive rational multipliers of the Hardy space. Any contractive multiplier, $b \in [H^\infty ] _1$, of $H^2$ corresponds, essentially uniquely, to a positive, finite and regular Borel measure, $\mu _b$, on the complex unit circle, $\partial \D$. Here, we use the notations $[X] _1$ and $(X) _1$ to denote the closed and open unit balls, respectively, of a Banach space, $X$. This measure is called the \emph{Aleksandrov--Clark} or \emph{Clark measure} of $b$ \cite{Clark,Aleks1,Aleks2}. Fatou's theorem implies that any contractive analytic function in the disk, $b$, is inner if and only if its Clark measure is singular with respect to Lebesgue measure \cite{Fatou,Hoff}. If $b = \fb$ is a contractive rational multiplier then $\mu _\fb$ is either a singular, finite positive sum of weighted point masses on the circle, in which case $\fb$ is an inner, finite Blaschke product, or, $\mu _\fb$ has non-zero absolutely continuous part with respect to Lebesgue measure with log--integrable Radon--Nikodym derivative and this implies that $\fb$ is not an extreme point of $[H^\infty ] _1$. One can define a one-parameter family of Clark measures, $\mu _\xi$, associated to $b \in [H^\infty ] _1$ as the Clark measures of $\ov{\xi} b$ for any $\xi$ on the unit circle, $\partial \D$. Defining $H^2 (\mu )$ as the closure of the analytic polynomials in $L^2 (\mu )$, for any $\xi \in \partial \D$, there is a natural, unitary `weighted Cauchy transform' from $H^2 (\mu _{\ov{\xi} b} )$ onto the \emph{de Branges--Rovnyak space}, $\scr{H} (b)$, of $b$. This space is a Hilbert space of analytic functions which is contractively contained in $H^2$, and in the case where $b$ is inner, this is simply $\scr{H} (b) = (bH ^2 ) ^\perp$, the orthogonal complement of the range of $b$ as an isometric multiplier. In this inner case, $H^2 (\mu _{\ov{\xi} b}) = L^2 (\mu _{\ov{\xi} b} )$, so that multiplication by the independent variable, $M_\zeta ^{(\xi )} : H^2 (\mu _{\ov{\xi} b} ) \rightarrow H^2 (\mu _{\ov{\xi} b} )$, is unitary. As discovered by D.N. Clark, the images of the adjoints of this one-parameter family of unitary operators, $M_\zeta ^{(\xi)}$, under weighted Cauchy transform is a family of rank--one unitary perturbations of the restricted \emph{backward shift}, $S^* | _{(bH^2)^\perp}$ \cite{Clark}. Here, recall that the \emph{shift}, $S:= M_z$, is the isometry of multiplication by the independent variable in $H^2$. Analysis of the shift plays a central role in Hardy space theory, and in operator theory in general \cite{Nik-shift,NF}. A fundamental result, due to Aronszajn and Donoghue, in the theory of Aleksandrov--Clark measures, is that the singular parts of the family $\mu _\alpha = \mu _{b \ov{\alpha}}$, $\alpha \in \partial \D$, are mutually singular \cite{Aronszajn,Dono}. Moreover, point masses of $\mu _{\ov{\alpha} b}$ on the unit circle correspond to points where $b$ has a finite Carath\'eodory angular derivative \cite{Caratheodory,Nevanlinna}. We will obtain natural and convincing analogues of these results in the non-commutative setting for NC rational multipliers of the Fock space.  

\subsection{Reader's Guide}

The subsequent section will provide some basic background on the Fock space and NC rational functions. In Section \ref{sect:ncratClark} we study the NC Clark measures of contractive NC rational multpliers. Classically, positive measures on the circle can be identified with positive linear functionals via the Riesz--Markov Theorem, and the appropriate NC analogue of a positive measure on the circle is then a positive linear functional on a certain operator system, the \emph{free disk system}. Theorem \ref{fincorthm} identifies the NC Clark measures of contractive NC rational multipliers as the finitely--correlated positive linear functionals on the free disk system. Finitely--correlated Cuntz states were originally introduced by Bratteli and J\o rgensen in their studies of representations of the celebrated Cuntz algebra \cite{BraJorg}, the universal $C^*-$algebra of a surjective row isometry \cite{Cuntz}. Here, a row isometry is an isometry from several copies of a Hilbert space into itself.  We further prove that an NC rational multiplier is inner, \emph{i.e.} isometric, if and only if its NC Clark measure is singular with respect to NC Lebesgue measure in the sense of the NC Lebesgue decomposition of \cite{JM-ncld,JM-ncFatou}, see Corollary \ref{singinner}. This is the analogue of a classical corollary to Fatou's theorem in the special case of rational multipliers: A contractive multiplier of $H^2$ is inner if and only if its Clark measure is singular. Theorem \ref{fincorstructure} provides a detailed characterization of the finitely--correlated positive NC measures, including a concrete formula for the NC Radon--Nikodym derivative of any finitely--correlated NC measure with respect to a canonical NC Lebesgue measure. Theorem \ref{minratreal} provides a complete description of the minimal realization of any contractive NC rational multiplier and, in particular, any inner (isometric) NC rational multiplier.

Proposition \ref{evalues} shows that if $\fb \in [ \mult ]_1$ is NC rational and inner, then there are certain row co-isometries, $A_\zeta$, on the boundary of the unit row--ball, $\partial \B ^d _\N$, so that $\fb (A _\zeta ^\mrt )$ has $\zeta \in \partial \D$ as an eigenvalue, where $\mrt$ denotes matrix transpose of each component. If $v$ is the eigenvector of $\fb (A _\zeta ^\mrt )^*$ to eigenvalue $\ov{\zeta}$, Theorem \ref{boundary} shows that $v^*\fb (Z )$ has a Carath\'eodory angular derivative at $A_\zeta ^\mrt$, and that the point evaluation $h \mapsto y^* h(A_\zeta ^\mrt) v$ is a bounded linear functional on the de Branges--Rovnyak space of $\fb$, for any vector $y$ of the appropriate size. Theorem \ref{ncAD} then partially extends the Aronszajn--Donoghue theorem to NC rational multipliers of Fock space: Under certain assumptions we show that the singular parts of the family of NC Clark measures of an inner rational multiplier of Fock space are mutually singular and we provide a finite upper bound on the number of distinct NC Clark measures which are not mutually singular. Here, a Gelfand--Naimark--Segal (GNS) construction applied to any positive NC measure produces a GNS Hilbert space, and a row isometry acting on this space. (This space and this row isometry are the multi-variable analogues of $H^2 (\mu )$ and $M_\zeta | _{H^2 _\mu }$ in the case where $\mu$ is a positive measure on the circle.) Two NC Clark measures $\mu _{\fb } , \mu _{\fb '}$ are then said to be mutually singular if their GNS row isometries are mutually singular in the sense that they have no unitarily equivalent direct summands. 

Any row isometry uniquely determines and is uniquely determined by a $*-$representation of the \emph{Cuntz--Toeplitz algebra}, $C^* \{ I , L_1 , \cdots , L_d \}$ \cite{Cuntz}. Here, $L := (L_1 , \cdots , L_d )$ is the \emph{left free shift}, the row isometry of left multiplications by the $d$ indpendent NC variables, $L_k := M^L _{\mf{z} _k}$, on the Fock space, and this plays the role of the shift operator, $S=M_z : H^2 \rightarrow H^2$, in this multi-variable NC Hardy space theory. This reveals a fundamental connection between the representation theory of the Cuntz and Cuntz--Toeplitz $C^*-$algebras and the study of positive NC measures. In fact any cyclic row isometry (or representation) can be obtained, up to unitary equivalence, as the GNS row isometry of a positive NC measure \cite[Lemma 2.2]{JMT-ncFMRiesz}. The Cuntz and Cuntz--Toeplitz $C^*-$algebras are important objects in $C^*-$algebra theory and they also play a central role in the dilation theory of row contractions \cite{Pop-dil}.

\section{Background}

\subsection{Multipliers of Fock space}

Left multipications by the $d$ independent NC variables $\mf{z} = ( \mf{z} _1 , \cdots , \mf{z} _d )$ define isometries on the Fock space with pairwise orthogonal ranges,
$$ L_k := M^L _{\mf{z} _k}, \quad \quad L_j ^* L_k = \delta _{j,k} I. $$ It follows that the row $d-$tuple $L := (L_1 , \cdots , L_d ) : \hardy \otimes \C ^d \rightarrow \hardy$ is an isometry from several copies of $\hardy$ into itself. Such an isometry is called a \emph{row isometry}, we call this row isometry of left multiplications on the Fock space the \emph{left free shift} and its components left free shifts. Similarly, one can define the right free shifts, $R_k := M^R _{\mf{z} _k}$, as right multiplication by the independent NC variables as well as the row isometric right free shift, $R = (R _1 , \cdots , R _d)$. The letter reversal map $\mrt : \F ^d \rightarrow \F ^d$, which reverses the order of letters in any word $\om  \in \F ^d$ defines an involution on the free monoid,
$$  \om = i_1 \cdots i_n \mapsto \om ^\mrt := i_n \cdots i_1. $$ Given a word, $\om = i_1 \cdots i_n \in \F ^d$, the \emph{length} of $\om$ is $|\om | =n$. The free monomials $ \{ e_\om := \mf{z} ^\om | \ \om \in \F ^d \}$ define a standard orthonormal basis of $\hardy \simeq \ell ^2 (\F ^d )$, and the letter reversal map gives rise to a unitary involution of the Fock space, $U_\mrt$, defined by $U_\mrt \mf{z} ^\om = \mf{z} ^{\om ^\mrt}$. Here $e_\emptyset = \mf{z} ^\emptyset =: 1$ is called the vacuum vector of the Fock space. It is straightforward to verify that $U_\mrt L_k U_\mrt = R_k$, so that the left shifts are isomorphic to the right shifts.

The NC Hardy algebra, $\mult$, of uniformly bounded NC functions can be identified, completely isometrically, with the unital Banach algebra of left multipliers of the NC Hardy space, $\hardy$. That is, given any NC function, $F \in \mult$ and $h \in \hardy$, the left multiplication operator $M^L _F : \hardy \rightarrow \hardy$, defined by 
$$ h(Z) \mapsto F(Z) \cdot h(Z), $$ is bounded, and $\| M ^L _F \|  = \| F \| _\infty$, where $\| \cdot \| _\infty$ denotes the supremum norm over $\B ^d _\N$. For any free polynomial $p \in \C \{ \mf{z} \}$, one can check that $p(L) = M^L _p$, and so we employ the notation $F(L) := M^L _F$. Similarly, if $p \in \fp$, then, $M^R _p = p ^\mrt (R) = U_\mrt p(L) U_\mrt$, acts as right multiplication by $p$, where if $h$ is a formal power series, $h(\mf{z} ) = \sum \hat{h} _\om \mf{z} ^\om$,
$$ h^\mrt (\mf{z} ) := \sum \hat{h} _\om \mf{z} ^{\om ^\mrt} = \sum \hat{h} _{\om ^\mrt} \mf{z} ^\om. $$ In particular, if $h \in \hardy$, $h^\mrt = U_\mrt h$. The left and right multiplier algebras of $\hardy$ are unitarily equivalent via the unitary letter reversal involution $U_\mrt$ and can be identified with the left and right analytic Toeplitz algebras, $L^\infty _d := \mr{Alg} \{ I , L_1 , \cdots , L_d \} ^{-WOT}$ and $R^\infty _d =\{ I , R_1 , \cdots , R_d \} ^{-WOT} = U_\mrt L^\infty _d U_\mrt$. Since $p (R) = M^R _{p ^\mrt}$ for any $p \in \fp$, we will write $G (R) = M^R _{G ^\mrt}$ for any $G \in \mult$. Namely $G \in \mult$ if and only if $G^\mrt \in \rmult := \mrt \circ \mult$.  We will use the following terminology: A left or right multiplier is \emph{inner} if it is isometric and \emph{outer} if it has dense range.

\subsection{Non-commutative reproducing kernel Hilbert spaces}

In syzygy with classical Hardy space theory, the Fock space is a (non-commutative) reproducing kernel Hilbert space, in the sense that for any $Z \in \B ^d _n$ and vectors $y,v \in \C ^n$, the \emph{matrix-entry point evaluation}, $\ell _{Z,y,v} : \hardy \rightarrow \C$,
$$ h \mapsto y^* h(Z) v, $$ is a bounded linear functional. Equivalently the linear map $h \mapsto h(Z)$ is bounded as a map from $\hardy$ into the Hilbert space, $\C ^{n \times n}$, equipped with the Hilbert--Schmidt inner product. By the Riesz lemma, $\ell _{Z,y,v}$ is implemented by inner products against vectors $K \{ Z , y , v \} \in \hardy$ which we call \emph{NC Szeg\"o kernel vectors}. 

In greater generality, let $\C ^d _\N := \bigsqcup _{n=1} ^\infty  \C ^d _n$ denote the $d-$dimensional \emph{complex NC universe}. Here, recall that we define $\C ^d _n := \C ^{n\times n} \otimes \C ^{1\times d}.$ A subset $\Om \subseteq \C ^d _\N$, is an \emph{NC set} if it is closed under direct sums, and we write $\Om = \bigsqcup \Om _n$ where $\Om _n := \Om \bigcap \C ^d _n$. 

A  Hilbert space, $\cH$, of free non-commutative functions on $\Om$ is a \emph{non-commutative reproducing kernel Hilbert space} (NC-RKHS), if for any $n\in \N$, $Z \in \Om _n$ and $y,v \in \C ^n$, the linear point evaluation functional 
$$ h \mapsto y^* h(Z) v, $$ is bounded on $\cH$ \cite{BMV}. As before, the Riesz lemma implies that these functionals are implemented by taking inner products against \emph{point evaluation} or \emph{NC kernel vectors} $k \{ Z, y ,v \} \in \cH$. Given any such NC-RKHS, $Z \in \Om _n$ and $W \in \Om _m$, one can define a completely bounded linear map on $n \times m$ complex matrices by: $k(Z,W) [ \cdot ] : \C ^{n\times m} \rightarrow \C ^{n\times m}$,
$$  y^* k(Z,W) [vu^*] x := \ip{k\{ Z,y,v \}}{k \{ W,x ,u \} }_{\cH}, $$ and this map is completely positive if $Z=W$ \cite{BMV}. Following \cite{BMV}, we call $k(Z,W)[\cdot ]$ the \emph{completely positive non-commutative (CPNC) reproducing kernel} of $\cH$, and we write $\cH = \cH _{nc} (k)$. One can check that adjoints of left and right multipliers of an NC-RKHS have a familiar action on NC kernel vectors:
$$ (M^L _F ) ^* k\{ Z , y ,v \} = k\{ Z , F(Z) ^* y , v \} \quad \mbox{and} \quad (M^R _G ) ^* k \{ Z , y ,v \} = k \{ Z , y , G (Z) v \}. $$ 

All NC-RKHS in this paper will be Hilbert spaces of free NC functions in the unit row-ball $\B ^d _\N$, 
$$ \B ^d _\N = \bigsqcup _{n=1} ^\infty \B ^d _n, \quad \quad \B ^d _n := \left\{ \left. Z \in \C ^{n\times n} \otimes \C  ^{1\times d} \right| \, ZZ^* = Z_1 Z_1 ^* + \cdots + Z_d Z_d ^* < I_n \right\}. $$ In the case of the Fock space, $\hardy = \cH _{nc} (K)$, where $K$ is the \emph{NC Szeg\"o kernel}: Given $Z \in \B ^d _n, W \in \B ^d _m$ and $P \in \C ^{n \times m}$,
$$ K(Z,W) := \left( \mr{id} _{n,m} [ \cdot ] - \mr{Ad} _{Z, W^*} [\cdot ] \right)  ^{-1} \circ P = \sum _{\om \in \F ^d} Z^\om P W^{*\om}, $$ 
 $\mr{Ad} _{Z,W^*} [P] := Z_1 P W_1 ^* + \cdots + Z_d P W_d ^*$.

\subsection{NC rational functions}

As described in the introduction, a complex NC rational expression is any combination of the several NC variables $\mf{z} _1 , \cdots , \mf{z} _d$, the complex scalars, $\C$, the operations $+ , \cdot , ^{-1}$ and parentheses $\left( , \right)$ with domain
$$\nbdom \mr{r} = \bigsqcup _{n=1} ^\infty \mr{Dom} _n \, \mr{r}, \quad  \mr{Dom} _n \,  \mr{r} := \bigsqcup _{n=1} ^\infty \left\{ \left. X = (X _1 , \cdots , X_d ) \in \C ^{n\times n} \otimes \C ^{1 \times d}  \right| \ \mr{r} (X) \ \mbox{is defined.} \right\}. $$ We will use the notation $\C ^d _n := \C ^{n\times n} \otimes \C ^{1\times d}$ for a row $d-$tuple of complex $n\times n$ matrices. An NC rational expression is \emph{valid}, if its domain is non-empty. An NC rational function, $\fr$, is then an equivalence class of valid NC rational expressions with respect to the relation $\mr{r} _1 \equiv \mr{r} _2$ if  $\mr{r} _1 (X) = \mr{r} _2 (X)$ for all $X \in \nbdom \mr{r} _1 \bigcap \nbdom \mr{r} _2$. (By \cite[Footnote, page 52]{KVV-ncratdiff}, given any two valid NC rational expressions, $\mr{r} _k$, their domains at level $n$, $\mr{Dom} _n \, \mr{r} _k := \nbdom \mr{r} _k \cap \C ^d _n$ have non-trivial intersection for sufficiently large $n$.) The set of all NC rational functions in $d$ variables with coefficients in $\C$ is a division ring or skew field, and is, in fact, the universal skew field of fractions for the ring $\C \{ \mf{z} \}  = \C \{ \mf{z} _1 , \cdots , \mf{z} _d \}$ of complex NC polynomials \cite{Amitsur,Cohn},\cite[Proposition 2.2]{KVV-ncratdiff}. 

Any NC rational function in $d-$variables, $\mf{r}$, which is regular at $0$ has a unique (up to joint similarity) \emph{minimal descriptor realization}. Namely, there is a triple $(A, b, c)$ with $A \in \C ^d _n$ and $b,c \in \C ^n$, so that for any $X \in \C ^d _m$,
$$ \mf{r} (X) = \left( b^* \otimes I_m \right) L_A (X) ^{-1} \left( c \otimes I_m \right); \quad \quad L_A (X) := I_n \otimes I_m - \sum A_j \otimes X_j, $$ and this realization is minimal in the sense that $n$ is as small as possible. Here, $L_A (\cdot ) $ is called a (monic, affine) \emph{linear pencil} and we will employ the simplified notations 
$$ ZA = Z \otimes A := \sum _{j=1} ^d Z_j \otimes A_j, $$ for any $Z, A \in \C ^d _\N$. Minimality implies that that the realization is both \emph{observable},
$$ \bigvee A^{*\om} b = \C ^{n}$$  and \emph{controllable} 
$$ \bigvee A^\om c = \C ^n, $$ see \emph{e.g.} \cite[Subsection 3.1.2]{HMS-realize}. Minimal realizations are unique up to joint similarity \cite[Theorem 2.4]{Berstel-ncrat}. The domains of NC rational functions which are regular at $0$ have a convenient description:
\begin{thm*}{ (\cite[Theorem 3.10]{Volcic}, \cite[Theorem 3.1]{KVV-ncrat})}
If $\mf{r}$ is an NC rational function which is regular at $0$ with minimal realization $(A,b,c)$, then,
$$ \nbdom \mf{r} = \bigsqcup _{n \in \N} \left\{ \left. X \in \C ^d _n \right| \ \mr{det} \, L_A (X) \neq 0 \right\}.$$
\end{thm*}

Any NC rational $\fr \in \hardy$ is necessarily defined on $\B ^d _\N$ and hence is regular at $0$. As proven in \cite{JMS-ncrat}, an NC rational function belongs to the Fock space if and only if it is regular at $0$ and has minimal realization $(A,b,c)$ so that the joint spectral radius of the $d-$tuple $A$ is less than one \cite[Theorem A]{JMS-ncrat}. Here, if $A := (A_1 , \cdots , A_d ) : \C ^n \otimes \C ^d \rightarrow \C ^n$ is any row $d-$tuple of $n\times n$ matrices, we define the completely positive map $\mr{Ad} _{A, A^*} : \C ^{n\times n } \rightarrow \C ^{n \times n}$ by 
$$ \mr{Ad} _{A,A^*} (P) := A_1  P A_1 ^* + \cdots + A_d  P A_d ^*. $$ The joint spectral radius, $\mr{spr} (A)$, of $A$ is then defined by the Beurling formula:
$$ \mr{spr} (A) := \lim _k \sqrt[2k]{\| \mr{Ad} _{A, A^*}  ^{(k)} (I_n) \|}. $$ By the multivariable Rota--Strang theorem, $A$ is jointly similar to a strict row contraction if and only if $\mr{spr} (A) <1$ \cite[Theorem 3.8]{PopRota} (see also \cite[Proposition 2.3, Remark 2.6]{SSS}). In particular, $A$ is said to be \emph{pure} if $\mr{Ad} ^{(k)} _{A , A^*} (I_n) \rightarrow 0$, and a finite--dimensional row $d-$tuple, $A \in \C ^d _n$, is pure if and only if $\mr{spr} (A) <1$ \cite[Lemma 2.5]{SSS}. If $\fr \in \hardy$ has minimal descriptor realization $(A,b,c)$, then $A$ is jointly similar to a finite strict row contraction $\ov{Z} \in \B ^d _n$, where $\ov{Z}$ denotes entry--wise complex conjugation, and one can verify that $\fr = K \{ Z , y , v \}$ is an NC Szeg\"o kernel vector in the Fock space, where $y,v$ are the image of $\ov{b}, \ov{c}$ under the (conjugate of the) similarity and its inverse that intertwine $A$ and $\ov{Z}$ \cite[Proposition 3.2]{JMS-ncrat}. That is, 
$$ \ip{\fr}{h}_{\bH ^2} = y^* h(Z) v $$ and 
$$ K \{Z , y ,v \} (W) = \sum _{\om \in \F ^d} \ov{y^* Z^\om v} \, W^\om; \quad \quad W \in \B ^d _\N. $$ 
In \cite[Theorem A]{JMS-ncrat}, we established (a more general version of) the following theorem which characterizes when an NC rational function belongs to $\hardy$.     
\begin{thmA} \label{ncratinH2}
Let $\fr$ be an NC rational function in $d$ variables with minimal realization $(A,b,c)$, $A \in \C ^d _n$. The following are equivalent: 
\bn
    \item[(i)] $\fr \in \hardy$.
    \item[(ii)] $\fr \in \mult$.
    \item[(iii)] $r \cdot \B ^d _\N \subseteq \nbdom \fr$ for some $r >1$. 
    \item[(iv)] $\mr{spr} (A) < 1$.
    \item[(v)] $\fr = K \{ Z , y ,v \}$ for some $Z \in \B ^d _n$ and $y,v \in \C ^n$. 
\en
\end{thmA}

Given $A \in \C ^d _m$ and $Z \in \C ^d _n$, consider the linear pencil,
$$ L_A (Z) = I_n \otimes I_m - Z \otimes A = I_n \otimes I_m - \sum _{j=1} ^d Z_j \otimes A_j. $$
Observe that 
$$ \| Z \otimes A \|  \leq  \| Z \| _{row} \| A \| _{col}  :=   \left\| (Z_1 , \cdots , Z _d ) \right\| _{\scr{L} (\C ^n \otimes \C ^d , \C ^n )} \left\| \bsm A_1 \\ \vdots \\ A_d \esm \right\|_{\scr{L} (\C ^m , \C ^m \otimes \C ^d )}. $$ Given $A \in \C ^d _m$, we will also write 
$$ \mr{col} (A) := \bsm A_1 \\ \vdots \\ A _d \esm  \in \C ^{m\times m} \otimes \C ^d, \quad \mbox{so that} \quad \| A \| _{col} = \| \mr{col} (A) \| _{\scr{L} (\C ^m , \C ^m \otimes \C ^d )}. $$ 
It follows that $Z\otimes A$ will be similar to a contraction if $Z$ is jointly similar to a row contraction and $A$ is jointly similar to a column contraction, and $Z \otimes A$ will further be similar to a strict contraction if, in addition, at least one of $Z, A$ is jointly similar to a strict row or column contraction, respectively. If $Z \otimes A$ is similar to a strict contraction, then $L_A (Z) ^{-1}$ can be expanded as a convergent geometric sum. A row $d-$tuple, $Z \in \C ^d _n$, is said to be \emph{irreducible}, if it has no non-trivial jointly invariant subspace, \emph{i.e.} there is no non-trivial subspace which is invariant for every $Z_k$, $1 \leq k \leq d$. The following lemma will be useful in the sequel:
\begin{lemma} \label{rowvscol}
If $Z = (Z _1 , \cdots , Z_d ) \in \C ^d _n$ then the column, $\mr{col} (Z)$, is jointly similar to a column $d-$tuple, $\mr{col} (W)$, with column--norm at most $\| W \| _{col} \leq \| Z \| _{row} +\eps$, for any $\eps >0$. Conversely, for any $\eps >0$, $Z$ is jointly similar to a row $d-$tuple $W' \in \C ^d _n$ with $\| W ' \| _{row} \leq \| Z \| _{col} + \eps$. If $Z$ is irreducible then one can take $\eps =0$. In particular, any strict row contraction is jointly similar to a strict column contraction and vice versa.   
\end{lemma}
\begin{proof}
The joint spectral radius of $Z$ obeys $\mr{spr} (Z) \leq \| Z \| _{row}$. Consider the row $d-$tuple $Z^* = \mr{row} (Z^*) := (Z_1 ^* , \cdots , Z_d ^* )$. Then,
\ba \mr{spr} (Z ^* ) & = & \lim _{k \uparrow \infty} \| \mr{Ad} _{Z ^* , Z} ^{(k)} (I_n ) \| ^{\frac{1}{2k}} \nn \\
& \leq & \lim  \sqrt[\leftroot{-2}\uproot{10}2k]{ \tr  \mr{Ad} _{Z ^*, Z } ^{(k)} (I_n ) } \nn \\
& = &  \lim \sqrt[\leftroot{-2}\uproot{10}2k]{ \sum _{|\om | = k} \tr Z ^{*\om} Z^{\om}} \nn \\
& = & \lim  \sqrt[\leftroot{-2}\uproot{10}2k]{ \sum _{|\om | = k} \tr Z ^{\om }  Z^{ * \om  }} \nn \\
& = & \lim  \sqrt[\leftroot{-2}\uproot{10}2k]{ \tr \mr{Ad} ^{(k)} _{Z , Z^* } (I_n ) } \nn \\
& \leq & \lim \sqrt[\leftroot{-2}\uproot{10}2k]{ n \cdot \| Z \| _{row} ^{2k}  }  \nn \\
& = & \| Z \| _{row}. \nn \ea 
By \cite[Lemma 2.4]{SSS}, the closure of the joint similarity orbit of the row $d-$tuple, $\mr{row} ( Z^* )$, contains a $d-$tuple, $W'$, with norm at most $\| W ' \| _{row} = \mr{spr} (Z^* ) \leq  \| Z \| _{row}$, and if $Z$ is irreducible then its joint similarity orbit is closed so that $W'$ is in the joint similarity orbit of $Z^*$. In particular, given any $\eps >0$, $Z ^*$ is jointly similar to some $W \in \C ^d _n$ with $\| W \| _{row} \leq \| Z \| _{row} + \eps$. Viewing $W : \C ^n \otimes \C ^d \rightarrow \C ^n$ as a linear map, its Hilbert space adjoint is $\mr{col} (W^*) : \C ^n \rightarrow \C ^n \otimes \C ^d$ with norm
$$ \| W^* \|_{col} ^2  = \| W  \| _{row} ^2 \leq (\| Z \| _{row} + \eps  ) ^2.$$ 
Since $\mr{row} (Z^*)$ is jointly similar to $\mr{row} (W)$, it follows that $\mr{col} (Z)$ is jointly similar to $\mr{col} (W ^*)$ where 
$\| W ^* \| _{col} \leq \| Z \| _{row} + \eps$. Proof of the other half of the claim is analogous. 
\end{proof}

\begin{lemma} \label{transkernel}
If $\fr \in \hardy$, then $\frt := U_\mrt \fr$ is also an NC rational function in $\hardy$. In particular, the transpose of any NC Szeg\"o kernel is an NC Szeg\"o kernel.
\end{lemma}
\begin{proof}
If $\fr \in \hardy$, then by \cite[Theorem A]{JMS-ncrat}, $\fr = K \{ Z , y , v \}$ for some $Z \in \B ^d _n$ and $y,v \in \C ^n$. Let $\mf{C}$ denote the conjugation (anti-linear isometric involution) with respect to the standard basis $\{ e_k \} _{k=1} ^n$ defined by $\mf{C} y = \ov{y}$, the entry--wise (with respect to the standard basis) complex conjugation of $y$. Given $A \in \C ^d _n$, let $A^\mrt := (A_1 ^\mrt , \cdots , A_d ^\mrt )$ where $A_j ^\mrt$ denotes matrix--transpose and let $\ov{A} = \mf{C} A \mf{C}$ denote complex conjugation applied entry--wise to $A$ in the standard basis. As described in \cite{JMS-ncrat}, since $\fr = K\{ Z , y ,v \}$ for some $Z \in \B ^d _\N$, if we define $A := \ov{Z}$, $b = \ov{y}$ and $c = \ov{v}$, then $(A,b,c)$ is a finite-dimensional realization of $\fr$. 
Recall that
\ba  \ip{L^\om 1}{K \{ Z , y, v \}}_{\bH ^2} & = &  \ov{y^* Z^\om v} \nn \\
& = & b^* A^\om c, \nn \ea 
and calculate,
\ba \ip{L^\om 1}{U_\mrt K \{ Z , y , v \}}_{\bH ^2}
& = & \ip{L^{\om ^\mrt} 1}{ K \{ Z , y , v \}}_{\bH ^2} \nn \\
& = & \ipcn{b}{A^{\om ^\mrt}c} \nn \\
& = & \ipcn{b}{ ((A^\mrt) ^\om) ^\mrt  c} \nn \\
& = & \ipcn{\ov{A} ^{\mrt \om} b }{c} \nn \\
& = & \ipcn{\mf{C} c}{ \mf{C} \ov{A} ^{\mrt \om} b}  \nn \\
& = & \ipcn{\ov{c}}{A^{\mrt \om} \ov{b}}. \nn \ea 
By the previous lemma, the $d-$tuple $A^\mrt \in \C ^d _n$ is pure, $\mr{spr} (A^\mrt ) <1$ and $A^\mrt$ is jointly similar to a strict column contraction. It follows that $\frt$ has a finite--dimensional realization $(A^\mrt , \ov{c} , \ov{b} )$, where $A^\mrt = \left( A_1 ^\mrt , \cdots , A_d ^\mrt \right)$ and hence $\frt \in \hardy$ is also NC rational.
\end{proof}

\begin{remark}
The domain of any NC rational function, $\fr$, which is regular at $0$, is open with respect to the uniform topology on $\C ^d _\N$. Namely, given any $X \in \C ^d _n$ and $Y \in \C ^d _m$, we define the \emph{row pseudo-metric},
\ba d_{row} (X,Y ) ^2 & := & \| X \otimes I_m - Y \otimes I_n \| ^2 _{row} = \| (X \otimes I_m - Y\otimes I_n ) ( X \otimes I_m - Y\otimes I_n) ^* \|  \nn \\
& = &  \left\| \sum _{k=1} ^d (X_k \otimes I_m - Y_k \otimes I_n ) ( X_k ^* \otimes I_m - Y_k ^* \otimes I_n )  \right\|. \nn \ea Since multiplication, summation and inversion are all jointly continuous in operator norm, and any NC rational function can be constructed by applying finitely many arithmetic operations applied to free polynomials, it follows that $\nbdom \fr$ is a uniformly open NC set, and hence contains some row--ball, $r \B ^d _\N$, of non-zero radius $r>0$. Hence, by re-scaling the argument, $\fr _r (Z) := \fr (rZ)$, we obtain an NC rational function $\fr _r$ with $\B ^d _\N \subseteq \nbdom \fr _r$ so that $\fr _r \in \mult$. It follows that given any NC rational function $\fr$, which is regular at $0$, there is essentially no loss in generality in assuming that $\fr \in \mult$, or even $\fr \in [ \mult ] _1$. Alternatively, if $\fr$ has minimal realization $(A,b,c)$, then the joint spectral radius of the $d-$tuple, $A \in \C ^d _n$, is bounded above by the row-norm of $A$. Hence $A' := r \cdot A$, $r ^{-1} := (1 +\eps ) \| A \|$ has $\mr{spr} (A' ) <1$, and if $\fr '$ has minimal realization $(A ' , b , c )$, then $\fr ' (Z) = \fr \left( r \cdot Z \right)$ is a re-scaling of $\fr$ so that $\fr ' \in \mult$.
\end{remark}

\subsection{Minimal realizations of $\fr \in \mult$} \label{ss:minFM}

The minimal realization of any $\fr \in \hardy$ is easily constructed as follows. Let $c = \fr$, set 
$$ \scr{M}  := \bigvee R^{*\om} \fr, \quad A_k := R^* _k | _{\scr{M}}, $$ and $b:= P_{\scr{M}} 1$. Since $\fr = K \{ Z ,y , v \}$ is an NC Szeg\"o kernel vector at a finite point $Z \in \B ^d _n$,
$$ \scr{M}  = \bigvee K \{ Z , y , Z^\om v \}, $$ is finite--dimensional. It is easily checked that the triple $(A,b,c)$ is a realization of $\fr$. The realization $(A,b,c)$ is controllable by construction, and it is also straightforward to check that it is observable. Alternatively, a minimal realization of $\fr$ can also be constructed by applying backward left shifts to $\frt$.

It will be convenient to also consider \emph{Fornasini--Marchesini (FM) realizations} of $\fr \in \hardy$. Here, an FM realization of $\fr \in \hardy$ is a quadruple $(A,B,C,D)$, where $A \in \C ^d _n$, $B \in \C ^n \otimes \C ^d$, $C \in \C  ^{1\times n}$ and $D \in \C$, so that for any $Z \in \nbdom \fr \supseteq \B ^d _\N$,
$$ \fr (Z) = D + C (I-ZA) ^{-1} B. $$ The NC rational function, $\fr$, is called the \emph{transfer function} of the \emph{FM colligation}:
$$ U_\fr := \bpm A & B \\ C & D \epm : \bpm \C ^m \\ \C \epm \rightarrow \bpm \C ^m \otimes \C ^d \\ \C \epm. $$ 
As before, such a realization is called \emph{controllable} if 
$$ \C ^m = \bigvee _{\substack{\om \in \F ^d \\ 1 \leq j \leq d }} A^\om B_j, $$ \emph{observable} if 
$$ \C ^m = \bigvee A^{*\om} C ^*, $$ and \emph{minimal} if $N$ is as small as possible. Again, an FM realization is minimal if and only if it is both observable and controllable, and minimal FM realizations are unique up to joint similarity \cite[Theorem 2.1]{Ball-control}. It is straightforward to pass back and forth between minimal descriptor and FM realizations. For example, beginning with a minimal descriptor realization $(A,b,c)$, set $\scr{M} _0 := \bigvee _{\om \neq \emptyset} A^\om c$, with projector $P_0$ and $A ^{(0)} := A | _{\scr{M} _0}$. Then define $B_k = A_k c$, $C := (P_0 b )^*$ and $D:= \fr (0)$. Then $( A^{(0)}, B , C , D )$ is a minimal FM realization of $\fr$. Similarly, a minimal descriptor realization can be constructed from a minimal FM realization.  

Any $b \in [ \mult ] _1$ has a (generally not finite--dimensional) \emph{de Branges--Rovnyak realization} \cite{BBF}. This is a Fornasini--Marchesini--type realization constructed using free de Branges--Rovnyak spaces. Here, given $b \in [ \mult ] _1$, the right \emph{free de Branges--Rovnyak space}, $\scr{H} ^\mrt (b)$, is the operator--range space of the operator $\sqrt{I - b(R)  b(R) ^*}$. That is, $\scr{H} ^\mrt (b) = \nbran \sqrt{I - b(R)b(R) ^*}$ as a vector space, and the norm on $\scr{H} ^\mrt (b)$ is defined so that $\sqrt{I - b(R) b(R) ^*}$ is a co-isometry onto its range. Equivalently, $\scr{H} ^\mrt (b)$ is the NC-RKHS $\cH _{nc} (K^b)$ with CPNC kernel:
$$ K ^b (Z,W) [ \cdot ] := K(Z,W) [ \cdot ] - K(Z,W) [b ^\mrt (Z) ( \cdot ) b^\mrt (W) ^*], $$ and NC kernel vectors
$$ K^b \{ Z, y ,v \} := (I - b(R) b(R) ^* ) K\{ Z , y ,v \} = K \{ Z , y ,v \} - b(R) K \{ Z ,y , b^\mrt (Z) v \}, $$
where $K(Z,W)$ is the CPNC Szeg\"o kernel of the free Hardy space and $K\{Z ,y , v \}$ is an NC Szeg\"o kernel vector. Any free right de Branges--Rovnyak space is contractively contained in $\hardy$, and it is always co-invariant for the left free shifts. While $b ^\mrt = b(R) 1$ does not belong to $\scr{H} ^\mrt (b)$ in general, $L^* _k b ^\mrt$ always belongs to $\scr{H} ^\mrt (b)$ \cite[Proposition 4.2]{BBF}.
One then defines the co-isometric \emph{de Branges--Rovnyak colligation}
$$ \mbf{U} _b := \bpm \mbf{A} & \mbf{B} \\ \mbf{C} & \mbf{D} \epm : \bpm \scr{H} ^\mrt (b) \\ \C \epm \rightarrow \bpm \scr{H} ^\mrt (b) \otimes \C ^d \\ \C \epm, $$ where
$$ \mbf{A} := L^* | _{\scr{H} ^\mrt (b)}, \quad \mbf{B} := L^* b^\mrt, \quad \mbf{C} := (K_0 ^b)^* \quad \mbox{and} \quad \mbf{D} := b(0). $$ One can then check that $b$ is realized as the transfer--function of this colligation.  In particular, if $b=\fb \in [\mult ] _1$ is an NC rational multiplier, one can cut--down this realization to obtain a finite--dimensional and minimal \emph{de Branges--Rovnyak FM realization} by setting
$$ \scr{M} _0 (\fb ) := \bigvee _{\om \neq \emptyset} L^{*\om} \fb ^\mrt, $$ with projector $P_0$ and then 
$$  U_\fb := \mbf{U} _\fb \bpm P_0 & 0 \\ 0 & 1 \epm  =  \bpm \mf{A} & \mf{B} \\ \mf{C}  & \mf{D}  \epm, $$ where 
$$ \mf{A} := \mbf{A} | _{\scr{M} _0 (\fb )}, \quad \mf{B} := \mbf{B}, \quad \mf{C} := \mbf{C} P_0 \quad \mbox{and} \quad \mf{D} := \mbf{D} = \fb (0). $$ 
Here note that if $\fb = K\{ Z, y , v \} \in \hardy$ is NC rational, then $\fb ^\mrt$ is also an NC Szeg\"o kernel at some finite point $W \in \B ^d _n$ by Lemma \ref{transkernel} and Theorem \ref{ncratinH2}. It immediately follows that $\scr{M} _0 (\fb )$ is finite--dimensional. 

\begin{lemma} \label{min}
The above finite de Branges--Rovnyak FM realization $(\mf{A} ,\mf{B} , \mf{C} , \mf{D} )$ of $\fb \in [\mult ] _1$ is minimal.
\end{lemma}
The proof is routine and omitted. As before one can alternatively construct a de Branges--Rovnyak realization of any $b \in [ \mult ] _1$ (or minimal de Branges--Rovnyak FM realization of an NC rational $\fb \in [ \mult ] _1$) by considering the left free de Branges--Rovnyak space $\scr{H} (b)$, the operator--range space of $\sqrt{I - b(L) b(L) ^*}$. This space is right-shift co-invariant and $R^{*\om} b \in \scr{H} (b)$ for any $\om \neq \emptyset$.  

\subsection{Clark measures} \label{ss:Clark}

In classical Hardy space theory, there are (essentially) bijections between contractive analytic functions in the disk, Herglotz functions, \emph{i.e.} analytic functions in $\D$ with positive harmonic real part and positive, finite and regular Borel measures on the unit circle. Namely, beginning with a positive measure on the circle, $\mu$, one can define its \emph{Herglotz--Riesz integral transform}:
$$ H_\mu (z) := \int _{\partial \D} \frac{1 + z\ov{\zeta}}{1-z\ov{\zeta}} \, \mu (d\zeta ), \quad z \in \D,$$ and this produces a Herglotz function in the disk. Note that $\nbre H_\mu (0) = \mu (\partial \D ) > 0$. Since $\nbre H_\mu (z) \geq 0$, applying the so-called \emph{Cayley Transform}, a fractional linear transformation which takes the complex right half-plane onto the unit disk, yields a contractive analytic function:
$$ b_\mu (z) = \frac{H_\mu (z) -1}{H_\mu (z) +1}. $$ Each of these steps is essentially reversible.
Beginning with a contractive analytic function $b \in [H^\infty ] _1$, its inverse Cayley transform, 
$$ H_b (z) := \frac{1+b(z)}{1-b(z)}, $$ is a Herglotz function. Moreover, given any Herglotz function, $H$, in the disk, the Herglotz representation theorem implies that there is a unique positive measure $\mu$ so that
$$ H(z) = i \nbim H(0) + H_\mu (z) = i \nbim H(0) + \int _{\partial \D} \frac{1 + z\ov{\zeta}}{1-z\ov{\zeta}} \, \mu (d\zeta ). $$ That is, the Herglotz function corresponding to a positive measure is unique modulo imaginary constants. If $H = H_b$, this concomitant measure is called the \emph{Aleksandrov--Clark measure} or \emph{Clark measure} of $b$ \cite{Clark,Aleks1,Aleks2}. Hence any two contractive multipliers $b_1, b_2 \in [ \mult ] _1$ whose Herglotz functions $H_k := H_{b_k}$ differ by an imaginary constant have the same Clark measure. In this case if $H_2 = H_1 +it$ for some $t \in \R$, then one can check that $$ b_2 = \frac{\ov{z (t)} }{ z(t)} \cdot \mu _{z(t)} \circ b_1, $$ is, up to a unimodular constant, a M\"obius transformation of $b_1$ corresponding to the point
$$ z(t) := \frac{t}{2i +t} \in \D, $$ so that the contractive analytic functions corresponding to a given positive measure are unique up to such transformations.

By the Riesz--Markov theorem, any positive, finite and regular Borel measure on $\partial \D$ can be viewed as a positive linear functional on the $C^*-$algebra of continuous functions on the unit circle, $\scr{C} (\partial \D )$. Recall that the \emph{disk algebra}, $A (\D )$, is the unital Banach algebra of analytic functions in the disk which extend continuously to the boundary and that this algebra is isomorphic to the operator algebra $\mr{Alg} \{ I , S \} ^{-\| \cdot \|}$, where $S=M_z : H^2 \rightarrow H^2$ is the shift. By the Weierstrass approximation theorem, $\scr{C} (\partial \D)$ is the supremum norm-closed linear span of the \emph{disk algebra} and its conjugates. That is, $\scr{C} (\partial \D) = \left( A(\D ) + A(\D ) ^* \right) ^{-\| \cdot \|}$. In the NC multi-variable setting of Fock space, the immediate analogue of a positive measure is then any positive linear functional on the norm-closed operator system of the \emph{free disk algebra}, $\A := \mr{Alg} \{ I , L_1 , \cdots , L_d \} ^{-\| \cdot \|}$. We will use the notation,
$$ \scr{A} _d := \left( \A + \A ^* \right) ^{-\| \cdot \|}, $$ for the \emph{free disk system}, and $\posncm$ will denote the set of positive \emph{NC measures}, \emph{i.e.} the set of all positive linear functionals on the free disk system.

As in the single-variable setting, one can define a \emph{free Herglotz--Riesz transform} of any positive NC measure $\mu \in \posncm$ and this produces an NC Herglotz function, $H_\mu$, which has positive semi-definite real part in the NC unit row-ball, $\B ^d _\N$. Namely, given $\mu \in \posncm$ and $Z \in \B ^d _n$, the Herglotz--Riesz transform of $\mu$ is:
$$ H_\mu (Z) := \mr{id} _n \otimes \mu \left( (I +ZL^*) (I - ZL^* ) ^{-1} \right); \quad \quad ZL^* := Z_1 \otimes L_1 ^* + \cdots + Z_d \otimes L_d ^*. $$  As before, the Cayley Transform of any such $H_\mu$ defines a bijection between NC Herglotz functions and contractive left multipliers, $b_\mu$, of the Fock space. Also as before, the correspondence $\mu \leftrightarrow H_\mu$ is bijective modulo imaginary constants, and if a positive NC measure, $\mu$, corresponds to a contractive left multiplier $b \in [ \mult ] _1$, we write $\mu = \mu _b$, and we call $\mu$ the NC Clark measure of $b$ \cite{JM-freeAC,JM-freeCE}.

\section{NC rational Clark measures} \label{sect:ncratClark}

One can apply a Gelfand--Naimark--Segal (GNS) construction to any $\mu \in \posncm$ to obtain a GNS--Hilbert space, $\hardy (\mu)$, the completion of the free polynomials, $\fp = \C \{ \mf{z} _1 , \cdots , \mf{z} _d \}$, modulo vectors of zero length, with respect to the GNS pre-inner product:
$$ \ip{p}{q}_\mu := \mu (p(L) ^* q(L) ). $$ Equivalence classes, $p + N_\mu \in \hardy (\mu )$, where $p \in \fp$ is a free polynomial and $N_\mu$ denotes the left ideal of zero-length vectors with respect to the $\mu-$pre-inner product, are dense in $\hardy (\mu )$.  This construction also comes equipped with a left regular representation of the free disk algebra,
$$ \pi _\mu (L_k) p + N_\mu := \mf{z}_k p + N_\mu. $$ This representation is unital, completely isometric and extends to a $*-$representation of the Cuntz--Toeplitz algebra so that $\Pi _\mu = \left( \Pi _{\mu ; 1 } , \cdots , \Pi _{\mu ;d } \right) := \pi _\mu (L)$ is a \emph{GNS row isometry} acting on $\hardy (\mu )$. For details, see \cite{JM-freeCE,JM-freeAC,JM-ncFatou,JM-ncld}. Any cyclic $*-$representation of the Cuntz--Toeplitz algebra can be obtained, up to unitary equivalence, as the GNS representation of a positive NC measure \cite[Lemma 2.2]{JMT-ncFMRiesz}. 

\begin{defn}
A positive NC measure, $\mu \in \posncm$, is a \emph{finitely--correlated Cuntz--Toeplitz functional} if the subspace, 
$$ \cH _\mu := \bigvee _{\om \in \F ^d} \Pi _\mu ^{*\om} \left(I + N _\mu\right), $$ is finite dimensional. If $\Pi _\mu$ is also a Cuntz row isometry, \emph{i.e.} a surjective row isometry, we say that $\mu$ is a finitely--correlated Cuntz functional. 
\end{defn}

\begin{remark}
In \cite{BraJorg}, finitely--correlated Cuntz states were defined as unital and positive linear functionals on the Cuntz algebra with the above property. However, if $\mu \in \posncm$ is a finitely--correlated Cuntz state according to our definition, \emph{i.e.} if $\mu$ is a unital, finitely--correlated Cuntz functional on the free disk system, then $\Pi _\mu$ is Cuntz, and in this case, $\mu$ has a unique positive extension to the Cuntz--Toeplitz algebra \cite[Proposition 5.11]{JMT-ncFMRiesz}. Moreover since $\Pi _\mu$ is Cuntz, this defines a unique finitely--correlated Cuntz state in the sense of Bratteli and J\o rgensen \cite{BraJorg}.
\end{remark}

\begin{thm} \label{fincorthm}
An NC measure, $\mu \in \posncm$, is the NC Clark measure of a contractive NC rational multiplier of Fock space, $\fb \in [\mult ] _1$, if and only if it is a finitely--correlated Cuntz--Toeplitz functional. 
\end{thm}

It will be convenient to recall the construction of the free Cauchy transform of elements of the GNS space $\hardy (\mu)$ \cite{JM-freeCE,JM-ncld,JM-ncFatou,JM-freeAC}. Given any $p \in \fp$, $p + N_\mu \in \hardy (\mu )$ and $Z \in \B ^d _n$, the right \emph{free Cauchy transform} of $p + N_\mu$ is the NC function $\scr{C} _\mu p \in \scr{O} (\B ^d _\N )$, 
\ba (\scr{C} _\mu p)  (Z) & := & \mrt \circ \mr{id} _n \otimes \mu \left( (I_n \otimes I - Z \otimes L^* ) ^{-1} p(L) \right) \nn \\
& = & \sum _{\om} Z^{\om ^\mrt}  \mu \left( L^{*\om} p(L) \right) \nn \\
& = & \sum Z^\om \ip{ \mf{z} ^{\om }  + N_\mu}{p + N_\mu}_\mu, \nn \ea and this final formula extends to any $x \in \hardy (\mu )$. Equipping this vector space of free Cauchy transforms with the inner product that makes $\scr{C} _\mu$ an onto isometry produces an NC-RKHS in $\B ^d _\N$, $\scr{H} ^+ (H _\mu )$, the right \emph{NC Herglotz space} of $\mu$, with CPNC kernel $K^\mu$: For any $Z \in \B ^d _n $ and $W \in \B ^d _m$,
$$ K^\mu (Z,W) [\cdot ] =    K (Z,W) \left[ \frac{1}{2} H _\mu ^\mrt (Z)  (\cdot)  + \frac{1}{2}  (\cdot ) H_\mu ^\mrt (W) ^* \right], $$ where $K$ denotes the CPNC Szeg\"o
kernel of the Fock space, and $H_\mu (Z)$ is the (left) NC \emph{Herglotz--Riesz transform} of $\mu$: For any $Z \in \B ^d _n$,
\ba  H_\mu (Z) & := & \mr{id} _n \otimes \mu \left( (I_n \otimes I + Z L ^* ) (I _n \otimes I - ZL ^* ) ^{-1}  \right) \nn \\
& = & 2 (\scr{C} _\mu 1 ) (Z)  - \mu (I) I_n. \nn \ea Any such $H_\mu$ is an NC Herglotz function in $\B ^d _\N$ as described in Section \ref{ss:Clark}. That is, $\nbre H_\mu (Z) \geq 0$. 

The image of the GNS row isometry, $\Pi _\mu$, under right free Cauchy transform is a row isometry, $V_\mu$, acting on $\scr{H} ^+ (H _\mu )$:
\be V _\mu = \scr{C} _\mu \Pi _\mu \scr{C}  _\mu  ^* := \scr{C} _\mu \left( \Pi _{\mu ; 1} , \cdots , \Pi _{\mu ; d } \right) \scr{C} _\mu ^* \otimes I_d : \scr{H} ^+ (H _\mu ) \otimes \C ^d \rightarrow \scr{H} ^+ (H _\mu ), \ee where  $\Pi _{\mu ; k } = \pi _\mu (L_k)$. The range of the row isometry $V _\mu $ is,
\be \nbran V_\mu =\bigvee _{\substack{(Z,y,v) \in \\  \B ^d _n \times \C ^n \times \C ^n;  \ n \in \N }} \left( K^{\mu} \{ Z , y , v \} - K^{\mu} \{ 0 _n , y, v \} \right) \ee and
for any $Z \in \B ^d _n, \ y, v \in \C^n$, 
\be 
V_\mu   ^* \left( K^{\mu} \{ Z , y , v \} - K^{\mu} \{ 0 _n , y, v \} \right) = K ^\mu \{ Z , Z^* y , v \} := \bsm K ^{\mu} \{ Z , Z_1 ^* y ,  v \}  \\ \vdots \\ K ^{\mu} \{ Z ,  Z_d ^* y ,  v \} \esm \in \scr{H} ^+ (H _\mu ) \otimes \C ^d. \label{Vmuaction} \ee
The linear span of all such vectors is dense in $\scr{H} ^+  (H_\mu) \otimes \C ^d$ since $V_\mu ^*$ is a co-isometry. See \cite[Section 4.4]{JM-freeCE} for details. 

\begin{lemma} \label{bwshift}
Each $V_{\mu ; k} ^*$ acts as a backward left shift on $\scr{H} ^+ (H _\mu )$, \emph{i.e.} given any $h \in \scr{H} ^+ (H _\mu )$ with Taylor--Taylor series at $ 0 \in \B ^d _1 $, 
$$ h(Z) = \sum _\om \hat{h} _\om Z^\om \quad \mbox{and} \quad (V_k ^* h ) (Z) = \sum _\om \hat{h} _{kw} Z^\om. $$
\end{lemma}
Here, note that if $h (Z) = \sum \hat{h} _\om Z^\om$ and $h \in \hardy$, that $(L_k ^* h) (Z) = \sum _\om \hat{h} _{k\om} Z^\om$. This motivates the terminology `backward left shift'.
\begin{proof}
The right free Cauchy transform of any $x \in \hardy (\mu)$ is $h := \scr{C} _\mu x$,
$$ h (Z) = \sum _\om Z^\om \underbrace{\ip{\mf{z} ^\om + N _\mu}{x}_\mu}_{=: \hat{h} _\om },$$ so that 
\ba (V_{\mu ; j}  ^* h ) (Z) & = & \sum _\om  Z^\om \ip{\mf{z} ^\om + N_\mu}{\Pi _{\mu ;j} ^* x } _\mu \nn \\
& = & \sum Z^\om \ip{\mf{z} _j \mf{z} ^{\om} + N_\mu}{x}_\mu = \sum _\om \hat{h} _{j\om} Z^\om. \nn \ea 
\end{proof}

\begin{proof}{ (of Theorem \ref{fincorthm})}
First, assume that $\mu \in \posncm$ is a finitely--correlated Cuntz--Toeplitz functional. Set $T_{\mu}^* = \Pi_{\mu}^*|_{\cH _\mu}$. Let $\mf{H} _\mu (Z)$ be the (left) NC Herglotz--Riesz transform of $\mu$ \cite[Theorem 3.4]{JM-freeCE}: For any $Z \in \B ^d _n$,
\ba \mf{H}_\mu (Z)  &= &  \mr{id} _n \otimes \mu \left( (I + Z  L^* ) (I - Z  L^* ) ^{-1} \right) \nn \\
& = & 2 \,  \mr{id} _n \otimes \mu \left( (I \otimes I -Z \otimes L^* ) ^{-1} \right) - \mu (I) I_n \nn \\
& = & 2 \sum _{\om \in \F ^d} Z^\om \ip{ 1 + N_\mu}{\Pi _\mu ^{*\om} 1 + N_\mu }_\mu - \mu (I) I_n \nn \\
& = & 2 \sum _{\om \in \F ^d} Z^\om \ip{ 1 + N_\mu}{T_\mu ^{*\om} 1 + N_\mu }_\mu - \mu (I) I_n \nn \\
& =: & 2 \mf{G}_\mu (Z) - \mf{G}_\mu (0_n) I_n, \nn
\ea where $\mf{G}_\mu (Z)  := (\scr{C} _\mu 1 ) (Z)$. Hence $(A,b,c) := (T^* _\mu , 1+ N_\mu , 1 + N_\mu )$ is a finite dimensional realization of $\mf{G}_\mu (Z)$. Moreover, clearly $1+N_\mu$ is cyclic for $T^* _\mu$ by definition of $\cH _\mu$, so that this realization is controllable. Similarly it is observable since $1+N_\mu$ is cyclic for $\Pi _\mu$. Indeed, since 
$$ \hardy (\mu ) = \bigvee \Pi _\mu ^\om \left( 1 + N_\mu \right), $$ it follows that if $P_\mu$ is the orthogonal projection onto $\cH _\mu$, then 
$$ \cH _\mu = P_\mu \hardy (\mu) = \bigvee T_\mu ^\om \left(1 + N_\mu\right), $$ and $(A,b,c)$ is the minimal realization of $\mf{G} _\mu$. Since $\mf{G} _\mu$ has a finite descriptor realization, it is an NC rational function with $\nbdom \mf{G} _\mu \supseteq \B ^d _\N$, and so $\mf{H} _\mu$ is also an NC rational function in $\B ^d _\N$. Applying Cayley transform,
$$ \fb _\mu (Z) := ( \mf{H} _\mu (Z) - I_n ) ( \mf{H} _\mu (Z) + I_n ) ^{-1} \in [ \mult ] _1, $$ is a contractive, NC rational left multiplier of Fock space with NC Clark measure $\mu$. \\

Conversely, if $\fb \in [\mult ] _1$ is NC rational, then we can reverse the above argument to see that $\mf{G} _{\mu}$ is NC rational. Moreover, by Lemma \ref{bwshift}, we have that for every $j=1,\ldots,d$,
\[
(V_{\mu,j}^* \mf{G} _{\mu})(Z) =  \sum_{\omega} Z^{\omega} \langle \mf{z}^{j \omega} + N_{\mu}, 1 + N_{\mu} \rangle _\mu.
\]
If  $(A,b,c)$ is a minimal realization of $\mf{G} _{\mu} (Z) = \sum _\om \hat{\mf{G}} _{\mu ; \om } Z^\om$, then for every word $\omega$, 
$$ \hat{\mf{G}} _{\mu ; \om} = \langle \mf{z}^{\omega} + N_{\mu}, 1 + N_{\mu} \rangle =  b^* A^{\omega} c, $$ and it follows that the minimal realization of $V_{\mu,j}^* \mf{G} _{\mu}$ is $(A,A_j^* b,c)$. Such `backward left shifts' of NC rational functions were studied in \cite[Section 2]{KVV-ncrat}.
To show that $\mu$ is finitely--correlated, we need to show that 
 $\cH _\mu$ is finite dimensional. Equivalently, we can show that 
 $$ \scr{M} _\mu := \bigvee V_\mu ^{*\om} \mf{G} _\mu, \quad \scr{M} _\mu = \scr{C} _\mu \cH _\mu, $$ is a finite dimensional subspace of the NC Herglotz space $\scr{H} ^+ (\mf{H} _\mu )$ of $\mu-$Cauchy transforms. It follows that $\scr{M} _\mu$ is finite dimensional since $\bigvee A^{*\om} b$ is finite dimensional, by assumption.

\end{proof}
\begin{remark}
It also follows from the above proof that $( T _\mu ^* , 1 + N_\mu , 1 + N_\mu )$ is a minimal descriptor realization of $\mf{G} _\mu$, the free Cauchy transform of $1 + N _\mu$. 
\end{remark}

\subsection{Row isometric dilations of finite row contractions}

Let $\mu \in \posncm$ be a finitely--correlated Cuntz--Toeplitz functional, and $\Pi = \Pi _\mu$. Then $1 + N_\mu$ is $\Pi-$cyclic so that by Popescu's NC Wold decomposition, $\Pi = \Pi _L \oplus \Pi _{Cuntz}$ where $\Pi _L$ is unitarily equivalent to $L$, and $\Pi _{Cuntz}$ is a cyclic Cuntz (surjective) row isometry \cite[Theorem 1.3]{Pop-dil}. If we define the finite dimensional space, 
$$ \cH _\mu := \bigvee \Pi ^{*\mu} (1 + N_\mu ), $$ with projection $P_\mu$, then $\Pi _\mu$ is the minimal row isometric dilation of the finite row contraction
$$ T_\mu := P_\mu \Pi _\mu | _{\cH _\mu \otimes \C ^d}. $$ In particular, $\Pi _\mu$ will be Cuntz if and only if $T_\mu$ is a row co-isometry by \cite[Proposition 2.5]{Pop-dil}. Let $A := (A_1 , \cdots , A _d ) : \cH \otimes \C ^d \rightarrow \cH$ be any row contraction on a finite dimensional Hilbert space $\cH \simeq \C ^n$. Let $V = V = (V_1 , \cdots , V_d ) : \cK \otimes \C ^d \rightarrow \cK$ be the minimal row isometric dilation of $A$ on $\cK \supsetneqq \cH$. Such row isometries, $V$, as well as the structure of the unital, weak operator topology (WOT) closed algebras they generate, were completely characterized and classified up to unitary equivalence by Davidson, Kribs and Shpigel in \cite{DKS-finrow}. We will have occasion to apply several results of \cite{DKS-finrow} and so we will record some of the main results of this paper here for future reference:  

Given $A, V, \cH$ and $cK$ as above, let $V = V_p \oplus V'$ be the Wold decomposition of $V$ concomitant to $\cK = \cK _p \oplus \cK '$, where $V_p \simeq L \otimes I_\cJ$ is pure and $V'$ is a Cuntz row isometry on $\cK '$. Further let $\wt{\cH}$ be the span of all minimal $A-$co-invariant subspaces $\wt{\cJ} $ of $\cH$ so that $\wt{B} := (A ^* | _{\wt{\cJ}  } ) ^*$ is a row co-isometry. Then, $\wt{\cH} = \bigoplus \wt{\cH}  _k$, where $\{ \wt{\cH} _k \} _{k=1} ^N$ is a maximal family of mutually orthogonal and minimal $A^*-$invariant subspaces so that $( A^* | _{\wt{\cH}  _k}) ^*$ is a row co-isometry. Note that if $\wt{\cJ} $ is minimal, then $\wt{B}$ is necessarily an irreducible row co-isometry. The following theorem is part of the statement of \cite[Theorem 6.5]{DKS-finrow}.

\begin{thmA}[Davidson--Kribs--Shpigel] \label{DKSdiltype}
Let $A$ be a finite--dimensional row contraction on $\cH$ with minimal row isometric dilation $V = V_p \oplus V'$ on $\cK = \cK _p \oplus \cK ' \supsetneq \cH$, with notations as above.  Then $V_p$ is unitarily equivalent to $L \otimes I_\cJ$ where $\nbdim \cJ = \mr{rank} \, I - A A^*$ and $V'$ is the minimal row isometric dilation of the row co-isometry $\wt{A} := (A^* | _{\wt{\cH}} ) ^*$. Moreover, $V ' = \bigoplus _{k=1} ^N V' _k $, where each $V' _k$ is an irreducible Cuntz row isometry and $V' _k$ is the minimal row isometric dilation of the irreducible row co-isometry $\wt{A} ^{(k)} := ( A ^* | _{\wt{\cH } _k} ) ^*$.
\end{thmA}
In the above statement, recall that a finite row contraction, $A$, on $\cH$ is said to be irreducible, if it has no non-trivial jointly invariant subspace. This is equivalent to $\mr{Alg} \{ I, A_1 , \cdots , A_d \} = \scr{L} (\cH )$. If $A$ is irreducible, then $A$ cannot have any non-trivial jointly co-invariant subspace either, so that also $\mr{Alg} \{ I, A _1 ^* , \cdots , A_d ^* \} = \scr{L} (\cH )$. In this case, any vector $x \in \cH$ is cyclic for both $A$ and $A^*$. On the other hand, we say that a row isometry, $V = (V_1, \cdots , V_d )$, is \emph{irreducible} if and only if the $V_k$, $1\leq k \leq d$, have no non-trivial jointly reducing subspace, \emph{i.e.} a subspace which is both invariant and co-invariant for each $V_k$.

The following theorem characterizes when the minimal row isometric dilations of two finite--dimensional row contractions are unitarily equivalent:
\begin{thmA}{ (\cite[Theorem 6.8]{DKS-finrow})} \label{DKSue}
Let $A := ( A_1 , \cdots , A _d )$ and $B := (B_1 , \cdots , B_d )$ be finite--dimensional row contractions acting on finite-dimensional Hilbert spaces $\cH _A$ and  $\cH _B$, respectively. Let $\Pi _A$, $\Pi _B$ be their minimal row isometric dilations acting on $\cK _A \supseteq \cH _A$ and $\cK _B \supseteq \cH _B$. Let $\wt{\cH} _A \subseteq \cH _A$ be the subspace spanned by all minimal $A-$co-invariant subspaces, $\cH$, of $\cH _A$ on which $A^*| _\cH$ is a column isometry and similarly define $\wt{\cH} _B$. Then $\Pi _A$ and $\Pi _B$ are unitarily equivalent if and only if:
\bn
    \item $\mr{rank} \, (I - A A^*) = \mr{rank} \, (I - BB^*)$ and 
    \item $A^* | _{\wt{\cH} _A}$ is jointly unitarily equivalent to $B^* | _{\wt{\cH} _B}$.
\en
\end{thmA}

We will apply these results to study and characterize the GNS row isometry, $\Pi _\mu$, arising from a finitely--correlated positive NC measure, $\mu$.

\begin{lemma} \label{Tcyclic}
Let $\mu \in \posncm$ be a finitely--correlated Cuntz--Toeplitz functional. Then the $\Pi _\mu-$cyclic vector $1+N_\mu$ is cyclic for both $T_\mu$ and $T_\mu ^*$. 
\end{lemma}
\begin{proof}
 By definition, the finite-dimensional subspace
$$ \cH _\mu = \bigvee \Pi _\mu ^{*\om} \left(1 + N_\mu\right) = \bigvee T_\mu ^{*\om} \left(1 + N_\mu\right), $$ is $T_\mu ^*-$cyclic. Also, by the GNS construction of $\hardy (\mu )$, $1+N_\mu$ is $\Pi _\mu -$cyclic. However, since $1 + N_\mu \in \cH _\mu$ is $\Pi _\mu-$cyclic, given any $h \in \cH _\mu$, there is a sequence of polynomials $p_n \in \fp$ so that $p_n (\Pi _\mu ) 1 + N _\mu \rightarrow h$.
Hence,
\ba h & = & P _{\cH _\mu} h \nn \\
& = &  P _{\cH _\mu} \lim p_n (\Pi _\mu ) 1 + N _\mu \nn \\
& = & \lim p_n (P_{\cH _\mu} \Pi _\mu P _{\cH _\mu}) 1 + N_\mu \nn \\
& = & \lim p_n (T_\mu ) 1 + N _\mu, \nn \ea since $\cH _\mu$ is $\Pi _\mu-$co-invariant. It follows that $1 +N_\mu$ is also cyclic for $T_\mu$.
\end{proof}

Let $T := (T_1 , \cdots , T_d ) : \cH \otimes \C ^d \rightarrow \cH$ be any row contraction on a finite dimensional Hilbert space, $\cH$.
Given any $x \in \cH$, define 
$$ \cH ' := \bigvee T^{*\om} x \subseteq \cH \quad \mbox{and} \quad T' := \left( T^* | _{\cH '} \right) ^*. $$ 
Finally, define 
$$ \check{\cH}  := \bigvee T^{ ' \om} x \subseteq \cH', $$ with projector $\check{P}$  and $\check{T} := T' | _{\check\cH \otimes \C ^d }.$
Observe that $\check{\cH}$ is $T'-$invariant and $T-$semi-invariant, \emph{i.e.} it is the direct difference of the nested, $T-$co-invariant subspaces
$$ \cH ' \quad \mbox{and} \quad \cH ' \ominus \check\cH.$$  Let $\cK := \bigvee V^\om \cH$ be the Hilbert space of the minimal row isometric dilation, $V$, of $T$, set 
$$ \cK _x := \bigvee V^\om x, $$ with projection, $P_x$, and let $V_x := V| _{\cK _x}$. 

\begin{prop} \label{confincor}
Given $T,x$ as above, the linear functional, $\mu := \mu _{T,x} \in \posncm$, defined by
$$ \mu _{T,x} (L^\om ) := \ip{x}{T^\om x}_{\cH} = \ip{x}{\check{T} ^\om x}_{\check\cH}, $$ is a finitely--correlated positive NC measure. The vector $x$ is both $\check{T}$ and $\check{T} ^*-$cyclic. The map 
$$ \Pi _\mu ^\om \left(1 + N_\mu\right) \stackrel{U_x}{\mapsto} V^\om x, $$ is an isometry of $\hardy (\mu )$ onto $\cK _x$, and 
$ U_x p(T_\mu ) ^* \left(1 + N_\mu\right) = P_x p(T) ^*x$. If $x$ is $V-$cyclic, then $V \simeq \Pi _\mu$ and $\check{T} \simeq T_\mu$ are unitarily equivalent. 
\end{prop}
\begin{proof}
By \cite[Theorem 2.1]{Pop-ncdisk}, the map $L^\om \mapsto T^\om$ is completely contractive and unital, and hence extends to a completely positive and unital map of the free disk system into $\scr{L} (\cH )$. In particular, $\mu = \mu _{T,x} \in \posncm$. The vector $x$ is, by definition $T^{'*}-$cyclic, so that 
$$ \bigvee \check{T} ^{*\om} x = \check{P} \bigvee T^{'*\om} x = \check{P} \cH ' = \check\cH. $$ This proves that $x$ is $\check{T} ^*-$cyclic. Since $\check{\cH}$ is $T-$semi-invariant, it follows that $\check{P} T^\om \check{P} = \check{T} ^\om$. Hence,
\ba \check{\cH} & = & \bigvee T^{'\om} x =  P' \bigvee T^{\om} x \nn \\
& = & \check{P} \bigvee  T ^\om  \check{P} x = \bigvee \check{T} ^\om x, \nn \ea so that $x$ also is $\check{T}-$cyclic. Semi-invariance further implies that $\mu = \mu _{T,x} = \mu _{\check{T} , x}$. 

To see that $\mf{z}^\om + N_\mu \mapsto V^\om x$ is an isometry, note that for any free polynomial $p \in \fp$, we can write
$$ p(L) ^* p(L) = 2 \nbre q(L) = q(L) + q(L) ^*, $$ for some $q \in \fp$, and then $p(\Pi ) ^* p (\Pi ) = 2 \nbre q(\Pi )$ for any row isometry $\Pi$. Hence, 
\ba \| p  + N_\mu \| ^2 _\mu & = & 2 \nbre \ip{1+ N_\mu}{q + N_\mu}_\mu \nn \\
& = & 2 \nbre \ip{x}{q(V) x}_{\cK} = 2 \nbre \ip{x}{q(T) x}_{\cH} \nn \\
& = & \| p(V) x \| ^2 _{\cK} = \| U_x \left(p + N_\mu\right) \| ^2. \nn \ea 
Given any $p(T_\mu ) ^* \left(1 + N_\mu\right) \in \cH _\mu$, consider 
\ba \ip{\mf{z} ^\om + N_\mu}{ p(T_\mu ) ^* \left(1 + N_\mu\right)}_\mu & = & \ip{ p \mf{z}^{\om} + N_\mu}{1+N_\mu}_\mu \nn \\
& = & \ip{p(V) V^{\om} x}{x}_{\cK} \nn \\
& = & \ip{V^\om x}{P_x p(T) ^* x}_{\cK} \nn \\
& = & \ip{\mf{z}^\om + N_\mu}{U_x ^* P_x p(T) ^* x}_\mu. \nn \ea 
It follows that 
\be U_x p(T_\mu ) ^* \left(1 + N_\mu\right) = P_x p(T) ^* x. \label{Uxue} \ee Finally, if $x$ is $V-$cyclic, then $\cK _x = \cK$ and $U_x \Pi _\mu ^\om = V^\om U_x$, so that $U_x$ is an onto isometry, and $\Pi _\mu$ and $V$ are unitarily equivalent. Hence,
$$ U_x \cH _\mu = U_x \bigvee \Pi _\mu ^{*\om } \left(1+N_\mu\right) = \bigvee V^{*\om} x = \bigvee T^{*\om} x = \cH '. $$ Moreover, since $x$ is $V-$cyclic, given any $h \in \cH'$, there is a sequence of polynomials $p_n \in \fp$ so that $p_n (V) x \rightarrow h$, and then
$$ h = \lim p_n (V) x = P ' h = \lim P' p_n (V) x = \lim p_n (T' ) x, $$ so that $x$ is also $T'-$cyclic, $\check\cH = \cH '$ and $\check{T} = T'$. It further follows that $T' = \check{T}$ and $T_\mu$ are unitarily equivalent via $U_x$. Indeed, if $\cK _x = \cK$ so that $P_x = I$, then Equation (\ref{Uxue}) becomes
$$ U_x p(T _\mu ) ^*  ( 1 + N_\mu ) = p(T) ^* x. $$ In particular, since $U_x$ restricts to a unitary map from $\cH _\mu$ onto $\cH '$, we obtain that for any $y _\mu := q (T_\mu ) ^*  (1 + N _\mu ) \in \cH _\mu$, $q \in \fp$ and $1\leq k \leq d$, 
$$ U_x T_{\mu ; k} ^* y _\mu = T_k ^* q(T) ^* x = T_k ^* U_x y_\mu. $$ This proves that $T_\mu$ is unitarily equivalent to $T '$.
\end{proof}
\begin{remark}
By the previous proposition, all examples of finitely--correlated Cuntz--Toeplitz functionals can be constructed from finite row contractions.
\end{remark}

\begin{lemma} \label{irredlemma}
Let $T$ be a finite and irreducible row co-isometry on $\cH$, and let $V$ be its minimal (Cuntz) row isometric dilation on $\cK \supseteq \cH$. Then any non-zero $x\in \cH$ is $V-$cyclic and $V$ is irreducible.
\end{lemma}
By \cite[Proposition 2.5]{Pop-dil}, a row contraction is a row co-isometry if and only if its minimal row isometric dilation is a Cuntz row isometry. If $V := \left( V_1 , \cdots , V _d \right)$ is a row isometry on a separable Hilbert space, $\cK$, recall that $\mf{V} := \mr{Alg} \{ I , V _1 , \cdots , V_d  \} ^{-WOT}$ is called the \emph{free semigroup algebra} of $V$ \cite{KRD-survey}.
\begin{proof}
If $T$ is an irreducible row co-isometry, then $\cH$ is the unique, minimal $T-$co-invariant subspace of $\cH$ so that $V$ is irreducible by \cite[Lemma 5.8]{DKS-finrow}. Moreover, by \cite[Theorem 5.2]{DKS-finrow}, since $T$ is irreducible, the weakly-closed unital algebra of $V$ contains $P _{\cH}$, the projection onto $\cH$.  Again, since $T$ is irreducible, given any fixed non-zero $x \in \cH$, any $h \in \cH$ can be written as $h=p(T) x$ for some $p \in \fp$, so that
\ba h & = & p(T) x \nn \\
& = & \underbrace{P _\cH p(V)}_{\in \mf{V}} x. \nn \ea Hence any $h \in \cH$ belongs to the weak and hence Hilbert space norm closure of $\mf{V} x$, so that 
$$ \bigvee \mf{V} x = \bigvee \mf{V} \cH = \cK, $$ since $V$ is the minimal row isometric dilation of $T$. 
\end{proof}
\begin{lemma} \label{Tcyclic2}
Let $T$ be a finite--dimensional row co-isometry on $\cH$ with minimal row isometric dilation $V$ on $\cK$. Any vector $h \in \cH$ is $T-$cyclic
if and only if it is $V-$cyclic.
\end{lemma}
\begin{proof}
If $h$ is $V-$cyclic, then it is clearly also $T-$cyclic. Conversely, if $h$ is $T-$cyclic, consider the space 
$$ \cK (h) := \bigvee V^\om h \subseteq \cK. $$ If $h$ is not $V-$cyclic then $\cK (h) \subsetneqq \cK$, and there is a non-zero $x \in \cK (h) ^\perp$ and $\cK (h) ^\perp$ is $V-$co-invariant. By \cite[Corollary 4.2]{DKS-finrow}, there is a non-zero $g \in \cH \bigcap \bigvee V^{*\om } x \subseteq \cK (h) ^\perp$. Hence, $V^{*\om} g = T^{*\om} g \perp \cK (h)$ for any $\om \in \F ^d$. However, by assumption $h$ is $T-$cyclic so that $g=p(T) h$ for some $p \in \fp$ and 
$$ \| g\| ^2 = \ip{p(T) h}{g} = \ip{h}{p(T) ^* g} =0, $$ contradicting that $g\neq 0$.
\end{proof}

\begin{remark} \label{irredfincor}
If $T$ is an irreducible row co-isometry, then any $x \in \cH$ will be $T^*,T$ and $V-$cyclic by the previous lemma. Proposition \ref{confincor} then implies that if $\mu = \mu _{T,x}$, that $T \simeq T_\mu$ and $V \simeq \Pi _\mu$. 
\end{remark}

In \cite{MK-rowiso}, M. Kennedy refined the Wold decomposition of any row isometry by further decomposing any Cuntz row isometry into the direct sum of three types: Cuntz type-L (or absolutely continuous Cuntz), von Neumann type and dilation type. 
\begin{defn}
A row isometry $\Pi : \cH \otimes \C ^d \rightarrow \cH$ on a separable Hilbert space, $\cH$ is \emph{type--L} or \emph{pure} if $\Pi$ is unitarily equivalent to $L \otimes I _{\cJ}$ for some separable Hilbert space $\cJ$. A Cuntz row isometry, $\Pi$ on $\cH$ is:
\bn
    \item \emph{Cuntz type--L} if the free semigroup algebra, $\mf{S} (\Pi)$, of $\Pi$, is completely isometrically isomorphic and weak$-*$ homeomorphic to the unital $WOT-$closed algebra of $L$, $L^\infty _d$.
    \item \emph{von Neumann type} if $\mf{S} (\Pi )$ is self-adjoint, \emph{i.e.} a von Neumann algebra.
    \item \emph{dilation type} if $\Pi$ has no direct summand of the previous two types. 
\en
\end{defn}

\begin{remark}
Any dilation type row isometry, $\Pi$ has an upper triangular decomposition of the form,
$$ \Pi \simeq \bpm L \otimes I & * \\ & T \epm, $$ so that $\Pi$ has a restriction to an invariant subspace which is unitarily equivalent to a pure row isometry and $\Pi$ is the minimal row isometric dilation of its compression, $T$, to the orthogonal complement of this invariant space. Since $\Pi$ is of Cuntz type, $T$ is necessarily a row co-isometry \cite[Proposition 2.5]{Pop-dil}. 
\end{remark}

\begin{lemma} \label{wandlem}
Let $T$ be a finite--dimensional row contraction on $\cH$ with minimal row isometric dilation $V$ on $\cK \supsetneqq \cH$. Any $V-$reducing subspace, $\cK ' \subseteq \cK$ contains wandering vectors for $V$.
\end{lemma}
\begin{proof}
By \cite[Corollary 4.2]{DKS-finrow} $\cH ' := \cK ' \bigcap \cH \neq \{ 0 \}$. Define the subspace:
$$ \scr{W} ' := \left( \cH' + \bigvee _{j=1} ^d V_j \cH' \right) \ominus \cH '. $$ By \cite[Lemma 3.1]{DKS-finrow}, this is a non-trivial wandering subspace for $V$.
\end{proof}

\begin{thm} \label{GNStype}
Let $V : \cK \otimes \C ^d \rightarrow \cK$ be the minimal row isometric dilation of a finite row contraction $T : \cH \otimes \C ^d \rightarrow \cH$, $\cH \simeq \C ^n$. Then $V$ contains no absolutely continuous Cuntz or von Neumann type direct summand so that $V = V_L \oplus V_{dil}$ is the direct sum of a pure type$-L$ and a dilation-type row isometry.
\end{thm}

\begin{proof}
Lemma \ref{wandlem} implies that any direct summand of $V$ has wandering vectors. It is then an immediate consequence of \cite[Corollary 4.13]{MK-wand} that $V$ has no direct summand of von Neumann type.

Suppose that $V$ had a non-trivial absolutely continuous Cuntz direct summand, $V_{ac}$, acting on the $V-$reducing subspace $\cK _{ac}$. Then by \cite[Corollary 4.2]{DKS-finrow}, $\cH _{ac} := \cK _{ac} \bigcap \cH \neq \{ 0 \}$, and $\cH _{ac}$ is $V_{ac}-$co-invariant. Let $T_{ac} := \left( V_{ac} ^* | _{\cH _{ac} } \right) ^*$, then $V_{ac}$ is the minimal row isometric dilation of $T_{ac}$. Indeed, 
$$ \wt{\cK} _{ac} := \bigvee _\om V_{ac} ^\om \cH _{ac}, $$ is $V_{ac}-$reducing, so that $\cK _{ac} \ominus \wt{\cK} _{ac} =: \cK ' _{ac}$ is also $V_{ac}-$reducing. Then by \cite[Corollary 4.2]{DKS-finrow}, $\cK ' _{ac} \bigcap \cH = \cH ' _{ac} \subseteq \cH _{ac}$ is non-trivial. This contradicts that $\cK ' _{ac} \perp \wt{\cK} _{ac}$, and we conclude that $\wt{\cK} _{ac} = \cK _{ac}$. 

Let $\{ \cH _{ac ; k} \}$ be a maximal family of minimal, pairwise orthogonal and $V_{ac}-$co-invariant subspaces of $\cH _{ac}$. Fix some $k$ and choose any non-zero $h \in \cH _{ac ; k}$. Then since $V_{ac}$ is absolutely continuous, every $y \in \cK _{ac}$ is an \emph{absolutely continuous vector} for $V_{ac}$ in the sense of \cite[Definition 2.4]{DLP-ncld}. In particular, by \cite[Theorem 2.7]{DLP-ncld}, $h \in \cH _{ac ;k} \subsetneq \cK _{ac}$ is in the range of a bounded intertwiner, $X: \hardy \rightarrow \cK _{ac}$. That is, $X L_k = V_{ac; k} X$, and there is a $g \in \hardy$ so that $Xg =h$. Note that $X^*$ must be bounded  below on $\cH _{ac ; k}$. First, $\nbker X^* $ is $V_{ac}-$co-invariant since $X^* x =0$ implies that 
$$ X^* V_{ac} ^{*\om} x = L^{*\om} X^* x =0. $$ The subspace $\cH _{ac ; k}$ is finite dimensional, so that if $X^*$ is not bounded below on this space, then it has non-trivial kernel. However, if $$ \nbker X^* \bigcap \cH _{ac ; k } \neq \{ 0 \}, $$ then this is a proper, non-trivial $V_{ac}-$co-invariant subspace of $\cH _{ac ; k}$, contradicting the minimality of $\cH _{ac ; k}$. Hence, $X^*$ is bounded below by say $\eps >0$ on $\cH _{ac ; k}$. Also note that since $V_{ac}$ is Cuntz, $T_{ac}$ is a row co-isometry. Then, for any $n \in \N$,
\ba \| h \| ^2 & = & \sum _{|\om| = n} \| T_{ac} ^{*\om} h \| ^2 \nn \\
& \leq & \eps ^{-2} \sum _{|\om| = n} \| X^* T_{ac} ^{*\om} h \| ^2 \nn \\
& = & \eps ^{-2} \sum _{|\om| = n} \| X^* V_{ac} ^{*\om} h \| ^2 \nn \\
& = & \eps ^{-2} \sum _{|\om| = n} \| L^{*\om} X^* h \| ^2 \nn \\
&\rightarrow & 0, \nn \ea since $L$ is pure. This contradiction proves the claim.

Alternatively, if $V_{ac}$ is an absolutely continuous direct summand of $V$ acting on $\cK _{ac}$, then the free semigroup algebra, $\mf{V} _{ac} = \mr{Alg} \{ I , V_{ac;1} , \cdots , V_{ac ;d } \} ^{-WOT}$, contains the projection, $P_{ac}$, onto the non-trivial and finite--dimensional subspace $\cH _{ac} := \cH \cap \cK _{ac}$. However, $V_{ac}$ is absolutely continuous so that $\mf{V} _{ac}$ is completely isometrically isomorphic and weak$-*$ homeomorphic to $L^\infty _d$, the left analytic Toeplitz algebra. This produces a contradiction as $L^\infty _d$ contains no non-trival projections by \cite[Corollary 1.5]{DP-inv}.

Any row isometry, $V$, has the Kennedy--Lebesgue--von Neumann--Wold decomposition $V = V_L \oplus V_{C-L} \oplus V_{dil} \oplus V_{vN}$, and we have shown that if $V$ is the minimal row isometric dilation of a finite row contraction, then the Cuntz type$-L$ and von Neumann type direct summands are absent.
\end{proof}

The analogue of normalized Lebesgue measure in this setting is \emph{NC Lebesgue measure}, $m (L^\om) := \ip{1}{L^\om 1}_{\bH ^2}$, the so--called vacuum state of the Fock space. (This NC measure is the NC Clark measure of the identically $0$ multiplier, just as normalized Lebesgue measure on the circle is the Clark measure of the identically $0$ function in the disk.) In \cite{JM-ncFatou} and \cite{JM-ncld}, the first two authors have constructed the Lebesgue decomposition of any positive NC measure $\mu \in \posncm$ with respect to NC Lebesgue measure, $m$. In particular, $\mu$ is \emph{singular} with respect to NC Lebesgue measure in the sense of \cite{JM-ncld,JM-ncFatou} if and only if its GNS row isometry is the direct sum of dilation type and von Neumann type Cuntz row isometries \cite[Corollary 8.13]{JM-ncld}. We say a positive NC measure $\mu \in \posncm$ is of a given type if its GNS row isometry is of that corresponding type. The GNS space of $\mu$ decomposes as the direct sum, 
$$ \hardy (\mu ) = \hardy (\mu _{ac} ) \oplus \hardy (\mu _s ), $$ and $\Pi _{\mu } = \Pi _{\mu _{ac}} \oplus \Pi _{\mu _s }$ with respect to this direct sum. Here,   
$$ \Pi _{\mu _{ac}} = \Pi _{\mu ; L} \oplus \Pi _{\mu ; C-L} \quad \mbox{and} \quad \Pi _{\mu _s} = \Pi _{\mu ; dil} \oplus \Pi _{\mu ; vN}, $$ 
see \cite[Section 8]{JM-ncld}.

A bounded operator $T \in \scr{L} (\hardy )$ is called \emph{left Toeplitz} if $L_j ^* T L_k = \delta _{j,k} I$. Such operators are called multi-Toeplitz in \cite{Pop-entropy}. Here, recall that a bounded operator, $T$, on the Hardy space, $H^2 (\D )$, is called \emph{Toeplitz} if $T = T_f = P_{H^2} M_f | _{H^2}$ for some $f \in L^\infty (\partial \D )$. A result of Brown and Halmos identifies the bounded Toeplitz operators as the set of all bounded opeators $T \in \scr{L} (H^2 )$ with the \emph{Toeplitz property}:
$$ S^* T S = T, $$ where $S = M_z$ is the shift on $H^2$ \cite[Theorem 6]{BrownHalmos}. If $b \in [H^\infty ] _1$, then $T := I - b(S) ^* b(S) \geq 0$ is a positive semi-definite Toeplitz operator, and $b$ is not an extreme point of the closed convex set $[ H^\infty ] _1$ if and only if there is a unique, outer $a \in [ H^\infty ]_1$, the \emph{Sarason function of $b$}, so that $a(0) >0$ and the column $c := \bsm b \\ a \esm$ is inner. A contractive left multiplier of Fock space, $b \in [\mult ]_1$, is said to be \emph{column--extreme (CE)}, if contractivity of the column left multiplier, $c := \bsm b \\ a \esm$, for $a \in \mult$ implies $a \equiv 0$ \cite{JM-freeCE}. In \cite{JM-freeCE} we observed that any CE $b$ is necessarily an extreme point, and that if $b$ is non-CE, then one can define a unique \emph{Sarason function} $a \in [\mult ] _1$ so that $a(0) >0$ and $c := \bsm b \\ a \esm$ is column--extreme. In \cite{JMS-bwshift} we proved that $a$ is outer and that if $b = \fb$ is NC rational and non-CE, then $a = \fa$ is NC rational and the column $\fc := \bsm \fb \\ \fa \esm$ is inner. 

\begin{thm} \label{fincorstructure}
Let $\mu \in \posncm$ be a finitely--correlated Cuntz--Toeplitz functional with NC Lebesgue decomposition $\mu = \mu _{ac} + \mu _s$. The absolutely continuous part of $\mu$, $\mu _{ac} = \mu _L$ is purely of type$-L$ and $\Pi _{\mu _L} \simeq L$. If $\fb \in [\mult ] _1$ is the contractive NC rational left multiplier so that $\mu = \mu _\fb$ is the NC Clark measure of $\fb$ then 
$$ \mu _{ac} (L^\om ) = \ip{1}{(I -\fb (R) ^* ) ^{-1} \fa (R) ^* \fa (R) (I -\fb (R) ) ^{-1} L^\om 1}_{\bH ^2}, $$ where $\fa \in [ \mult ] _1$ is the contractive outer NC rational Sarason function of $\fb$,
$$ T := (I -\fb (R) ^* ) ^{-1} \fa (R) ^* \fa (R) (I -\fb (R) ) ^{-1}, $$ is a bounded left Toeplitz operator and $\fa (1 - \fb ) ^{-1} \in \mult$.

The singular part $\mu _{s} = \mu _{dil}$ is purely of dilation type and $\Pi _{\mu ; s} = \Pi _{\mu ; dil} = \bigoplus _{j=1} ^N \Pi ^{(j)}$ acting on $\hardy (\mu _{dil} ) = \bigoplus _{j=1} ^N \cK _j$ is the direct sum of at most finitely many irreducible Cuntz row isometries of dilation type. If $T^{(j) *} := (\Pi  ^{(j)}) ^* | _{\cH _\mu \cap \cK _j}$ then each $T^{(j)}$ is a finite and irreducible row co-isometry with irreducible and minimal row isometric dilation $\Pi ^{(j)}$. 
\end{thm}
Of course it may be that either $\mu _{ac} =0$ or $\mu _s =0$. Recall that we say that a row isometry, $V$, is irreducible, if the $V_k$ have no jointly reducing subspace. 
\begin{proof}
If $\Pi = \Pi _\mu$ is the GNS row isometry of a finitely--correlated NC measure, $\mu$, then $\Pi _\mu$ is the minimal row isometric dilation of the finite row contraction $T_\mu = (\Pi _\mu ^* | _{\cH _\mu } ) ^*$, and $1 + N_\mu$ is cyclic for $\Pi _\mu$. By Theorem \ref{GNStype}, 
$ \Pi = \Pi _L \oplus \Pi _{dil}$, is the direct sum of a pure row isometry and a Cuntz row isometry of dilation type. By \cite[Section 8, Corollary 8.12, Corollary 8.13]{JM-ncld}, $\Pi _L = \Pi _{ac} $ and $\Pi _{dil} = \Pi _s$ are the GNS row isometries of the absolutely continuous and singular parts of $\mu$, respectively. The fact that $\Pi _\mu$ is cyclic implies that the wandering space of its pure part is at most one dimensional, so that $\Pi _\mu \simeq L \oplus \Pi _{dil}$ where either direct summand may be absent and $\Pi _{dil}$ is cyclic.

The Radon--Nikodym formula for the absolutely continuous (and pure) part of $\mu = \mu _\fb$ in the theorem statement is established in \cite[Theorem 6]{JMS-bwshift}, and is a consequence of an NC rational Fej\'er--Riesz Theorem \cite[Theorem 5]{JMS-bwshift} and the NC Fatou Theorem of \cite{JM-ncFatou}. It further follows from \cite[Theorem 6.5]{DKS-finrow} (see Theorem \ref{DKSdiltype}) that $\Pi _{dil} = \oplus _{k=1} ^N \Pi ^{(k)} _{dil}$ is the direct sum of finitely many irreducible Cuntz row isometries of dilation type. The remaining claim follows from Theorem \ref{DKSdiltype}.
\end{proof}

\begin{cor} \label{singinner}
Let $\mu \in \posncm$ be a finitely--correlated Cuntz--Toeplitz functional. The following are equivalent:
\bn
\item The NC rational multiplier $\fb \in [\mult ] _1$ such that $\mu = \mu _\fb$ is inner.
\item $\Pi _\mu$ is purely Cuntz.
\item $T_\mu$ is a finite row co-isometry.
\item $\Pi _\mu$ is purely of dilation type.
\item $\mu$ is a singular NC measure.
\en
\end{cor}

\begin{remark}
Classically, a contractive multiplier of Hardy space is inner if and only if its Clark measure is singular. In the NC setting we were able to prove one half of this fact in \cite[Corollary 6.29]{JM-ncFatou}. Namely, if $b \in [ \mult ] _1$ is inner, then $\mu _b \in \posncm$ is singular.  By Corollary \ref{singinner}, we see that if $\fb \in [\mult ] _1$ is such that $\mu _\fb$ is a singular finitely--correlated Cuntz--Toeplitz functional, then it is Cuntz (and of dilation type) and $\fb$ is an NC rational inner. Hence the NC analogue of this classical corollary to Fatou's theorem holds, at least for NC rational multipliers.
\end{remark}

\begin{proof}
 By \cite[Theorem 4]{JMS-bwshift}, a contractive NC rational multiplier of Fock space is inner if and only if it is column--extreme. By \cite[Theorem 6.4]{JM-freeCE}, $\Pi _\mu$ is Cuntz if and only if $\fb _\mu$ is CE. By \cite[Proposition 2.5]{Pop-dil}, we know that a finitely--correlated Cuntz--Toeplitz functional $\mu = \mu _\fb$ is such that $\Pi _\mu$ is Cuntz if and only if $T_\mu$ is a finite row co-isometry.  By Theorem \ref{GNStype}, $\Pi _\mu$ is Cuntz if and only if it is a cyclic row isometry purely of dilation type.   By \cite[Corollary 8.13]{JM-ncld}, this happens if and only if $\mu$ is a singular NC measure.
\end{proof}

\begin{remark}
If $\mu \in \posncm$ is any positive NC measure, we can define a positive extension of $\mu$ from the free disk system to the Cuntz--Toeplitz $C^*-$algebra, $\mc{E} _d = C^* \{ I , L_1 , \cdots , L_d \}$, by 
$$ \hat{\mu} ( a_1 ( L ) a_2 (L) ^* ) := \ip{1 + N_\mu}{\pi _\mu (a_1) \pi _\mu (a_2 ) ^* \left(1 + N_\mu\right)}_\mu, $$ where $\pi _\mu : \mc{E} _d \rightarrow \scr{L} (\hardy (\mu ) )$ is the GNS $*-$representation obtained from $\mu$. By well known results in $C^*-$algebra theory, $\hat{\mu}$ will be an extreme point in the state space of $\mc{E} _d$, \emph{i.e.} a \emph{pure state}, if and only if $\pi _\mu$ is irreducible \cite[Theorem I.9.8]{KRD-eg}. Equivalently, $H_\mu$ will be an extreme point in the set of all NC Herglotz functions obeying $H(0) =1$.
\end{remark}

\begin{remark}
Let $\fb$ be NC rational inner. That $\mu _\fb \in \posncm$ is a finitely--correlated Cuntz functional is an analogue of classical theory. Indeed, any rational inner $\fb \in H^2$ is a finite Blaschke product,
$$ \fb (z) = \zeta \, \prod _{k=1} ^N \frac{z-w_k}{1-\ov{w}_k z}; \quad \quad w_k \in \D, \ \zeta \in \partial \D. $$ In this case, the Clark measure, $\mu _\fb$, is a finite positive linear combination of exactly $N$ Dirac point masses. This singular measure, $\mu _\fb$, is supported on the set of points, $\zeta \in \partial \D$, at which $\fb (\zeta ) =1$, so that the point masses are located at the $N$ roots of the degree $N$ polynomial,
$$ \prod _{k=1} ^N (z-w_k) - \prod _{k=1} ^N (1 - \ov{w} _k z). $$ A singular finitely--correlated functional can then be thought of as an analogue of a positive linear combination of finitely many point masses. If the GNS representation of the functional is irreducible, this can be interpreted as the analogue of a single atom. In this case where $d=1$, if $\mu$ is such a finite linear combination of point masses, then $H^2 (\mu ) = L^2 (\mu )$ so that $\Pi _\mu := M_\zeta | _{H^2 (\mu)}$ is unitary, and, 
$$ \cH _\mu := \bigvee _{k\geq 0} M_\zeta ^{*k} 1 = H^2 (\mu ). $$ In this case $\cH _\mu = H^2 (\mu )$ is finite dimensional (of dimension $=N$ if $\mu$ is a linear combination of $N$ point masses) and $\Pi _\mu$ is a finite dimensional unitary. Indeed, since $T_\mu$ is then a finite co-isometry, it must be unitary, so that $T _\mu = \Pi _\mu$ in this single-variable case. In this regard, the theory becomes more complicated when $d>1$ as there are no finite--dimensional row isometries.

Finally, suppose that a contractive rational multiplier $\fb \in [ H^\infty ] _1$ is not an extreme point. Then $1 - |\fb (\zeta ) | ^2 = |\fa (\zeta) | ^2$, $ a.e. \ \partial \D$ where $\fa$ is the rational Sarason function of $\fb$, and it follows that $M_\zeta | _{H^2 (\mu _{\fb ; ac})}$ is a pure cyclic isometry unitarily equivalent to the shift. In particular, $H^2 ( \mu _{\fb } )$ is infinite--dimensional. Recall that there is onto isometry, the \emph{weighted Cauchy transform}, $\scr{F} _\fb : H^2 (\mu _\fb ) \rightarrow \scr{H} (\fb )$ of $H^2 (\mu _b )$ onto the de Branges--Rovnyak space of $\fb$, and that the image of $\Pi _\fb ^*$, where $\Pi _\fb := M_\zeta | _{H^2 (\mu _\fb )}$, under the weighted Cauchy transform is a rank--one perturbation of the restricted backward shift,
$$ X(1)  := \underbrace{S^* | _{\scr{H} (\fb )}}_{=: X} + \ip{K_0 ^\fb}{\cdot}_{\scr{H} (\fb )} S^* \fb, $$ $X(1) = \scr{F} _\fb \Pi _\fb ^* \scr{F} _\fb ^*$ \cite{Clark}. In this case, since $\fb$ is rational, we obtain that $$ \cH _\mu := \bigvee \Pi _\fb ^{*k} 1 \simeq \bigvee X(1) ^{k} K_0 ^\fb =: \scr{M} (\fb ) \subsetneqq \scr{H} (\fb ), $$ is a finite--dimensional subspace of $\scr{H} (\fb )$.
 Hence, in this case if $P$ is the projection onto $\scr{M} (\fb )$, then 
$$ T (1) :=  \left( X(1)  | _{\scr{M} (\fb ) } \right) ^*, $$ is a finite dimensional contraction, and $X(1) ^* \simeq \Pi _\fb = M _\zeta | _{H^2 (\mu _\fb )}$ is its minimal isometric dilation. This again shows that our results are natural extensions of classical theory.
\end{remark}

\begin{remark}
Let $\mu \in \posncm$ be a state, \emph{i.e.} a positive NC measure such that $\mu(I) = 1$. Let $b$ be the associated contractive NC function so that $H_\mu = H_b$. Let $b_n \in \fp$ be the $n$th Ces\`aro sum of $b$. By \cite{DP-alg}, $\|b_n(L)\| \leq \|b (L) \| \leq 1$. Moreover $b_n(L) \stackrel{SOT-*}{\longrightarrow} b(L)$. Since $\bH^2_d$ is an NC-RKHS, we obtain that $b_n$ converges uniformly to $b$ on sub-balls. This implies that the inverse Cayley transforms of $b_n$, $H_{b_n}$, converge uniformly on sub-balls to the NC Herglotz function $H_{\mu}$. Since the Taylor--Taylor coefficients of $H_{\mu}$ are essentially the moments of $\mu$, the states $\mu_n := \mu _{b_n}$ converge weak$-*$ to $\mu$ (these are states since $H_{\mu}(0) = 1$ and thus $b(0) = 0 = b_n (0)$ and conversely). In other words, the finitely--correlated Cuntz--Toeplitz functionals are weak$-*$ dense in $\posncm$. This is consistent with our interpretation of finitely--correlated NC measures as NC analogues of finite positive sums of point masses. 


\end{remark}

\section{Minimal realizations of rational multipliers} \label{minreal}

The results of the previous section show that any NC rational multiplier of Fock space is determined by a positive and finitely--correlated NC Clark measure. Moreover, any such NC rational Clark measure can be constructed from a finite--dimensional row contraction, $T$ on $\cH$, and a vector $x \in \cH$ which is $T^*-$cyclic and $V-$cyclic, where $V$ is the minimal row isometric dilation of $T$. Namely, if $\mu = \mu _\fb$ is the finitely--correlated NC Clark measure of a contractive NC rational multiplier $\fb \in [\mult ] _1$, then $\mu (L^\om ) = \mu _{T,x} (L^\om ) = \ip{x}{T^\om x}_{\cH}$ where the pair $(T,x)$ has the above properties. Given such a finite row contraction $T$ and vector $x$ so that $\mu = \mu _\fb = \mu _{T,x}$, our goal now is to express the minimal Fornasini--Marchesini realization of $\fb$ soley in terms of $T$ and $x$. We will accomplish this by determining the relationship between $T,x$ and the minimal de Branges--Rovnyak FM realization of $\fb$ as described in Section \ref{ss:minFM}.

Assume that $T$ is a finite row contraction on $\cH$, $V$ is its minimal row-isometric dilation on $\cK \supsetneqq \cH$, and $x \in \cH$ is cyclic for $T^*$, $V$ and hence $T$. As in the previous section we define the positive NC measure $\mu := \mu _{T,x}$, $\mu _{T,x} (L ^\om) = \ip{x}{T^\om x} _{\cH}$. By Proposition \ref{confincor}, $T\simeq T_\mu$ and $V \simeq \Pi _\mu$ via the unitary $U_x \left(\mf{z}^\om + N_\mu\right) = V^\om x$. Recall that the right free Cauchy transform is an isometric map from $\hardy (\mu )$ onto the right free Herglotz space $\scr{H} ^+ ( \mf{H} _\mu )$, where $\mf{H} _\mu$ is the NC rational Herglotz--Riesz transform of the finitely--correlated NC measure, $\mu$. If $\mu = \mu _\fb$ is the NC Clark measure of an NC rational $\fb \in [ \mult ] _1$, $\mf{H} _\fb (Z) = \mf{H} _\mu (Z) + it I_n$ where $t:= \nbim \mf{H} _\fb (0) \in \R$, and $\mf{H} _\fb (Z) = (I_n + \fb (Z) ) (I_n - \fb (Z) ) ^{-1}$.  Observe that the CPNC kernel for the Herglotz space $\scr{H} ^+ (\mf{H} _\mu )$ is:
\ba K^\mu (Z,W) [\cdot ]& = & \frac{1}{2} K(Z,W) [ \mf{H} _\mu ^\mrt (Z) (\cdot )] + \frac{1}{2} K(Z, W)  [ (\cdot) \mf{H}  _\mu ^\mrt (W) ^*] \nn \\
& = & \frac{1}{2} K(Z,W) [ (\mf{H} _\mu ^\mrt (Z) +it I_n) (\cdot )] + \frac{1}{2} K(Z,W) [ (\cdot) (\mf{H} _\mu ^\mrt (W) ^* -it I_m )] \nn \\
& = & \frac{1}{2} K(Z,W) [ \mf{H} _\fb ^\mrt (Z) (\cdot )] + \frac{1}{2} K(Z,W) [ (\cdot) \mf{H} _\fb ^\mrt (W) ^* ]. \nn \ea 
That is, the Herglotz space of any two NC Herglotz functions which differ by an imaginary constant is the same. Hence, if we define 
$$ \fb _\mu (Z) := (\mf{H} _\mu (Z) - I_n) (\mf{H} _\mu (Z) + I_n ) ^{-1}, $$ then as described in the background section, $\fb = \frac{\ov{z (t)}}{z(t)} \cdot \la _{z(t)} \circ \fb _\mu$,
where $\la _{z(t)}$ is the M\"obius transformation
$$ \la _{z(t)} = \frac{z - z(t)}{1-\ov{z(t)}z} \quad \mbox{and} \quad z(t) = \frac{t}{2i +t} \in \D. $$ Moreover, we have that both
\ba K^\mu (Z,W) [P] & = & K(Z,W)\left[ (I - \fb _\mu  ^\mrt (Z) ) ^{-1} \left( P - \fb _\mu ^\mrt(Z) P \fb _\mu ^\mrt (W) ^*    \right) (I - \fb _\mu ^\mrt (W) ^* ) ^{-1}  \right] \quad \mbox{and} \nn \\
& = & K(Z,W)\left[ (I - \fb  ^\mrt (Z) ) ^{-1} \left( P - \fb  ^\mrt(Z) P \fb  ^\mrt (W) ^*    \right) (I - \fb  ^\mrt (W) ^* ) ^{-1} \right]  \nn \ea for any $Z \in \B ^d _n$, $W \in \B ^d _m$ and $P \in \C ^{n \times m}$. This identity shows that $M^R _{(I -\fb _\mu ^\mrt) }$ is an isometric right multiplier of $\scr{H} ^+ ( \mf{H} _\mu ) = \scr{H} ^+ ( \mf{H} _\fb )$ onto the right free de Branges--Rovnyak space $\scr{H} ^\mrt (\fb _\mu )$ and that $M^R _{(I-\fb ^\mrt)}$ is an isometric right multiplier of $\scr{H} ^+ (\mf{H}  _\mu )$ onto $\scr{H} ^\mrt (\fb )$. The \emph{weighted free Cauchy transform}, $\scr{F} _\fb := M^R _{I_n - \fb ^\mrt (Z)} \circ \scr{C} _\mu : \hardy (\mu ) \rightarrow \scr{H} ^\mrt (\fb  )$ is then an onto isometry \cite{JM-freeAC,JM-freeCE}, and $\scr{U} _x := \scr{F} _\fb \circ U_x ^* : \cK \rightarrow \scr{H} ^\mrt (\fb  )$ will be an onto isometry. Further recall from \cite{JM-freeAC,JM-freeCE}, that the weighted Cauchy transform intertwines the adjoint of the GNS row isometry, $\Pi _\mu$, with a rank--one (co-isometric) Clark perturbation, $X(1)$, of the restricted backward shift $X := L^* | _{\scr{H} ^\mrt (\fb )}$:
$$ X(1) _k := \underbrace{X_k}_{=L_k^*| _{\scr{H} ^\mrt (\fb )}} + \frac{1}{1-\fb (0)} L_k^* \fb ^\mrt \ip{K_0 ^\fb}{\cdot}_{\fb}, \quad \quad 1 \leq k \leq d.$$ Hence $\scr{U} _x V ^* _k = X(1) _k \scr{U} _x$. In the above, $K_0 ^\fb = K ^\fb \{ 0 , 1 ,1 \}$, is the point evaluation vector at the point $0 \in \B ^d _1$ for $\scr{H} ^\mrt (\fb )$. Observe that  
\ba X(1) _k K_0 ^\fb & = & -L_k ^* \fb ^\mrt \ov{\fb (0)} + L_k ^* \fb^{\mrt} \frac{1 - | \fb (0) | ^2}{1 - \fb (0)} \nn \\
& = & L^* _k \fb ^\mrt  \, \frac{-\ov{\fb (0)} + |\fb (0) | ^2 + 1 - |\fb (0) | ^2 }{1 - \fb (0) }  \nn \\
& = & \frac{1 - \ov{\fb (0)}}{1- \fb (0)} \,  L_k ^* \fbt. \label{Xoneonone} \ea 
That is, we can write:
\be X_k = L_k ^* | _{\scr{H} ^\mrt (\fb )} = X(1) _k \left( I -  \frac{1}{1-\ov{\fb (0)} } K_0 ^\fb \ip{K_0 ^\fb}{\cdot} \right). \label{Clarkform} \ee 
Here, note that 
$$ \fb (0) = \frac{ \mf{H} _\fb (0) -1}{\mf{H} _\fb (0) +1} = \frac{ \| x \| ^2 +it -1}{\|x \| ^2 +it +1} \in \D. $$ 
If we define,
$$ \scr{M} (\fb ) := \bigvee K_0 ^\fb + \bigvee _{\om \neq \emptyset} L^{*\om} \fbt, $$ then by Equation (\ref{Xoneonone}), 
$$ \scr{M} (\fb ) = \bigvee  X(1) ^\om K_0 ^\fb, $$ is both $X$ and $X(1)-$invariant. Let $T(1) ^* := X(1) | _{\scr{M} (\fb )}$ and $T(0) ^* := X | _{\scr{M}  (\fb )}$. Observe that the image of $x \in \cH$ under $\scr{U} _x$ is:
\ba \scr{U} _x x & = &  \left( M^R _{(I - \fbt (Z) ) ^{-1}} \right) ^* \scr{C} _\mu 1 + N_\mu \nn \\
& = & \left( M^R _{(I - \fbt (Z) ) ^{-1}} \right) ^* K_0 ^\mu \nn \\
& = & \frac{1}{1 - \ov{\fb (0)} } K_0 ^\fb. \label{freeCTofone} \ea It follows that $\scr{M}  (\fb ) = \scr{U} _x \cH $ and that $T(1)^*$ is unitarily equivalent via $\scr{U} _x$ to $T ^* = V ^* | _{\cH}$.
It then follows from Equations (\ref{Clarkform}) and (\ref{freeCTofone})  that $T (0) ^* _k$ is unitarily equivalent to 
\be T_k ^* \left( I - (1 - \fb (0)) x \ip{x}{\cdot}_{\cH} \right). \ee
If we define, as in Section \ref{ss:minFM}, 
$$ \scr{M} _0 (\fb ) := \bigvee _{\om \neq \emptyset} L^{*\om} \fb ^\mrt, $$ with projector $Q_0$ then $\scr{M} _0 (\fb)$ is $T(0) ^*$ and $T(1)^*-$invariant, and if we set $A := T(0) ^* | _{\scr{M} _0 (\fb )}$,  then the minimal de Branges--Rovnyak FM realization is given by $(A,B,C,D)$ where 
$$ A = L^* | _{\scr{M} _0 (\fb )} = T(0) ^* | _{\scr{M} _0 (\fb )}, \quad B= L^* \fbt = \frac{1 - \fb (0) }{1-\ov{\fb (0)}} T(1) ^* K_0 ^\fb, \quad C = \ip{Q_0 K_0 ^\fb}{\cdot}_{\fb} \quad \mbox{and} \quad D= \fb (0). $$ 
It follows that $\scr{M} _0 (\fb ) = \scr{U} _x \cH _0$ where 
$$ \cH _0 := \bigvee _{\om \neq \emptyset} T^{*\om} x, $$ with projector $P_0$. A minimal FM realization of $\fb$, where $\mu _\fb = \mu _{T,x}$, is then $(\hat{A} , \hat{B} , \hat{C} , \hat{D} )$ where
$$ \hat{A} := \left. T^* \left( I - (1 - \fb (0 ) ) x \ip{ x}{\cdot}_{\cH} \right) \right| _{\cH _0}, $$ 
$$ \hat{B} _k := (1 - \fb (0) ) T^* _k x, \quad \hat{C} := (1 - \fb (0) ) \ip{P_0 x}{\cdot}_{\cH _0} $$ and 
$$ \hat{D} = \fb (0) = \frac{ \| x \| ^2 +it -1 }{\| x \| ^2 +it +1}. $$ 

In summary, we have proven the following theorem:

\begin{thm} \label{minratreal}
An NC rational function, $\fb$, with $0 \in \nbdom \fb$ belongs to $[ \mult ] _1$ if and only if it is realized as the transfer--function of the minimal and finite--dimensional Fornasini--Marchesini colligation: 
$$ U_{T,x} := \bpm T_0 ^* & (1 - \fb (0) ) T^* x \\ (1 -\fb (0) ) \ip{P_{\cH _0} x}{\cdot}_{\cH _0} &  \fb (0)  \epm, \quad \quad \fb (0) = \frac{\| x \| ^2 +it - 1 }{ \| x \| ^2 +it +1 }, \ t \in \R, $$ where $T$ is a finite--dimensional row contraction on $\cH$, $x \in \cH$ is cyclic for both $T^*$ and the minimal row isometric dilation, $V$, of $T$, 
$$ T_0 ^* := \left. T^* \left( I - (1 - \fb (0) ) \ip{x}{\cdot} x \right) \right| _{\cH _0}  \quad \mbox{and} \quad \cH _0 := \bigvee _{\om \neq \emptyset} T^{*\om} x.$$
Moreover, $\fb$ is inner if and only if $T$ is also a row co-isometry.
\end{thm}
If $T$ is an irreducible finite row-coisometry on $\cH$, Lemma \ref{irredlemma} implies that given any $x \in \cH$, the pair $T,x$ satisfies the conditions of Theorem \ref{minratreal} and hence generates an NC rational inner. Further recall that Lemma \ref{Tcyclic2} implies that if $T$ is a finite row co-isometry, then $x \in \cH$ is $T-$cyclic if and only if it is $V-$cyclic. 
\begin{remark}
Observe that $\fb  (0) =0$ if and only if $\mf{H} _\fb = \mf{H} _{\mu _\fb}$ and $\mu (I) = \| x \| ^2 =1$. Also note that 
$$ \| x \| ^2 + it = \mf{H} _\fb (0) = \frac{1 + \fb (0)}{1 - \fb (0) }. $$ 
\end{remark}
\begin{eg}
Let $q : \C ^{n\times n } \rightarrow \C ^{n\times n}$ be a completely positive and completely contractive linear map. By Choi's Theorem \cite{Choi}, $q$ has Kraus operators or quantum effects, $q_k \in \C ^{n \times n}$, $1\leq k \leq d$ (for some $d \in \N$) so that 
$$ q (A) := \sum _{j=1} ^d q_j A q_j ^*; \quad \quad A \in \C ^{n\times n}. $$ 
Since $q$ is completely contractive, the $d-$tuple $Q:= \left( q_1 , \cdots , q_d \right)$ is a row contraction, and $Q$ will be a row co-isometry if and only if $q$ is unital. Provided there is a vector $x \in \C ^n$ which is cyclic for both $Q^*$ and the minimal row isometric dilation of $Q$, we can associate a unique contractive NC rational multiplier, $\mf{q}$, to the CP map $q$.
\end{eg}
\begin{eg}
Consider the row co-isometry $T : \C ^2 \otimes \C ^2 \rightarrow \C ^2$ defined by 
$$ T_1 := \bpm 0 & 1 \\ 0 & 0 \epm, \quad \mbox{and} \quad T_2 := \bpm 0 & 0 \\ 1 & 0 \epm. $$ 
It is not difficult to check that this is irreducible. Hence we can choose any non-zero $x \in \C ^2$, and apply Theorem \ref{minratreal} to construct an NC rational inner function.

First consider $x = \bsm 1 \\ 0 \esm$. Since $\| x \| =1$, we will have $\fr (0) =0$. Then, 
\ba   I_2 - x x^*  =   \bpm 0 & 0 \\ 0 & 1 \epm,  \quad \quad   & \mbox{so that} & \nn \\
  T_{0;1} ^* =  \bpm 0 & 0 \\ 1 & 0 \epm  \bpm 0 & 0 \\ 0 & 1 \epm = 0 _2  \quad \quad    & \mbox{and} & \quad \quad
 T_{0;2} ^* = \bpm 0 & 1 \\ 0 & 0 \epm \bpm 0 & 0 \\ 0 & 1 \epm = \bpm 0 & 1 \\ 0 & 0 \epm.  \nn \ea
Then,
\begin{align*}  &   I - Z \otimes T_0 ^*  =  \bpm I & - Z _2 \\ 0 & I \epm, \quad \quad \left( I - Z \otimes T_0 ^* \right) ^{-1} = \bpm I & Z_2 \\ 0 & I \epm  \\
\mbox{and} \quad \quad  & Z \otimes T^* I \otimes x  =  \bpm 0 \\ Z_1  \epm,  \\
\mbox{so that} \quad \quad & \fb _{T,x} (Z)  =  \bpm I, & 0 \epm  \bpm I & Z_2 \\ 0 & I \epm \bpm 0 \\ Z_1 \epm  \\
&  =  Z_2 Z_1. \end{align*} This is clearly inner. Note that $T$ is an irreducible row co-isometry, and yet $\fb (Z) = Z_1 Z_2$ is reducible as an inner left multiplier of the Fock space. That is, $\fb$ is the product of two NC rational inner left multipliers. For this $\fb$,
$$ T_{0;1} ^* = 0 _2 \quad \mbox{and} \quad T_{0;2} ^* = \bpm 0 & 1 \\ 0 & 0 \epm, $$ so that $T_0 ^*$ is a reducible $2-$tuple. 


Similarly, taking $x = \bsm 0 \\ 1 \esm$ gives: 
$$ T^* _{0; 1} = \bpm 0 & 0 \\ 1 & 0 \epm, \quad T^* _{0;2} = 0 _2,$$ $$ (I - Z \otimes T_0 ^* ) ^{-1} = \bpm I & 0 \\ Z_1 & I \epm, $$ and $$ Z\otimes T^* (I_n \otimes x ) = \bpm Z_2 \\ 0_n \epm. $$ Putting this together gives, for $Z \in \B ^2 _n$,
\ba \fb (Z) & = & \underbrace{\fb (0)}_{\equiv 0} I_n + I_n \otimes x ^* \left( I -Z \otimes T_0 ^* \right) ^{-1} Z\otimes T^* I_n \otimes x \nn \\
& = & (0_n , I_n ) \bpm I_n & 0 _n \\ Z_1 & I_n \epm \bpm Z_2 \\ 0_n \epm \nn \\
& = & Z_1 Z_2. \nn \ea Again, this is inner. To obtain a less trivial example, take $x = \frac{1}{\sqrt{2}} \bsm 1 \\ 1 \esm$.  In this case,
$$ I_2 - x x^*  = \frac{1}{2} \bpm 1 & -1 \\ -1 & 1 \epm, $$
$$ T_{0;1} ^* = \frac{1}{2} \bpm 0 & 0 \\ 1 & 0 \epm  \bpm 1 & -1 \\ -1 & 1 \epm = \frac{1}{2} \bpm 0 & 0 \\ 1 & -1 \epm $$ and 
$$ T_{0;2} ^* = \frac{1}{2} \bpm -1 & 1 \\ 0 & 0 \epm. $$ Then, 
$$ I - Z \otimes T_0 ^* = \frac{1}{2} \bpm 2I +  Z_2 & -Z_2 \\ -Z_1 & 2I  +Z_1\epm. $$ The inverse can be computed using Schur complements. If 
$$ S := I + \frac{Z_2}{2} - \frac{1}{4} Z_2 \left(I + \frac{Z_1}{2} \right) ^{-1} Z_1, $$ is the Schur complement of the upper left block, then
$$ \left( I - Z \otimes T_0 ^* \right) ^{-1} = \bpm S^{-1} & S^{-1} \frac{Z_2}{2} \left( I + \frac{Z_1}{2} \right) ^{-1} \\ \left( I +\frac{Z_1}{2} \right) ^{-1} \frac{Z_1}{2} S^{-1} & \quad \left( I+\frac{Z_1}{2} \right) ^{-1}  + \frac{1}{4} \left(I + \frac{Z_1}{2} \right) ^{-1} Z_1 S^{-1} Z_2  \left(I + \frac{Z_1}{2} \right) ^{-1} \epm. $$ 
Finally,
$$ Z \otimes T^* x =  \frac{1}{\sqrt{2}} \bpm Z_2 \\ Z_1  \epm. $$ Hence, 
\ba \fb _{T,x} (Z) & = &  \frac{1}{2} \bpm I, & I \epm  \left( I -Z \otimes T_0 ^* \right) ^{-1} \bpm Z_2 \\ Z_1 \epm \nn \\ 
& = & \frac{1}{2} S^{-1} Z_2  + \left( I +\frac{Z_1}{2} \right) ^{-1} \frac{Z_1}{4} S^{-1}  Z_2 +
S^{-1} \frac{Z_2}{4} \left( I + \frac{Z_1}{2} \right) ^{-1}Z_1 + \frac{1}{2} \left( I+\frac{Z_1}{2} \right) ^{-1} Z_1 \nn \\
& & + \frac{1}{8} \left(I + \frac{Z_1}{2} \right) ^{-1} Z_1 S^{-1} Z_2  \left(I + \frac{Z_1}{2} \right) ^{-1} Z_1.\nn \ea This must be an NC rational inner.
\end{eg}

\begin{eg}
Consider an irreducible point arising from anti-commuting unitaries
\[
T = \frac{1}{\sqrt{2}} \left( \begin{pmatrix} 1 & 0 \\ 0 & -1 \end{pmatrix},\, \begin{pmatrix} 0 & -1 \\ 1 & 0 \end{pmatrix} \right).
\]
Let $x = \bsm \alpha \\ \beta \esm$ and assume that $|\alpha|^2 + |\beta|^2 = 1$. Then 
\[
I_2 - x x^* = \begin{pmatrix} |\beta|^2 & -\alpha \bar{\beta} \\ - \bar{\alpha} \beta & |\alpha|^2 \end{pmatrix}.
\]
Therefore,
\[
I - Z \otimes T_0^* = I - \frac{1}{\sqrt{2}} \begin{pmatrix} \beta( \bar{\beta} Z_1 + \bar{\alpha} Z_2) & -\alpha (\bar{\beta} Z_1 + \bar{\alpha} Z_2) \\ \beta( \bar{\alpha} Z_1  + \bar{\beta} Z_2) & -\alpha (\bar{\alpha} Z_1 + \bar{\beta} Z_2) \end{pmatrix}
\]
If we set $\alpha = \beta = \frac{1}{\sqrt{2}}$, we obtain
\[
I - Z \otimes T_0^* = I - \frac{1}{2 \sqrt{2}} (Z_1 + Z_2) \otimes \begin{pmatrix} 1 & - 1 \\ 1 & -1 \end{pmatrix}.
\]
Hence, 
\ba
\fb _{T,x}(Z) * & = & \frac{1}{2 \sqrt{2}} \begin{pmatrix} I & I \end{pmatrix} \left( I + \frac{1}{2 \sqrt{2}} (Z_1 + Z_2) \otimes \begin{pmatrix} 1 & - 1 \\ 1 & -1 \end{pmatrix} \right) \left((Z_1 - Z_2) \otimes \begin{pmatrix} 1 \\ -1 \end{pmatrix}\right) \nn \\  & = &  \frac{1}{2} (Z_1 + Z_2) (Z_1 - Z_2), \nn \ea
which is obviously inner. On the other hand, if we set $\alpha = 0$ and $\beta = 1$, then
\[
\fb _{T,x}(Z) = \frac{1}{\sqrt{2}} \begin{pmatrix} 0 & I \end{pmatrix} \begin{pmatrix} \left( I - \frac{1}{\sqrt{2}} Z_1 \right)^{-1} & 0 \\ \frac{1}{\sqrt{2}} Z_2 \left( I - \frac{1}{\sqrt{2}} Z_1 \right) ^{-1} & I  \end{pmatrix} \begin{pmatrix} -Z_2 \\ -Z_1 \end{pmatrix} = - \frac{1}{2} Z_2 \left( I - \frac{1}{\sqrt{2}} Z_1 \right)^{-1} Z_2 - \frac{1}{\sqrt{2}} Z_1 .
\] One can verify directly that this $\fb _{T,x} (Z)$ is also inner: 
\ba \fb _{T,x} (L ) ^* \fb _{T,x} (L) & = & \frac{1}{2} I + \frac{1}{4} L_2 ^* \sum _{j,k =0} ^\infty \frac{1}{\sqrt{2} ^{k+j}} L_1 ^{*j} L_2 ^* L_2 L_1 ^k \, L_2 \nn \\
& = & \frac{1}{2} I + I \, \frac{1}{4} \sum _{j=0} ^\infty \frac{1}{2^j} = I. \nn \ea 
\end{eg}

\section{Rank--one Clark--Cuntz peturbations}

Let $\fb \in [ \mult ] _1$ be a contractive NC rational multiplier and choose any $\zeta \in \partial \D$. Consider the one-parameter family of NC Clark measures $\mu _\zeta := \mu _{\fb \ov{\zeta}}$ indexed by the unit circle. Every $\mu _\zeta$ is a finitely--correlated Cuntz--Toeplitz functional, so that if $\Pi (\zeta) := \Pi _{\mu _\zeta}$ then, 
$$ \Pi (\zeta) = \bigoplus _{n=1} ^{N_\zeta} \Pi _\zeta ^{(n)}, $$ is a direct sum of finitely many irreducible representations by Theorem \ref{fincorstructure}. Let $\scr{F} _\zeta : \hardy (\mu _\zeta ) \rightarrow \scr{H} ^\mrt (\fb )$ be the onto, isometric, NC weighted Cauchy transform onto the right free de Branges--Rovnyak space of $\fb$ \cite{JM-freeCE,JM-freeAC}. Here, given any positive NC Clark measure $\mu = \mu _b \in \posncm$ for $b \in [\mult ] _1$, the weighted (right) free Cauchy transform $\scr{F} _\mu : \hardy (\mu ) \rightarrow \scr{H} ^\mrt (b)$ is the onto isometry defined by 
$$ \scr{F} _\mu := M^R _{1 - b^\mrt} \circ \scr{C} _\mu, $$ where $\scr{C} _\mu : \hardy (\mu ) \rightarrow \scr{H} ^+ (H _\mu)$ is the free Cauchy transform as described following the statement of Theorem \ref{fincorthm}. Recall that the image of $\Pi (\zeta) ^*$ under this unitary transformation is the rank--one co-isometric Clark pertubation of $X := L^* | _{\scr{H} ^\mrt (\fb )}$,
\be X (\zeta ) = X + \frac{\ov{\zeta}}{ 1 - \fb (0) \ov{\zeta} }  \ip{K_0 ^\fb}{\cdot} _{\fb} \, \vfb;  \label{CCuntzfam} \ee where $\vfb := L^* \fbt$ and $\ip{\cdot}{\cdot}_\fb := \ip{\cdot}{\cdot}_{\scr{H} ^\mrt (\fb )}$, see \cite[Theorem 6.3]{JM-freeAC} and \cite[Section 6.17]{JM-freeCE}. More generally, we will also consider the Clark perturbations $X(\zeta) $ for any $\zeta \in \C$ defined by the above formula. In particular, $X = X(0)$. 

\begin{prop}
Either every $\Pi _{\mu _{\ov{\zeta} \fb}} = : \Pi (\zeta ) = \Pi (\zeta ) _L \oplus \Pi (\zeta) _{dil}$, $\zeta \in \partial \D$, has a non-zero direct summand of type$-L$ or every $\Pi (\zeta)$ has no type$-L$ direct summand.
\end{prop}

\begin{proof}
If $\Pi (\zeta ) = \Pi (\zeta ) _{L} \oplus \Pi (\zeta ) _{dil}$ has a non-zero pure type$-L$ direct summand, then since $\Pi (\zeta )$ has a cyclic vector, $\Pi (\zeta ) _L \simeq L$ is unitarily equivalent to exactly one copy of $L$. If $\Pi (\zeta )$ has non-zero type$-L$ direct summand and $\Pi ( \xi )$ does not, for some $\xi \neq \zeta$, $\xi ,\zeta \in \partial \D$, then $\Pi (\xi )$ is a Cuntz row isometry purely of dilation--type. However by \cite[Theorem 6.4]{JM-freeCE}, this would imply that $\ov{\xi} \fb ^\mrt \notin \scr{H} ^\mrt ( \fb )$. Since $\Pi ( \zeta )$ has non-zero type$-L$ summand, the same result would imply that $\ov{\zeta } \fb ^\mrt \in \scr{H} ^\mrt (\fb )$. This contradiction shows that either all Clark perturbations have non-zero pure type$-L$ direct summands of multiplicity one or none do. 
\end{proof}

Given any NC rational $\fb \in [ \mult ] _1$, recall that we can define the finite--dimensional subspaces
\be \scr{M} _0 (\fb ) = \bigvee _{\om \neq \emptyset} L^ {*\om} \fbt \quad \quad  \mbox{and}  \quad \quad \scr{M} (\fb ) = \scr{M} _0 (\fb ) + \bigvee \{ K_0 ^\fb \} \subseteq \scr{H} ^\mrt (\fb ). \label{finitespaces} \ee 

\begin{lemma} \label{samespace}
For any $\zeta \in \C$, $\scr{M} (\fb )$ and $\scr{M} _0 (\fb )$ are $X (\zeta )-$invariant.
\end{lemma}
\begin{proof}
This follows immediately from the formulas (\ref{finitespaces}) and (\ref{CCuntzfam}). 
\end{proof}
For any $\zeta \in \C $, let 
\be T(\zeta) ^* := X(\zeta) | _{\scr{M} (\fb )}. \label{Tzetadef} \ee It follows that for any $\zeta \in \partial \D$, $X(\zeta ) ^*$ is the minimal row isometric dilation of $T(\zeta ) $ and $X(\zeta ) ^*$ is a cyclic Cuntz row isometry if and only if $T(\zeta) $ is a row co-isometry.

It will be convenient to assume that $\fb (0) = 0$. There is no loss in generality in making this assumption, as if $w :=\fb (0) \neq 0$, we can apply the M\"obius transformation 
$$ \la _{w}  (z) = \frac{z - w }{1 - \ov{w} z}, $$ to $\fb$ to obtain a new contractive NC rational multiplier, $\fb _0$, so that $\fb _0 (0) = 0$. Moreover, the composition of an isometry with a M\"obius transformation is again an isometry, so that $\fb$ is inner if and only if $\fb _0$ is. As in \cite[Proposition 6.6]{JMS-ncBSO}, right multiplication by
$$ C ^\mrt _w (Z) := \sqrt{1 - |w|^2} (I + \ov{w} \fbt _0 (Z) ) ^{-1} = \sqrt{1 -|w| ^2 } ^{-1} (I - \ov{w} \fbt (Z) ); \quad w = \fb (0), $$ is an isometric right multiplier from $\scr{H} ^\mrt ( \fb _0  )$ onto $\scr{H} ^\mrt (\fb )$. We will denote this right multiplier by $\mf{C} _w = M^R _{C^\mrt _w}$, and this is the NC analogue of a \emph{Crofoot transformation} \cite{Crofoot}. As in the commutative setting of \cite[Theorem 5.7, Proposition 5.8]{MMdBR}, one can verify that 
\be \frac{1}{\sqrt{1-|w| ^2}} \mf{C} _w ^* \otimes I_d \vfb = \vfb _0, \label{Crofoot1} \ee where $\vfb := L^* \fbt \in \scr{H} ^\mrt (\fb ) \otimes \C ^d$ and that if $X := L^* | _{\scr{H} ^\mrt (\fb )}$, $X^{(0)} := L^* | _{\scr{H} ^\mrt (\fb _0 )}$, then 
\be \mf{C} _w \otimes I_d X^{(0)} \mf{C} _w ^*  = X + \frac{\ov{w}}{1-|w| ^2} \vfb  \ip{K_0 ^{\fb}}{\cdot}_\fb. \label{Crofoot2} \ee 
\begin{prop}
Any Clark--Cuntz perturbation, $X (\zeta)$, $\zeta \in \partial \D$, of $\fb $ is unitarily equivalent to the corresponding Clark--Cuntz perturbation, $X ^{(0)} (\zeta )$, of $\fb _0$ via the NC Crofoot transformation. 
\end{prop}
\begin{proof}
Given $\zeta \in \partial \D$ and $w = \fb (0)$ we apply Equation (\ref{Crofoot1}) and Equation (\ref{Crofoot2}) to obtain
\ba \mf{C} _w \otimes I_d X^{(0)} (\zeta ) \mf{C} _w ^* & = & \mf{C} _w \otimes I_d X^{(0)} \mf{C} _w ^* + \ov{\zeta}  \mf{C} _w \otimes I_d  \vfb _0 \ip{\mf{C}_w 1}{\cdot}_\fb \nn \\
& = &  X + \frac{\ov{w}}{1-|w| ^2} \vfb  \ip{K_0 ^{\fb }}{\cdot}_\fb +\ov{\zeta} \sqrt{1 -|w| ^2} \vfb \mf{C} ^{-1} (0) \ip{K_0 ^\fb}{\cdot}_\fb \nn \\
& = & X + \frac{\ov{w} +\ov{\zeta} }{1-|w| ^2} \vfb \ip{K_0 ^{\fb }}{\cdot}_\fb \nn \\
& = & X + \frac{\ov{\zeta}}{1 - \ov{\zeta} w} \vfb \ip{K_0 ^{\fb }}{\cdot}_\fb = X (\zeta). \nn \ea 
\end{proof}

By the above proposition, we can and will assume, without loss in generality, that $\fb (0) =0$ for the remainder of this section.
Note that if $\fb (0) =0$, then $K_0 ^\fb = 1 - \fbt \ov{\fb (0)} = 1$, and our formulas for the finite--dimensional Clark perturbations $T(\zeta) ^*$ simplify:
\be T(\zeta ) ^* = \left.  T(0) ^* + \ov{\zeta} L^* \fbt \ip{1}{\cdot}_\fb  \right| _{\scr{M} (\fb) }, \quad \quad \zeta \in \C. \label{bzerofinClark} \ee 

\subsection{Boundary values}

Let $\fb \in [ \mult ] _1$ be NC rational. We further assume, without loss in generality, that $\fb (0) =0$.

\begin{lemma} \label{sprisone}
For any $\zeta \in \partial \D$, the $d-$tuple $T (\zeta ) ^\mrt$ has joint spectral radius $\mr{spr} \, T (\zeta ) ^\mrt \leq 1$.
\end{lemma}
In the above statement, $T (\zeta ) ^\mrt = \left( T (\zeta ) _{ 1} ^\mrt, \cdots T (\zeta ) _{d } ^\mrt \right)$, where $\mrt$ denotes matrix transpose of each matrix entry of $T(\zeta)$ with respect to a choice of orthonormal basis of $\scr{M} (\fb )$.
\begin{proof}
Without any loss in generality, assume $\zeta =1$. The $d-$tuple, $T:= T(1)$, is a finite--dimensional row contraction. If we identify $T$ with a row contraction acting on $\C ^m \otimes \C ^d$, let $\ov{T} = ( \ov{T} _1 , \cdots , \ov{T} _d )$ denote the entry--wise complex conjugation of the matrix of each $T_j$ with respect to the standard orthonormal basis, so that $\ov{T_k} = (T_k ^\mrt) ^*$. Since $T$ is a row contraction, $T T^* \leq I$. Taking transposes shows that $\ov{T} T ^\mrt \leq I$, so that $\ov{T}$ is also a row contraction, $\| \ov{T} \| _{row} \leq 1$ and so $\| T ^\mrt \| _{col} \leq 1$. By Lemma \ref{rowvscol}, we have that for any $\eps >0$, $T ^\mrt$ is jointly similar to a row $d-$tuple, $W  \in \C ^d _m$, with $\| W \| \leq \| T^\mrt \| _{col} + \eps = 1 + \eps$. Hence $\mr{spr} (T ^\mrt ) \leq 1 + \eps $ for any $\eps >0$ and the claim follows.
\end{proof}

Since $T (\zeta )$ is a row contraction for $\zeta \in \partial \D$, $T (\zeta ) ^\mrt = \ov{T} (\zeta ) ^*$ is a column contraction. Hence, for any $Z \in \B ^d _n$ and $\zeta \in \partial \D$, 
$$ I_n \otimes I - Z \otimes T (\zeta ) ^\mrt, $$ is invertible.  Since we are assuming that $\fb (0) =0$, a finite FM transfer--function formula for $\ov{\zeta} \cdot \fb$ is:
\be \ov{\zeta} \fb (Z) = I_n \otimes 1^* \left( I_n \otimes I - Z \otimes T (0) ^* \right) ^{-1} I_n \otimes T (\zeta ) ^* 1; \quad \quad Z \in \B ^d _n. \label{finFM} \ee 
By the previous lemma and \cite[Lemma 2.4]{SSS}, the closure of the joint similarity orbit of $T (\zeta ) ^\mrt$ contains a row contraction. 
Moreover, since the minimal de Branges--Rovnyak FM realization, $(A,B,C,D)$, of $\fb$ is such that $A = T(0) ^* | _{\scr{M} _0 (\fb )}$, Theorem \ref{ncratinH2} implies that $A$ and hence $T(0)$ is a pure and finite dimensional row contraction. Hence $T(0)$ it is jointly similar to a strict row contraction by Lemma \ref{rowvscol}, $T (0) ^*$ is jointly similar to a strict column contraction and
$$ I_n \otimes I - T (\zeta) ^\mrt \otimes T (0) ^*, $$ is invertible for any $\zeta \in \partial \D$. We conclude that $\fb (T (\zeta ) ^\mrt )$ is well--defined for any $\zeta \in \partial \D$. Alternatively, since $T(\zeta)$ is a row contraction, $\mr{spr} (T (\zeta) ) \leq 1$ and Lemma \ref{rowvscol} implies that the closure of the joint similarity orbit of $T(\zeta ) ^\mrt$ contains a row contraction. Since $\fb \in \mult$ is NC rational, Theorem \ref{ncratinH2} implies that $r \cdot \B ^d _\N \subseteq \nbdom \fb$ for some $r>1$ and it follows that $T(\zeta )^\mrt \in \nbdom \fb$. 
 
Recall the concept of vectorization of matrices and completely bounded maps on matrices. If $A \in \C^{m\times m}$ and $B \in \C ^{n\times n}$, then $A \otimes B$ is an $mn \times mn$ matrix, and it can also be identified with a completely bounded linear map on $\C ^{m \times n}$. To describe this correspondence, given  $Z \in \C ^{n\times m}$ let $\wvec{Z}$ denote the column vector of size $m \cdot n$ obtained by stacking the columns of $Z$ one on top of the other (in order from left to right). That is, dividing $Z \in \C ^{n \times m}$ into $m$ columns, $\mbf{z} _k \in \C ^n$ (see for example \cite[Section 4.2]{HornJohnson})
$$ Z = \bpm  \left. \mbf{z} _1 \right| \cdots \left| \mbf{z} _m \right. \epm \quad  \mapsto  \quad \wvec{Z} = \bpm \mbf{z} _1 \\ \vdots \\ \mbf{z} _m \epm \in \C ^{mn}. $$ By \cite[Lemma 4.3.1]{HornJohnson},
\[
\left( A \otimes B \right) \wvec{Z} = \wvec{B Z A^T}.
\]
This \emph{vectorization map} $\mr{vec} : \C ^{m \times n} \rightarrow \C ^{mn}$, $\mr{vec} (A) := \wvec{A}$, is linear and invertible, and for any linear map $\ell \in \scr{L} (\C ^{m \times n} )$, we define the \emph{matrization} of $\ell$, $\wvec{\ell} \in \C ^{mn \times mn}$ by
$$ \wvec{\ell} \wvec{Z} :=  \wvec{\ell (Z)}, \quad \quad \quad  \mbox{\emph{i.e.}} \quad \wvec{\ell} = \mr{vec} \circ \ell \circ \mr{vec} ^{-1}. $$ In particular, if $\ell$ is any completely bounded linear map on the operator space $\C ^{m \times n}$, 
\be \ell (X) = \sum _{j=1} ^d A_j X B_j; \quad \quad A_j \in \C ^{m\times m}, B_j \in \C ^{n \times n}, \ X \in \C ^{m \times n} \label{cb} \ee then
$$\wvec{\ell} = \sum B_j^\mrt \otimes A_j. $$

\begin{prop} \label{spectra}
Let $\fr$ be an NC rational function with $0 \in \nbdom \fr$ and finite Fornasini--Marchesini realization $(A,B,C,D)$. Then for any $\la \in \C \sm \{ 0 \}$ and $Z \in \C ^d _n$ so that $L_A (Z)$ is invertible,
$$ \la ^n \mr{det} \, L_{A ^{(\la )}} (Z) = \mr{det} \,  L_A (Z ) \cdot \mr{det} \left( (\la + \fr (0) ) I_n - \fr (Z) \right), $$
$$ A^{(\la)} _k = A_k + \la ^{-1} B_k C. $$
In particular, if the monic linear pencil $L_A (Z) = I_n \otimes I - Z \otimes A$ is invertible then $\la + \fr (0)$ belongs to the spectrum of $\fr (Z)$ if and only if $L_{A ^{(\la)}} (Z)$ is singular.
If $(A,B,C,D)$ is a minimal realization then this formula holds for all $Z \in \nbdom \fr$ and the characteristic polynomial of $\fr (Z)$ is 
$$ p_{\fr (Z)} ( \la + \fr (0) ) = \la ^n \frac{\mr{det} \, L_{A^{(\la)}} (Z) }{\mr{det} \, L_A (Z)}. $$ 
\end{prop}
\begin{proof}
If $(A,B,C,D)$ is a minimal realization then $Z \in \nbdom \fr$ if and only if $I - Z \otimes A$ is invertible by \cite[Theorem 3.10]{Volcic}. By the generalized matrix determinant lemma, 
\ba \mr{det} \left( I - Z \otimes A  - \la ^{-1} Z \otimes BC \right) & = & \mr{det} \left( I - Z \otimes A \right) \cdot \mr{det} \left( I_n - \la ^{-1} I_n \otimes C (I - Z \otimes A ) ^{-1} Z \otimes B \right) \nn \\
& = & \la ^{-n} \mr{det} \left( I_n \otimes I - Z \otimes A \right) \cdot \mr{det} \left( \la I_n - (\fr (Z) - \fr (0) ) I_n ) \right). \nn \ea 
\end{proof}

\begin{prop} \label{evalues}
Let $A_\zeta ^*$ be a column--isometric restriction of $T (\zeta ) ^*$ to an invariant subspace. Then $\zeta \in \partial \D$ is an eigenvalue of $\fb ( A _\zeta ^\mrt )$. 
\end{prop}

\begin{proof}
As discussed above, $\fb (A _\zeta ^\mrt )$ is well--defined. To simplify notation we drop the subscript $\zeta$. Identify $\scr{M} (\fb ) \simeq \C ^m$, and suppose that $\cK \subseteq \scr{M} (\fb )$ is $T (\zeta) ^*-$invariant and $A ^* := T (\zeta) ^* | _{\cK}$, $\cK \simeq \C ^k$, $k \leq m$. Then, 
$$ T(\zeta) ^* = \bpm A ^* & B ^* \\ & C ^* \epm. $$ Note that since $T (\zeta )$ is a row contraction and $A$ is a row co-isometry,
\ba  I_m & \geq & T (\zeta ) T (\zeta ) ^* = \bpm A & \\ B & C  \epm \bpm A ^* & B ^* \\ & C ^* \epm \nn \\
& = & \bpm A A ^* & AB^* \\ BA^* & BB^* + C C^* \epm = \bpm I & AB^* \\ BA^* & BB^* + CC^*\epm, \nn \ea so that 
$$ \bpm 0 & -AB^* \\ -BA^* & I - BB^* - CC ^* \epm \geq 0. $$
In the above, we view $A \in \C ^d _n$ as a row $d-$tuple, $A = (A _1 , \cdots , A_d ) : \C ^n \otimes \C ^d \rightarrow \C ^n$, and $A^*, B^*$ as column $d-$tuples so that, for example, 
$$ AB^* = A_1 B_1 ^* + \cdots + A_d B_d ^*. $$ 
It follows, by Schur complement theory, that $AB^* =0$ and that $BB^* + CC^* \leq I$. Observe that
$$ I_m \otimes I_m - \bsm A ^\mrt & 0 \\ 0 & 0 \esm  \otimes T (\zeta ) ^*  \simeq  I \otimes I - T (\zeta ) ^* \otimes \bsm A ^\mrt & 0 \\ 0 & 0 \esm, $$
and the second formula is the matrization of 
$$ \mr{id} _m - \bsm A  & 0 \\ 0 & 0 \esm (\cdot ) \bsm A ^* & B ^* \\ 0 & C  ^* \esm. $$ 
Then,
\ba & & \mr{id} _m \circ \bpm I_k & 0 \\ 0 & 0 \epm - \bpm  A  & 0 \\ 0 & 0 \epm \bpm I_k & 0 \\ 0  & 0 \epm \bpm A ^* & B ^* \\ 0 & C  ^* \epm \nn \\
& = & \bpm I_k & 0 \\ 0  & 0 \epm - \bpm AA ^* & A B^* \\ 0 & 0 \epm = \bpm I_k & 0 \\ 0 & 0 \epm - \bpm I_k & 0 \\  0 & 0 \epm \equiv 0, \nn \ea 
and it follows that $\wvec{\bsm I_k & 0 \\ 0  & 0 \esm} $ is an eigenvector of $T (\zeta ) ^* \otimes \bsm A^\mrt & 0 \\ 0 & 0 \esm$ to eigenvalue $1$, and the image of this eigenvector under the tensor swap unitary is then an eigenvector of $\bsm A^\mrt & 0\\ 0 & 0 \esm \otimes T (\zeta) ^*$ to eigenvalue $1$. Proposition \ref{spectra} then implies that $\zeta$ is an eigenvalue of 
$$ \fb \bpm A^\mrt & 0 \\ 0 & 0 \epm = \bpm \fb (A ^\mrt ) & 0 \\ 0  & 0 \epm, $$ so that $\zeta$ is an eigenvalue of $\fb (A^\mrt )$. Namely, for any $Z \in \mr{Dom} _n \, \fb$, $$  L_{T(\zeta) ^*} (Z)  =  I_n \otimes I_m  - Z \otimes A - \ov{\zeta} Z \otimes BC, $$ where 
$$ A = T(0) ^*, \quad B:= T(1) ^* 1, \quad C := \ip{1}{\cdot} \quad \mbox{and} \quad D:= \fb (0) =0, $$ gives the finite FM realization formula for $\fb$ from Equation (\ref{finFM}). 
\end{proof}

\begin{prop} \label{GMprop}
Let $\cK := \bigoplus _{k=1} ^N \cK _k \subseteq \scr{M} (\fb )$ be the direct sum of minimal and mutually orthogonal $T (\zeta )-$co-invariant subspaces of $\scr{M} (\fb )$ so that each $T(\zeta ) ^* | _{\cK _k}$ is a column isometry. Then $\zeta$ is an eigenvalue of $\fb (T _\zeta ^\mrt )$ of geometric multiplicity at least $N$.
\end{prop}
\begin{proof}
To simplify notations, assume without loss in generality that $\zeta =1$ and that $N=2$. We will also write $T:= T(1)$. Then,
$$T^* = \bpm[margin] \cellcolor{blue!15} \begin{array}{cc} A_1 ^* &   \\  & A_2 ^* \end{array} & \cellcolor{green!15} \mbox{\Large $B^*$} \\ \mbox{\Large $ \quad $ } & \cellcolor{red!15} \mbox{\Large $C^*$} \epm, $$ where $A_k ^* := T^* | _{\cK _k}$ are (irreducible) row co-isometries.  Hence, 
$$ \fb (T ^\mrt ) = \bpm[margin] \cellcolor{blue!15} \begin{array}{cc} \fb (A_1 ^\mrt ) &   \\  & \fb (A_2 ^\mrt ) \end{array} & \cellcolor{green!15} \mbox{\Large $*$} \\ \mbox{\Large $ 0 $ } & \cellcolor{red!15} \mbox{\Large $\fb (C^\mrt )$} \epm. $$
For any $x = \bsm x_1  \\ x_2 \\ x_3 \esm$, 
$$ \fb (T ^\mrt ) x = \bpm \fb (A_1 ^\mrt ) x_1 + * x_3 \\ \fb (A_2 ^\mrt ) x_2 + * x_3 \\ \fb (C ^\mrt ) x_3 \epm. $$ 
By the previous proposition, both $\fb (A _1 ^\mrt )$ and $\fb (A_2 ^\mrt )$ have eigenvectors, $y_1$ and $y_2$ to eigenvalue $1$. It follows that 
$$ \bsm y_1 \\ 0 \\ 0 \esm \quad \mbox{and} \quad  \bsm 0 \\ y _2 \\ 0 \esm$$  are two linearly independent eigenvectors of $\fb (T ^\mrt )$ to eigenvalue $1$.
\end{proof}

\begin{thm} \label{boundary}
Let $A (\zeta) ^* := T(\zeta )^* | _{\cK _\zeta}$ be a column--isometric restriction of $T(\zeta ) ^*$ to an invariant subspace $\cK _\zeta \subseteq \scr{M} (\fb )$. Suppose that $\mr{dim} \, \cK _\zeta =n$, identify $\cK _\zeta$ with $\C ^n$, and let $v$ be a unit eigenvector of $\fb (A (\zeta ) ^\mrt ) ^\mrt $ corresponding to the eigenvalue $\zeta$. Then for any $y \in \C ^n$, the limit $K^\fb \{ A (\zeta )  , y , v \} := \lim _{r \uparrow 1} K ^\fb \{ r A (\zeta )  , y , v \} $ exists, and for any $h \in \scr{H} ^\mrt (\fb)$, the limit 
$$  y^* h (A (\zeta)  ) v :=  \lim _{r\uparrow 1} y ^* h (r A (\zeta )  ) v = \ip{ K^\fb \{ A (\zeta ) , y , v \} }{h}_{\scr{H} ^\mrt (\fb ) } \quad \quad \mbox{exists.} $$ 
\end{thm}
\begin{proof}
To simplify notations we simply write $A$ in place of $A (\zeta )$. For any $0<r<1$, the NC de Branges--Rovnyak kernel vector $K ^\fb \{ r A (\zeta )  , y , v \}$ is well--defined since $r A$ is a strict row contraction. Observe that the net $K^\fb \{ r A  , y , v \}$ converges pointwise in $\B ^d _\N$. Indeed, given any $W \in \B ^d _m$ and $x,u \in \C ^m$,
\ba x^* K^\fb \{ r A  , y , v \} (W) u & = & x^* K ^\fb (W , r A  ) [ v u^*] y  \nn \\
& = & x^* K (W , r A  ) [ v u^*] y - x^* K (W, r A ) [ \fb ^\mrt (Z ) v u^* \fb ^\mrt (r A  ) ^*] y. \nn \ea Since $R \cdot \ov{\B ^d _\N } \subseteq \nbdom \fbt$ for some $R >1$ and $A$ is a row contraction, the limit
$$ \lim _{r \uparrow 1} \fb ^\mrt (r A  ) = \fbt ( A ), $$ exists. Similarly, since $W \in \B ^d _m$, the entire expression converges to
$$ x^* K (W ,  A  ) [ v u^*] y - x^* K (W ,  A  ) [ \fb ^\mrt ( Z ) v u^* \fb ^\mrt ( A  ) ^*] y. $$
Let $K^{\fb} \{  A  , y ,  v \} (W) $ denote this pointwise limit. Observe that by assumption,
$$ \fb ^\mrt ( A) v = \fb (A^\mrt ) ^\mrt v = \zeta v. $$ 
Now consider,
\ba \| K ^\fb \{ r A  , y , v \} \| _{\bH ^2} ^2 & = &  y^* K (rA , rA) \left[ vv^* - \fbt (r A ) vv^* \fbt (rA )^* \right] y \nn \\
& \leq & y^* K (rA , rA) [I_n ] y  \cdot  \| vv^* - \fbt (rA ) vv^* \fbt (rA) ^* \| \nn \\
& \leq & \tr \left( vv^* - \fbt (rA ) vv^* \fbt (rA) ^*  \right) \cdot y^* \sum _{j=0} ^\infty r^{2j} \underbrace{\mr{Ad} _{A , A^*} ^{(j)} (I_n )}_{=I_n} y \nn \\
& = &  \| y \| ^2 \, \frac{v^* \left( I_n - \fbt (rA ) ^* \fbt (rA)  \right) v}{1-r^2} \nn \\
& = & \| y \| ^2 \, \frac{v^* \left( \fbt (A) ^* \fbt(A) - \fbt (rA ) ^* \fbt (rA ) \right) v}{(1-r) (1+r)}. \label{boundeval}  \ea 
Since $R \cdot \B ^d _\N \subseteq \nbdom \fbt$ for some $R<1$, it follows that $\fbt (Z)$ and hence $\fbt (Z ) ^* \fbt (Z)$ is G\^ateaux differentiable at any point on the boundary, $\partial \B ^d _\N$. We conclude that the limit supremum of Equation (\ref{boundeval}) as $r \uparrow 1$ is finite. By weak compactness, there is a weakly convergent subsequence $K ^\fb \{ r_k  A , y , v \}$ which necessarily converges pointwise to $K ^\fb \{  A , y,  v \} (Z) $. Hence any weakly convergent subsequence has the same limit and the entire net converges weakly to $K ^\fb \{  A , y,  v \} \in \scr{H} ^\mrt (\fb )$. By weak convergence, given any $h \in \scr{H} ^\mrt (\fb )$, 
\ba  \ip{ K^\fb \{ A , y , v \} }{h}_{\scr{H} ^\mrt (\fb ) } & = & \lim _{r \uparrow 1}  \ip{ K^\fb \{ r A , y , v \} }{h}_{\scr{H} ^\mrt (\fb ) } \nn \\
&= & \lim y^* h (r A ) v. \nn \ea It follows that $h (r A )$ is convergent and we let $h( A )$ denote this limit.
\end{proof}

\begin{remark}
If $\mu = \mu _{\fb}$ is the singular NC rational Clark measure of an NC rational inner left multiplier, $\fb \in [ \mult ] _1$, then we can define the \emph{support of $\mu $}, $\mr{supp} (\mu )$, on the boundary, $\partial \B ^d _\N$, as the set of all finite--dimensional row co-isometries, $Z \in \partial \B ^d _\N$, so that the minimal row isometric dilation, $V$, of $Z$, is unitarily equivalent to a direct summand of $\Pi _\mu$. By Theorem \ref{DKSue}, $Z \in \mr{supp} (\mu )$ if and only if there is a $T_{\mu }-$co-invariant subspace $\cJ _\mu \subseteq \cH _{\mu  }$, so that $Z ^*$ is jointly unitarily equivalent to $T_{\mu } ^* | _{\cJ _\mu}$. Proposition \ref{evalues} then implies that $1$ is an eigenvalue of $\fb (Z ^\mrt )$ for any $Z \in \mr{supp} (\mu )$. 

Recall that in one variable, any rational inner function in $H^2$ is a finite Blaschke product:
$$ \fb (z) = \zeta \prod _{k=1} ^N \frac{z- w_k}{1-\ov{w} _k z}; \quad \quad w_k \in \D, \ \zeta \in \partial \D. $$ In this case 
$\mu _\fb$ is a finite, positively weighted sum of exactly $N$ point masses on the unit circle, so that $\mu _\fb$ is singular and $L^2 (\mu _b ) = H^2 (\mu _b )$. Hence $\Pi _\fb \simeq M_\zeta$ is unitary, and $\Pi _\fb = \Pi _{\mu _\fb }$ is irreducible if and only if $N=1$ and $\fb$ is a single Blaschke factor. Moreover, the point masses of $\mu _\fb$ are located precisely at the $N$ points on the unit circle where $\fb (\zeta ) =1$.  Proposition \ref{GMprop} can be viewed as an analogue of this classical fact. 

If $\fb \in H^\infty$ is rational, it extends analytically to a disk of radius $>1$, and so it has finite Carath\'eodory angular derivatives at any point $\zeta \in \partial \D$. That is, $|\fb (\zeta ) | =1$ for all $\zeta \in \partial \D$ and $\fb ' (\zeta )$ has a non-tangential limit at each point on the boundary. By \cite[VI-4]{Sarason-dB}, this is equivalent to saying that every $h \in \scr{H} (\fb )$ has a non-tangential limit at every point on the boundary. Theorem \ref{boundary} can then be viewed as a generalization of this classical result.
\end{remark}

\subsection{Mutual singularity of Clark--Cuntz perturbations}

Let $b \in [ H^\infty ] _1$ be a contractive analytic function in the complex unit disk. Given any $\zeta \in \partial \D$, let $\mu _\zeta := \mu _{b \ov{\zeta}}$ be the one--parameter family of positive Clark measures of the contractive functions $\ov{\zeta} b$. The goal of this subsection is to obtain an analogue of the following classical theorem of Aronszajn and Donoghue \cite{Aronszajn,Dono}, for the case of contractive NC rational multipliers of the Fock space.
\begin{thm*}[Aronszajn--Donoghue]
The singular parts of the family of Clark measures $\{ \mu _\zeta = \mu _{b \ov{\zeta}} | \ \zeta \in \partial \D \}$ are mutually singular,
$$ \mu _{\zeta ; s} \ \perp \ \mu _{\xi ; s} \quad \mbox{for} \quad \zeta , \xi \in \partial \D, \ \zeta \neq \xi. $$ 
\end{thm*}
One can show that two positive, finite and regular Borel measures on the complex unit circle are mutually singular if and only if their Herglotz spaces of Cauchy transforms have trivial intersection \cite[Section 1.1, Corollary 8.5]{JM-ncld}. The following is then an exact NC analogue of the Aronszajn--Donoghue theorem for arbitrary contractive left multipliers of Fock space:
\begin{thm}[NC Aronszajn--Donoghue]
Given $b \in [\mult ] _1$, consider the one--parameter family of NC Clark measures $\mu _\zeta := \mu _{\ov{\zeta} b}$, $\zeta \in \partial \D$. The singular parts of this family of NC measures are mutually singular in the sense that their spaces of NC Cauchy transforms have trivial intersection,
$$ \scr{H} ^+ (H _{\ov\zeta b ; s } ) \bigcap \scr{H} ^+ (H _{\ov\xi b ; s} ) = \{ 0 \}. $$ 
\end{thm}

\begin{proof}
Suppose that $h \in \scr{H} ^+ (H _{\ov\zeta b ; s } ) \bigcap \scr{H} ^+ (H _{\ov\xi b ; s} )$. Then right multiplication by $I - \ov{\zeta} b^\mrt $ and by $I - \ov\xi b^\mrt$ both take this intersection space into $\scr{H} ^\mrt (b)$. Hence $h - \ov\zeta h b^\mrt$ and $h - \ov\xi h b^\mrt$ both belong to $\scr{H} ^\mrt (b)$, and hence both $h$ and $h b^\mrt$ belong to $\scr{H} ^\mrt (b) \subseteq \hardy$. In particular $h \in \hardy \bigcap \scr{H} ^+ (H_{\ov\zeta b ; s} ) = \{ 0 \}$, since $\mu _{\ov\zeta b ; s}$ is a singular NC measure, by assumption \cite[Corollary 8.13]{JM-ncld}. 
\end{proof}

The classical Aronszajn--Donoghue theorem can also be restated in operator--theoretic language. If $\mu , \la$ are positive measures, both singular with respect to Lebesgue measure on the complex unit circle, consider the measure spaces $L^2 (\mu ) = H^2 (\mu )$ and $L^2 (\la ) = H^2 (\la)$, where as before, $H^2 (\mu )$ denotes the closure of the analytic polynomials. It follows that the isometries of multiplication by the independent variable on $H^2 (\mu )$ and $H^2 (\la)$, $U_\mu := M_\zeta | _{H^2 (\mu )} = M_\zeta$ and $U_\la$, are unitary. To say that the singular measures $\mu, \la$ are mutually singular is then equivalent to the statement that $U_\mu$ and $U_\la$ are mutually singular in the sense that they have no unitarily equivalent restrictions to reducing subspaces.  Similarly, we will say that two Cuntz row isometries, $U$ and $U'$ are \emph{mutually singular}, and we write $U \perp U'$, if they have no unitarily equivalent direct summands, \emph{i.e.} unitarily equivalent restrictions to reducing subspaces. The exact NC analogue of this formulation of the Aronszajn--Donoghue theorem would then state that if $b \in [\mult ]_1$, then the singular Cuntz GNS row isometries $\Pi (\zeta) _s = \Pi _{\mu _{\ov\zeta b ; s}}$ are mutually singular for $\zeta, \xi \in \partial \D$, $\zeta \neq \xi$. While the proof of this general statement eludes us at this time, we can prove the following weaker statement for contractive NC rational multipliers:

\begin{thm}[NC rational Aronszajn--Donoghue] \label{ncAD}
Let $\fb \in [\mult ]_1$ be an NC rational contractive left multiplier of Fock space. For any $\zeta \in \partial \D$, let $\mu _\zeta = \mu _{\ov\zeta \cdot \fb}$ be the finitely--correlated NC Clark measure of $\ov\zeta \cdot \fb$ with GNS representation $\Pi (\zeta )$. 

If $\fb$ is inner so that $\Pi (\zeta) = \Pi (\zeta ) _s$, if $\Pi (\zeta ) _s$ is irreducible for some $\zeta \in \partial \D$ and if $\fb | _{\B ^d _1}$ does not vanish identically, then $\Pi (\zeta ) _s \perp \Pi (\xi ) _s$ for any $\xi \neq \zeta$, $\xi \in \partial \D$. 

If $\fb \in [\mult ] _1$, $\nbdim \cH _\mu = n \in \N$, $\mu = \mu _1$ and $\{ \zeta _1 , \cdots , \zeta _{n+1} \} \subseteq \partial \D$ is any set of $n+1$ distinct points on the circle, then there is a $j \in \{ 1, \cdots , n+1 \}$, so that $\Pi (\zeta _j ) _s \perp \Pi (\zeta _k ) _s$ for all $k \neq j$. 
\end{thm}

\subsection{Mutual singularity and disjointness}

Let $U, U'$ be Cuntz row isometries acting on Hilbert spaces $\cK , \cK '$ respectively. These row isometries are said to be \emph{disjoint} if  there is no bounded operator $X : \cK \rightarrow \cK '$ so that $X U_k = U' _k X$ \cite{DHJorg-atom}. Such an $X$ is called an \emph{intertwiner}. The following is well known and can be found in \cite[Lemma 8.9]{JM-ncld}:

\begin{lemma}
Let $U, U'$ be row isometries on $\cK , \cK '$, respectively, and suppose that $X : \cK \rightarrow \cK '$ is a bounded intertwiner, $X U ^\om = U ^{'\om} X$. If $U$ is a Cuntz unitary then also
$$ X^* U ^{'\om} = U ^\om X^*, $$ so that $D := X^* X$ belongs to the commutant of the von Neumann algebra, $\mr{vN} (U)$, generated by $U$, and $D' = X X ^*$ belongs to the commutant of $\mr{vN} ( U' )$. 
\end{lemma}
In particular, if $U$ is a Cuntz row isometry then the commutant of $U$, \emph{i.e.} the set of all operators $A$ that commute with each $U_k$, $1 \leq k \leq d$, is a von Neumann algebra.
\begin{prop}
Let $U, U'$ be Cuntz row isometries on $\cK , \cK'$, and let $X : \cK \rightarrow \cK '$ be an intertwiner. Then $U' | _{P_{\nbran X}}$ and $U | _{P _{\nbran X^*}}$ are unitarily equivalent sub-representations.
\end{prop}
That is, two Cuntz row isometries are disjoint if and only if they are mutually singular. In the above statement, we are identifying any row isometry, $U$, with a $*-$representation, $\pi _U$ of the Cuntz--Toeplitz $C^*-$algebra, $\mc{E} _d := C^* \{ I , L_1 , \cdots , L_d \}$. To say that $\pi _U$ and $\pi _{U'}$ have no unitarily equivalent sub-representations is equivalent to the statement that $U, U'$ have no unitarily equivalent restrictions to reducing subspaces.
\begin{proof}
By the previous lemma $XX^* \in \mr{vN} (U') '$, the commutant of $\mr{vN} (U')$, so that $P _X := P _{\ov{\nbran X}} \in \mr{vN} (U') '$ is $U'-$reducing, and similarly $P_{X^*}$ is $U-$reducing. Hence $U ' | _{\nbran P_X }$ and $U | _{\nbran P_{X^*}}$ are sub-representations of the Cuntz algebra. Define $W : \nbran P_{X^*} \rightarrow \nbran P_{X}$ by the formula $W \sqrt{X^*X} h = Xh \in \nbran X \subseteq \cK '$. Then $W$ is defined on a dense subset of $\nbran P_{X^*} = \ov{\nbran X^* } \subseteq \cK$, it has dense range in $\nbran P_X$, and
\ba \| W \sqrt{X^*X} h \| ^2 & = &  \ip{Xh}{Xh}_{\cK '} \nn \\
&= & \ip{h}{X^* X h} _{\cK} = \| \sqrt{X^*X} h \| ^2. \nn \ea  The linear map $W : \ov{\nbran X^* } \rightarrow \ov{\nbran X}$ extends by continuity to an onto isometry. Finally,
\ba W U_k \sqrt{X^*X} h & = & W \sqrt{X^* X} U_k h \nn \\
& = & X U_k h = U_k ' Xh \nn \\
& = & U_k ' W \sqrt{X^* X} h. \nn \ea This proves that $W U_k P_{X^*} = U_k' P_X W$ so that these Cuntz sub-representations are unitarily equivalent.
\end{proof}

\subsection{Proof of the NC Aronszajn--Donoghue theorem}

Recall that we are assuming that $\fb (0) =0$ so that $K_0 ^\fb =1$.
\begin{lemma}
If $\fb (0) =0$, then for any $\zeta \in \partial \D$, 
$$ X (\zeta ) _j 1 = \ov{\zeta} L_j ^* \fbt \quad \quad  \mbox{and} \quad \quad
 X(\zeta) = X +  X(\zeta) 1 \ip{1}{\cdot}_\fb. $$ 
\end{lemma}
\begin{proof}
We performed this calculation for the case $\zeta =1$ in Section \ref{minreal}. Since we are assuming that $\fb (0) =0$, for any $\zeta \in \partial \D$,
$$ X(\zeta ) _j \underbrace{K_0 ^\fb}_{=1} =  0  + \ov{\zeta} L_j ^* \fbt. $$ 
The formula for $X (\zeta )$ becomes
\ba X(\zeta) & = & X + \ov{\zeta}  L^* \fbt \ip{1}{\cdot}_\fb \nn \\
& = & X +  X(\zeta) 1 \ip{1}{\cdot}_\fb. \nn \ea
\end{proof}

Recall that $X(\zeta )_k  = \scr{F} _{\zeta } \Pi (\zeta ) _{k} \scr{F} _{\zeta } ^{-1}$ is the image of a component of the GNS row isometry $\Pi (\zeta ) = \Pi _{\mu _{\ov{\zeta} \cdot \fb}}$ under the unitary weighted free Cauchy transform, $\scr{F} _\zeta : \hardy ( \mu _\zeta ) \rightarrow \scr{H} ^\mrt (\fb )$. Further recall that
$$ T(\zeta) ^* := X(\zeta) | _{\scr{M} (\fb )}, $$ is defined as the restriction of $X(\zeta)$ to the finite--dimensional space $\scr{M} (\fb ) \subseteq \scr{H} ^\mrt (\fb )$, for any $\zeta \in \C$. Each $T(\zeta ) ^*$ is a rank--one perturbation of $T(0) ^* = L^* | _{\scr{M} (\fb )}$, 
\ba T (\zeta) ^* & = & T (0) ^* + \ov{\zeta}  L^* \fbt \ip{1}{\cdot}_{\scr{H} ^\mrt (\fb )} | _{\scr{M} (\fb )} \nn \\
& = & T (0) ^* +  T (\zeta ) ^* 1 \ip{1}{\cdot}. \nn \ea A finite Fornasini--Marchesini realization of $\ov{\zeta} \cdot \fb$, for any $\zeta \in \partial \D$, is then given as the transfer function of the colligation,
\be \bpm A & B \\ C & D \epm = \bpm T (0) ^* &   T (\zeta) ^* 1 \\ 1 ^* &  0 \epm. \label{notquitemin} \ee
Note that the minimal FM realization is obtained by compressing to $\scr{M} _0 (\fb ) := \bigvee _{\om \neq \emptyset} L^{*\om} \fbt \subseteq \scr{M} (\fb )$, so that the above realization is not necessarily minimal, although it is close to it in the sense that the size of this realization is at most one greater than that of the minimal realization.  By construction, $X(\zeta)^* \simeq \Pi (\zeta )$ is the minimal row-isometric dilation of $T (\zeta )$ for every $\zeta \in \partial \D$.

By Theorem \ref{DKSue}, for any $\zeta , \xi \in \partial \D$, $\Pi (\zeta ) _{s}$ and $\Pi (\xi ) _{s}$ will have unitarily equivalent direct summands if and only if there are minimal $T(\zeta)$ and $T (\xi)$ co-invariant subspaces, $\cK _\zeta$ and $\cK _\xi$ of $\scr{M} (\fb )$, so that 
$$ F_\zeta ^* := T (\zeta ) ^* | _{\cK _\zeta} \quad \mbox{and} \quad F_\xi ^* := T (\xi ) ^* | _{\cK _\xi}, $$ are unitarily equivalent and irreducible row co-isometries. That is,
\be T (\zeta ) ^* = \bpm F (\zeta)  ^* & * \\ & G (\zeta ) ^* \epm, \label{blockdecomp} \ee has some block upper triangular decomposition with respect to $K_\zeta$ and $K_\zeta ^\perp$,  $T (\xi ) ^*$ has a similar decomposition with respect to $K_\xi$, and $F(\zeta )$ is unitarily equivalent to $F(\xi)$.  Without loss of generality, we will assume for the remainder of this section that $\xi =1$ and $\zeta \neq 1$. Let $P, P _\zeta$ be the projections onto $\cK = \cK _1 , \cK _\zeta$. 

\begin{lemma}
Assume that $\fb \in [\mult ] _1$ and that $\fb (0) =0$. Then for any $Z \in \ov{\B ^d _\N}$, $I - Z \otimes T(0) ^*$ is invertible and
$$  \mr{det} \left( I_n \otimes I - Z \otimes T (\zeta) ^* \right) = \mr{det} \left( I - Z \otimes T (0) ^* \right) \cdot \mr{det} \left(  I_n - \ov{\zeta} \fb (Z) \right). $$
\end{lemma}
\begin{proof}
This follows from Proposition \ref{spectra}, the formula 
$$ T(\zeta ) ^* = T(0) ^* + \ov{\zeta} T(1) ^* 1 \ip{1}{\cdot}, $$ the fact that $(T(0) ^* , T(1) ^* 1 , 1 ^* , 0 )$ is a finite (but not necessarily minimal) FM realization of $\fb$ and the fact that $T(0)$ is a pure and finite-dimensional row contraction so that $I - Z \otimes T(0) ^*$ is invertible for any row contraction, $Z$, by Lemma \ref{rowvscol}. Since $\fb \in [ \mult ] _1$, the minimal de Branges--Rovnyak FM realization, $(A,B,C,D)$, of $\fb$ is such that $A = T(0) ^* | _{\scr{M} _0 (\fb )}$ and it follows that $A$ is pure and similar to a strict row contraction by Theorem \ref{ncratinH2} and \cite[Theorem 3.8]{PopRota} or \cite[Proposition 2.3, Remark 2.6]{SSS}. It is not difficult to show that since $\nbran T(0) ^* \subseteq \scr{M} _0 (\fb )$, that $T(0) ^*$ is then itself also pure, hence similar to a strict row contraction. Lemma \ref{rowvscol} then implies that $T(0) ^*$ is jointly similar to a strict column contraction so that $I - Z \otimes T(0) ^*$ is invertible for every $Z \in \ov{\B ^d _\N}$.
\end{proof}

\begin{proof}{ (of Theorem \ref{ncAD})}
First assume that $\fb$ is inner, that $\fb$ does not vanish identically on the first level of the row-ball, that $\Pi (1)$ is irreducible and that $\Pi (1)$ is unitarily equivalent to $\Pi (\zeta )$. By the previous lemma, we obtain that 
$$  \mr{det} (I -  \fb (Z) ) =  \mr{det} (I - \ov{\zeta} \fb (Z) ), $$ for all $Z \in \B ^d _\N$. In particular, choosing $Z =z  \in \B ^d _1$ gives
$$ 1 - \fb (z) = 1 - \ov{\zeta} \fb (z), $$ which implies that $\zeta =1$, a contradiction.

To prove the second part of the theorem statement, assume that $\fb \in [ \mult ] _1$ is an arbitrary and contractive NC rational left multiplier of Fock space, and that the finitely--correlated NC measure $\mu = \mu _\fb$ is such that $\nbdim \cH _\mu = n$. Assume that $\{ \zeta _1 , \cdots , \zeta _{n+1} \}$ are $n+1$ distinct points on the circle and that there is no $1 \leq j \leq n+1$ so that $\Pi (\zeta _j ) \perp \Pi (\zeta _k )$ for all $k \neq j$. Equivalently given any fixed $1 \leq j \leq n+1$ and every $1 \leq k \leq n+1$, $j \neq k$, $\Pi (\zeta _j )$ and $\Pi (\zeta _k )$ have unitarily equivalent singular Cuntz direct summands. 
Hence, for every $1 \leq k \leq n+1$, as described above,
$$ T (\zeta _k ) ^* = \bpm F (\zeta _k ) ^* & * \\ 0 & G(\zeta _k ) ^* \epm,$$
where each $F(\zeta _k)$ is an irreducible row co-isometry, and the $F(\zeta _k)$ are jointly unitarily equivalent for each $1 \leq k \leq n+1$. Since $\fb$ is an NC function, it follows that the matrices $\fb \left( F(\zeta _k ) ^\mrt \right) \in \C ^{n\times n}$ are unitarily equivalent for $1\leq k \leq n+1$, so that each $\zeta _k$ is an eigenvalue of $\fb \left( F(\zeta _k ) ^\mrt \right)$, and hence also of
$$ \fb \left( T (\zeta _k ) ^\mrt \right) = \bpm \fb ( F(\zeta _k ) ^\mrt ) & * \\ & \fb ( G(\zeta _k ) ^\mrt ) \epm, $$ by Proposition \ref{evalues}. This is impossible as $\fb \left( T(\zeta _k ) ^\mrt \right)$ is isomorphic to an $n\times n$ matrix and has at most $n$ distinct eigenvalues. 
\end{proof}

\begin{remark}
Our NC Aronszajn--Donoghue theorem, Theorem \ref{ncAD} shows that `most of' the singular parts of the GNS row isometries, $\Pi (\zeta )$, associated to a contractive NC rational $\fb \in [\mult ] _1$, are mutually singular or disjoint. Although we suspect that it may generally be that $\Pi (\zeta ) _s \perp \Pi ( \xi ) _s$ for any $\zeta \neq \xi$ there are several obstacles to extending our above argument. First, if $\fb$ is inner, vanishes identically on $\B ^d _1$ and $\Pi (\zeta ) = \Pi (\zeta ) _s$ is irreducible, then if $\Pi (\zeta )$ and $\Pi (\xi )$ are not disjoint, then $\Pi (\zeta )$ is unitarily equivalent to the restriction of $\Pi (\xi )$ to a reducing subspace. By Theorem \ref{DKSue}, this happens if and only if $T (\zeta)$ and $T(\xi)$ are unitarily equivalent, or equivalently if and only if $\Pi (\zeta )$ and $\Pi (\xi)$ are unitarily equivalent.
In this case we obtain, as in the above proof, that
$$  \mr{det} (I -  \fb (Z) ) =  \mr{det} (I - \ov{\zeta} \fb (Z) ), $$ for all $Z \in \B ^d _\N$. Hence if 
$$ \mf{f} _\zeta (Z) := (I - \ov{\zeta} \fb (Z) ) (I - \fb (Z) ) ^{-1}, $$ then $\mf{f} _\zeta \in \scr{O} (\B ^d _\N )$ is NC rational and $\mr{det} \, \mf{f} _\zeta (Z) \equiv 1$. However, such NC rational functions exist, one example is:
$$ \mf{f} (x,y) = (1 -xy) (1-yx) ^{-1}, $$ \cite{Jurij}. If we set $\zeta = -1$, $\mf{f} := \mf{f} _{-1}$,
$$ \mf{f} =: (1 - \fr ) ( 1 + \fr ) ^{-1} $$ and solve for $\fr$, we obtain that 
$$ \fr (x,y) = (xy -yx) ( 2 -xy -yx ) ^{-1}. $$ By re-scaling the variables, $x \mapsto r \cdot x$, $y \mapsto r \cdot y$, for some sufficiently small $0<r<1$, we then obtain a contractive NC rational function $\fb \in [\mult ] _1$, $\fb (Z) := \fr (r Z)$, so that 
$$ \wt{\mf{f}} (Z) = (I + \fb (Z) ) (I - \fb (Z) ) ^{-1}, $$ is an NC rational Herglotz function with constant determinant $1$ on its domain.  Hence to prove the NC rational Aronszajn--Donoghue theorem in the case where $\fb$ is inner and $\Pi (\zeta )$ is an irreducible Cuntz row isometry, one would need to argue that these assumptions on $\fb$ imply that the function $\mf{f} _\zeta (Z)$ cannot have constant determinant. The reducible case seems even more difficult: If $\Pi (\zeta) _s$ and $\Pi (\xi ) _s$ are reducible and not mutually singular, then we obtain that 
$$ \mr{det} \,  L_{G (\zeta ) ^*} (Z) \cdot \mr{det} \left( I - \ov{\xi} \fb (Z) \right) = \mr{det} \, L_{G(\xi ) ^*} (Z) \cdot \mr{det} \left( I - \ov{\zeta} \fb (Z) \right). $$ Appendix \ref{constdet} provides a characterization of NC functions with constant determinant.

Another class of examples of NC functions with constant determinant can be constructed as follows: If $f, g \in \mult$ are any two outer or singular inner left multipliers of Fock space, then $f,g$ are pointwise invertible in the NC unit row-ball, $\B ^d _\N$ \cite{JMS-ncBSO}, and 
$h := f g f^{-1} g^{-1} \in \scr{O} (\B ^d _\N )$ will have constant determinant equal to $1$.
\end{remark}

The following two examples illustrate phenomena in the behaviour of the components of $T (\zeta ) ^*$ as a function of $\zeta$. 

\begin{eg}
In this example each $T(\zeta )$ is a row co-isometry, $T (1) ^*$ is reducible, however, for $\zeta \neq 1$, the $T (\zeta) ^*$ are all irreducible. All of the $T (\zeta ) ^*$ are pairwise non-similar. Set
\[
T (1) ^*_{1} = \begin{pmatrix} 1 & 0 \\ 0 & 0 \end{pmatrix}, \quad T (1) ^*_{2} = \begin{pmatrix} 0 & 0 \\ 0 & 1 \end{pmatrix}.
\]
It is immediate to check that $x = \frac{1}{\sqrt{2}} \bsm 1 \\ 1 \esm$ is a cyclic vector for $T^* (1)$. Hence,
\[
T (\zeta ) ^* _{1} = T (1) ^*_{1} (1 - (\zeta - 1)xx^*) = \begin{pmatrix} \frac{\zeta + 1}{2} & \frac{\zeta - 1}{2} \\ 0 & 0 \end{pmatrix}, \quad \quad  T (\zeta ) ^*_{2} = \begin{pmatrix} 0 & 0 \\ \frac{\zeta - 1}{2} & \frac{\zeta+1}{2} \end{pmatrix}.
\]
It is well known that a pair of $2\times 2$ matrices is reducible if and only if the determinant of their commutator is $0$. Therefore, we compute,
\[
\det [T (\zeta ) ^*_{1}, T (\zeta ) ^*_{2}] = \frac{1}{16} \det \begin{pmatrix} (\zeta - 1)^2 & \zeta^2 - 1 \\ 1 - \zeta^2 & - (\zeta - 1)^2 \end{pmatrix} = \frac{1}{4} \zeta (\zeta - 1)^2.
\]
In particular, this polynomial does not vanish for any $\zeta \neq 1$ on the unit circle. Thus, for every $1 \neq \zeta \in \partial \D$, the point $T(\zeta ) ^*$ is irreducible. Moreover, since $\tr T(\zeta ) ^*_{1} = \frac{\zeta + 1}{2}$, we conclude that these matrices are pairwise non-similar.
By \cite[Theorem 6.8]{DKS-finrow}, since $X (\zeta ) ^* \simeq \Pi (\zeta )$ is the minimal row isometric dilation of $T (\zeta )$, where $\Pi (\zeta )$ is the GNS row isometry of $\mu _{\ov{\zeta} \fb }$, each $\Pi (\zeta )$ is a Cuntz row isometry of dilation type, $\Pi (\zeta )$ is irreducible for $\zeta \neq 1$, $\Pi (1 )$ is reducible and $\Pi (\zeta ) \,  \perp \, \Pi (\xi )$ are mutually singular Cuntz row isometries. 
\end{eg}

Recall that the reducible tuples of matrices form an algebraic subvariety of $\C _n^d$. Let $p_1, \ldots, p_k$ be the polynomials in the co-ordinates of $\C _n^d$ that cut out the subvariety of reducible matrices. The map $\zeta \mapsto T (\zeta ) ^*$ is affine in $\zeta$, hence we obtain a family of polynomials in $\zeta$: $q_1(\zeta) = p_1(T (\zeta) ^* ),\ldots, q_k(\zeta) = p_k(T (\zeta ) ^*)$. Since the points where $T (\zeta ) ^*$ is reducible are precisely the common zeroes of $q_1,\ldots,q_k$, there are either at most finitely many of them, or $q_1 = \cdots = q_k = 0$. The following example shows that the second case can occur. 

\begin{eg}
The matrices considered in this example are $4 \times 4$. We will denote by $e_1,e_2,e_3,e_4$ the standard basis for $\C^4$. Consider the row co-isometry
\[
T (1) ^*_{1} = \begin{pmatrix} 0 & 0 & 0 & 0 \\ 1 & 0 & 0 & 0 \\ 0 & 0 & 0 & 0 \\  0 & 0 & 1 & 0 \end{pmatrix} \quad \mbox{and} \quad T(1) ^*_{2} = \begin{pmatrix}  0 & 0 & 0 & 0 \\  0 & 1 & 0 & 0 \\ 0 & 0 & 0 & 1 \\  0 & 0 & 0 & 0 \end{pmatrix}.
\]
We take $x = \frac{1}{2} (e_1 + e_2+e_3 +e_4)$. It is straightforward to check that $x$ is cyclic for both $T(1) ^*$ and $T (1 )$. For example, $T (1) _{1} x = e_1 + e_3$, $T (1) _{2} x = e_2 + e_4$, $T (1) _{2} T (1) _{1} x = e_4$, $T (1) _{2}^2 x = e_2$, and $T (1) _{1} T (1) _{2} T (1) _{1} x = e_3$. Now we set $\omega = \frac{\zeta - 1}{4}$ and calculate,
\[
T(\zeta ) ^*_{1} = \begin{pmatrix} 0 & 0 & 0 & 0 \\ \omega + 1 & \omega & \omega & \omega \\ 0 & 0 & 0 & 0 \\ \omega & \omega & \omega + 1 & \omega \end{pmatrix}, \quad T (\zeta ) ^*_{2}  = \begin{pmatrix} 0 & 0 & 0 & 0 \\ \omega & \omega + 1 & \omega & \omega \\ \omega & \omega & \omega & \omega + 1 \\ 0 & 0 & 0 & 0 \end{pmatrix}.
\]
Note that the subspace $V$ spanned by $\{e_2,e_3,e_4\}$ is always $T (\zeta ) ^*-$invariant. However, it is easy to see that there are two minimal $T(1)^*-$invariant subspaces, the one spanned by $e_2$ and the one spanned by $\{e_3,e_4\}$. Let $C^*_{\zeta} = T(\zeta) ^*|_V$. Since $C^*_{\zeta}$ is $3\times 3$, if it is reducible, then $\det[C^*_{\zeta;1}, C^*_{\zeta;2}] = 0$, since this pair will have either an invariant or a coinvariant one-dimensional subspace. However,
\[
\det[C^*_{\zeta;1}, C^*_{\zeta;2}] = -2 \omega^2( 2 \omega + 1 ).
\]
Hence $C^*_{\zeta}$ is reducible if and only if $\zeta = 1$ or $\zeta = -1$. We understand the former case. In the latter case, there is a minimal $C^*_{-1}$ (and $T^*_{-1}$) invariant subspace spanned by $e_2 + e_4$. This vector is $T (-1)-$cyclic and thus by \cite[Corollary 5.5]{DKS-finrow}, we have that this minimal subspace is unique. Note that $\tr T(\zeta ) ^*_{1} = 2 \omega$ and thus these pairs are pairwise non-similar. Moreover, the Cuntz isometries $X(\zeta)$ are irreducible for all $\zeta \in \partial \D \sm \{ 1  \}$. However, the free semigroup algebra of $X(-1)^*$ is different from those of the $X(\zeta) ^*$ for $\zeta \neq \pm 1$ by \cite[Theorem 5.15]{DKS-finrow}.
\end{eg}

\subsection{Additional NC rational Aronszajn--Donoghue results}

\begin{prop}
Suppose that $T(\zeta)$ is an irreducible row co-isometry. If there is a $Z \in \C ^d _\N$ so that $Z \otimes T(\zeta) ^*$ is not singular, then $\Pi (\zeta ) _s \perp \Pi (\xi ) _s$ for any $\zeta \neq \xi$, $\zeta, \xi \in \partial \D$.
\end{prop}
\begin{proof}
This follows easily from the fact that 
$$ T (\zeta ) ^* = T (\xi ) ^* V_{\zeta, \xi}, \quad \mbox{where} \quad V_{\zeta ,\xi} = I - \xi \cdot \ov{\zeta} \, 1 \ip{1}{\cdot}. $$ 
The matrix $V_{\zeta , \xi}$ is unitary with determinant $\mr{det} \, V_{\zeta ,\xi} = \xi \cdot \ov{\zeta} \neq 1$. Assuming that $T(\zeta), T(\xi )$ are unitarily equivalent gives the contradiction
\ba \mr{det} \, Z \otimes T (\xi ) ^* &  = & \mr{det} \, Z\otimes T(\zeta ) ^*  \nn \\
& = & \mr{det} \, Z \otimes T(\xi ) ^* \cdot \mr{det} \, I \otimes V_{\zeta, \xi } \nn \\
& = & \xi \ov\zeta \cdot \mr{det} \, Z \otimes T(\xi ) ^*. \nn \ea
\end{proof}

\begin{eg}
There are irreducible column isometries/ row co-isometries that violate the condition of the preceding proposition. For example, denote by $E_{ij} \in \C ^{3\times 3}$ the matrix units. Then the tuple $T^* = (E_{12}, \frac{1}{\sqrt{2}} E_{21}, E_{13}, \frac{1}{\sqrt{2}} E_{31})$ is an irreducible column isometry, such that for every $Z$, $Z \otimes T^*$ is singular. However, this tuple does not contradict the general NC Aronszajn-Donoghue conjecture that $T (\xi ) ^*$ and $T(\zeta ) ^*$ have no unitarily equivalent restrictions to invariant subspaces for $\xi \neq \zeta$, $\xi , \zeta \in \partial \D$.
\end{eg}

Let $A = (A_1,\cdots,A_d) \in \C ^d _n$ be an irreducible tuple. Let $x \in \C^n$ and consider the functions $A_j (z) \colon \C \to \C ^{n\times n}$ given by $A_j (z) = A_j(I + z \, x x^*)$ so that $A = A(0)$ and set $A(z) = \left( A_1 (z),\cdots, A_d (z) \right)$. Let 
\[
S_A = \left\{ \left.  z \in \C \setminus \{0\} \right| \ A(z) \text{ is similar to } A\right\}.
\]

\begin{lemma} \label{lem:procesi}
Let $\om$ be a word in $\{1,\cdots,d\}$ with $|\om | = m$. Let $p_\om (z) = \frac{1}{z}\left( \tr A(z)^\om - \tr A^\om \right)$. If $S_A \neq \emptyset$, then $\deg p_\om \leq \lfloor \frac{m}{2} \rfloor - 1$.
\end{lemma}
\begin{proof}
Let $z_0 \in S_A$. Since $A(z_0)$ is similar to $A$, we have that $\tr A_j(z_0) = \tr A_j$. But $\tr A_j(z_0) = \tr A_j + z_0 \langle A_j x, x\rangle$. Hence, for all $j=1,\ldots,d$, $\langle A_j x, x \rangle = 0$. Now consider 
\[
\tr A(z)^\om = \tr (A_{\om_1} + z A_{\om_1} x x^*)\cdots (A_{\om_m} + z A_{\om_m} x x^*).
\]
For $\ell \geq \lfloor \frac{m}{2} \rfloor + 1$, we note that the coefficient of $z^{\ell}$ will a sum of traces of products of matrices. Each product will contain a pair of adjacent elements of the form $A_{\om_j} x x^* A_{\om_{j+1}} x x^* $, where if $j+1 > m$, then we reduce it modulo $m$. Hence, this product is going to be $0$. Therefore, all of the coefficients of $\tr A(z)^\om$ of $z^{\ell}$ for $\ell \geq \lfloor \frac{m}{2} \rfloor + 1$ are $0$. The claim follows from the definition of $p_\om$.
\end{proof}

\begin{cor} \label{cor1}
In the setting of the previous lemma, if $|S_A| \geq n^2/2$, then $A(z)$ is similar to $A(0) = A$ for all $z \in \C$.
\end{cor}
\begin{proof}
Since $A$ is irreducible, the similarity orbit of $A$ is closed \cite{Artin}. Moreover, the ring of similarity invariant functions on $\C ^d _n$ is generated by traces of monomials \cite{Procesi}. Hence, $A(z)$ is similar to $A = A(0)$ if and only if for all words $\om$, $\tr A(z)^\om = \tr A^\om $. By a result of Razmislov \cite{Razmyslov}, taking words with $|\om | \leq n^2$ is enough to generate the algebra of invariants. Since $S_A$ is contained in the zeroes of $p_\om$ for all $\om$ and $|S_A| \geq \deg p_\om$ for all $|\om | \leq n^2$, we obtain that they are identically $0$. Therefore all traces are identically $0$ and we have that $A(z)$ are all similar to $A$.
\end{proof}

\begin{remark}
Kuzmin (see \cite{DrenskyFormanek} and the references therein) has provided a lower bound of $\frac{n(n+1)}{2}$ on the length of words needed to generate the invariant algebra. He has conjectured that the lower bound is always sufficient. Dubnov and Lee have verified the conjecture for $n \leq 4$.
\end{remark}

\begin{cor} \label{cor2}
In the setting of Lemma \ref{lem:procesi}, if $n = 2$ and $S_A \neq \emptyset$,  then $f(z)$ is similar to $A$ for all $z \in \C$.
\end{cor}
\begin{proof}
By a result of Dubnov \cite{DrenskyFormanek}, we need traces of words of length at most $3$ to generate the algebra of invariant functions. For all words of length at most $3$, $\deg p_\om \leq 0$. However, $p_\om$ all vanish on $S_A$ and thus these polynomials are identically $0$. 
\end{proof}

Setting $A(z) _j := T(\ov{z} ) ^* _j$, where $T(z) ^*$ is the finite--dimensional Clark perturbation defined in Equation (\ref{bzerofinClark}), corresponding to an NC rational inner, $\fb$, Corollary \ref{cor1} and Corollary \ref{cor2} yield additional Aronszajn--Donoghue-type results:

\begin{cor}
Suppose that $\nbdim \scr{M} (\fb ) =n$ so that $\mr{row} \, T(\ov{z} ) ^*  \simeq A(z) \in \C ^d _n$, where the $T(\ov{z} )$ are the finite--dimensional Clark perturbations corresponding to an NC rational inner. Then there are at most $n^2 /2 -1$ points $\zeta _k \in \partial \D$ so that the row co-isometric $T(\zeta _k )$ are mutually and jointly similar. If $n=2$ then $T(\zeta)$ cannot be jointly similar to $T (\xi )$ for any $\zeta \neq \xi$, $\zeta , \xi \in \partial \D$.
\end{cor}

\begin{proof}
If either of these statements hold, Corollary \ref{cor1} or Corollary \ref{cor2} imply that every $T(z) ^*$ is jointly similar, for every $z \in \C$. In particular $T(\zeta) ^*$ is similar to $T(0) ^*$ for every $\zeta \in \partial \D$. This is impossible as $T(0)$ is a pure row $d-$tuple and each $T(\zeta )$ is a row co-isometry for $\zeta \in \partial \D$ (since we assume that they are the Clark peturbations corresponding to an NC rational inner). Hence, either $|S_A | < n^2 /2$, or, if $n=2$, then $S_A = \emptyset$. 
\end{proof}

\appendix

\section{$SL(\N )-$valued NC functions} \label{constdet}

Suppose that $f$ is a free NC function so that $\mr{det} \, f(Z)$ is constant on its domain $\nbdom f \subseteq \C ^d _\N$. By taking direct sums, it follows that if $f$ is not identically zero, then $\mr{det} \, f(Z) \equiv 1$. 

\begin{thm}
Let $f$ be a free NC function with uniformly open and connected domain $\nbdom f \subseteq \C ^d _\N$. Further assume that $0 \in \nbdom f$ and that $f^{-1} (Z)$ is defined in a uniformly open neighbourhood of $0 \in \B ^d _1$. If  
$$ f(Z) = \sum _{j=0} ^\infty f_j (Z) \quad \mbox{and} \quad  f(Z) ^{-1} = \sum _{k=0} g_k (Z), $$ are the Taylor--Taylor series expansions of $f, f^{-1}$ at $0 \in \B ^d _1$ so that $f_j, g_j \in \fp$ are homogeneous free polynomials of degree $j$, then $\mr{det} \, f(Z) \equiv 1$ on $\nbdom f$ if and only if 
$$ 0 = \sum _{\substack{j+k = \ell \\ j \in \N, \ k \in \N \cup \{ 0 \}}} j \, \tr f_j (Z) g_k (Z), $$ for every $\ell \in \N$. 
\end{thm}
\begin{proof}
Choose $r>0$ so that $r \ov{\B ^d _\N } \subseteq \nbdom f \bigcap \nbdom f^{-1}$. Fix $Z \in r \ov{\B ^d _\N }$ and define the analytic function on $ \ov{\D}$ by 
$$ h(\la ) := \mr{det} \, f (\la Z ) =1. $$ Taking the derivative and applying Jacobi's formula yields
\ba 0 & = & h' (\la ) = \underbrace{\mr{det} \, f(\la (Z))}_{\equiv 1} \cdot \tr \partial _Z f (\la Z) f(\la Z) ^{-1} \nn \\
& = & \sum _{j=1, \ k=0} ^\infty \la ^{j+k -1} \, j \, \tr f_j (Z) g_k (Z) \nn \\
& = & \sum _{\ell =1} ^\infty \la ^{\ell-1} \sum _{\substack{j+k = \ell \\ j \in \N, \ k \in \N \cup \{ 0 \} }} j \tr f_j (Z) g_k (Z). \nn \ea
Multiplying both sides of this expression by $\ov{\la} ^n$ and integrating with respect to normalized Lebesgue measure over the complex unit circle 
yields 
\ba 0 & = & \sum _\ell \int _{\partial \D} e^{i (\ell -n -1) \theta } d\theta \sum _{\substack{j+k = \ell \\ j \in \N, \ k \in \N \cup \{ 0 \} }} j \tr f_j (Z) g_k (Z) \nn \\
& = & \sum _{\substack{j+k = n +1 \\ j \in \N, \ k \in \N \cup \{ 0 \} }} j \tr f_j (Z) g_k (Z), \nn \ea for any $n \in \N \cup \{ 0 \}$. 

Conversely, if the above condition holds, then it follows that for any fixed $Z \in r \ov{\B ^d _n}$, the function $h(\la) := \mr{det} \, f (\la Z)$ has vanishing derivative. Hence $h(\la)$ is constant so that $h(\la) =1$ since $f$ is NC. In particular $$1 = h(1) = \mr{det} \, f(Z) = h(0) = \mr{det} f(0), $$ and this holds for every $Z \in r\B ^d _\N$, and hence for every $Z \in \nbdom f$ since $\nbdom f$ is connected. 
\end{proof}

\bibliography{Bibs/ncBla}

\Addresses

\pagebreak

\end{document}